\documentclass[11pt]{amsart}

\usepackage[utf8]{inputenc}
\usepackage[T1]{fontenc}
\usepackage{lmodern}

\usepackage{amsmath,amssymb,amsthm,mathtools}
\usepackage{mathrsfs}
\usepackage{enumerate}
\usepackage{enumitem}
\usepackage{hyperref}
\usepackage{tikz}
\usetikzlibrary{arrows.meta,calc,decorations.pathmorphing,positioning}

\usepackage{graphicx}
\usepackage{booktabs}
\usepackage{microtype}
\usepackage[nameinlink]{cleveref}

\usepackage[a4paper,margin=1in]{geometry}

\numberwithin{equation}{section}

\newtheorem{theorem}{Theorem}[section]
\newtheorem{proposition}[theorem]{Proposition}
\newtheorem{lemma}[theorem]{Lemma}
\newtheorem{corollary}[theorem]{Corollary}

\theoremstyle{definition}
\newtheorem{definition}[theorem]{Definition}
\newtheorem{example}[theorem]{Example}

\theoremstyle{remark}
\newtheorem{remark}[theorem]{Remark}

\newcommand{\R}{\mathbb{R}}
\newcommand{\Rp}{\mathbb{R}_{>0}}

\newcommand{\Rc}{\mathbb{R}_{\mathrm{cl}}} 
\newcommand{\Hnum}{H}                      
\newcommand{\Sset}{S}                      
\newcommand{\Snon}{S^{\times}}             
\newcommand{\sgn}{\operatorname{sgn}}
\newcommand{\AssocCurv}{\boldsymbol{\kappa}} 

\newcommand{\Kamb}{K}                      
\newcommand{\op}{\boxplus}                 
\newcommand{\sop}{\oplus_S}                
\newcommand{\Eemb}{E}                      
\newcommand{\iemb}{i}                      
\newcommand{\Cfun}{C}                      
\newcommand{\idH}{1_H}

\title[Three-sign cancellation hypernumbers]{Three-sign cancellation hypernumber systems and associator curvature}

\author{Jaehwan Kim}
\address{Hankuk Academy of Foreign Studies, Yongin-si, Gyeonggi-do 17035, Republic of Korea}
\email{020080@hafs.hs.kr}

\date{\today}

\begin{document}

\begin{abstract}
We introduce and study a three--sign cancellation hypernumber system
$H$ which extends the real field by adjoining a third sign $\Lambda$.
The underlying set is $H=\{0\}\cup\{+,-,\Lambda\}\times\mathbb{R}_{>0}$,
with a single–valued multiplication and a hyperaddition $\boxplus$
designed to encode cancellation phenomena between positive and negative
reals. The classical real line embeds as a genuine subfield
$\mathbb{R}_{\mathrm{cl}}\subset H$, and all field operations agree
with the usual ones on $\mathbb{R}$.

The additive structure of $H$ is almost associative but not a canonical
hypergroup. We give an explicit description of where associativity
fails and compute, for triples of the form $(+,a),(-,b),(\Lambda,c)$,
a closed formula for the associativity defect
\[
  \AssocCurv(a,b,c)=2\min(a,b)=a+b-|a-b|,
\]
which coincides with the loss of absolute value when adding $a$ and
$-b$ in $\mathbb{R}$. To explain this behaviour, we construct an
ambient ``cancellation monoid'' $(K,\oplus)$ on $\mathbb{R}\times\mathbb{R}_{\ge0}$
which is strictly associative and records both real sums and accumulated
cancellation mass. We prove that $H$ cannot be recovered from $K$ by
any simple projection, and formulate an ambient reconstruction problem.

In addition, scalar multiplication by real numbers (defined via the embedded copy of $\mathbb{R}$) distributes over $\op$, and the sign--layer admits a canonical hypergroup envelope governing the possible signs of hyper\-sums. The results provide a controlled example of a non\-associative
hyperaddition sitting over the real field and suggest several directions
for multi\-sign generalizations and connections with hyperfields and
tropical geometry.
\end{abstract}

\maketitle

\tableofcontents


\section{Introduction}
\label{sec:intro}

The arithmetic of the real field merges \emph{sign} and \emph{magnitude} in a way
that is inherently insensitive to cancellation: whenever $x>0$ and $y<0$ with
$|x|=|y|$, the sum $x+y$ vanishes and no record of the opposing contributions
remains.  Hyperstructures (hypergroups, hyperrings, hyperfields) offer a natural
language for retaining such ambiguity by allowing the ``sum'' of two elements to
be \emph{set--valued}, a viewpoint that goes back to Marty and Krasner and has
been developed in modern treatments and applications
\cite{Marty1934Hypergroup,Golzio2018Survey,Krasner1983Hyperrings,Vougiouklis1994Hyperstructures,CorsiniLeoreanu2003Applications,DavvazLeoreanuFotea2007HyperringTheory,ConnesConsani2010Monoids,Jun2018AGHyperrings,Viro2010Hyperfields}.
In particular, hyperfields have become standard tools in tropical geometry and
its signed variants, and in matroid theory over hyperfields
\cite{BakerBowler2019MatroidsHyperfields,Lorscheid2022TropicalHyperfield,MaxwellSmith2024Extensions}.

\medskip
\noindent\textbf{Three--sign cancellation hypernumbers.}
We introduce a simple three--sign system that explicitly records the appearance
of cancellation.  The underlying set is
\[
H=\{0\}\cup\{+,-,\Lambda\}\times\Rp,
\]
where each nonzero element carries a sign $\sigma\in\{+,-,\Lambda\}$ and a size
$a\in\Rp$.  The additional symbol $\Lambda$ is meant to encode the
\emph{cancellation layer}.  We define a commutative unital hyperaddition
$\boxplus$ on $H$ (Definition~\ref{def:H-hyperadd}) that is designed so that, in
mixed sign interactions, the output retains information about the amount of
cancellation that would be lost in~$\R$.

\medskip
\begin{theorem}[Associator curvature for $(+,a),(-,b),(\Lambda,c)$]
\label{thm:intro-defect}
Let $a,b,c>0$ and set
\[
  x=(+,a),\quad y=(-,b),\quad z=(\Lambda,c)\in\Hnum.
\]
Then both bracketings are singletons in the $\Lambda$--sector:
\[
  (x\op y)\op z = \{(\Lambda,\,c+|a-b|)\},\qquad
  x\op (y\op z) = \{(\Lambda,\,a+b+c)\}.
\]
In particular, the associator curvature (associativity defect)
\[
  \AssocCurv(a,b,c):=(a+b+c)-(c+|a-b|)=a+b-|a-b|=2\min(a,b)
\]
is independent of $c$.
\end{theorem}

This invariant functions as a quantitative proxy for ``how much'' associativity
is obstructed by cancellation.

\begin{figure}[t]
\centering
\begin{tikzpicture}[>=Stealth, node distance=22mm]
\node (x) {$x$};
\node[right=of x] (y) {$y$};
\node[right=of y] (z) {$z$};

\node[below=12mm of y] (xy) {$(x\boxplus y)$};
\node[below=12mm of z] (yz) {$(y\boxplus z)$};

\node[below=12mm of xy] (l) {$(x\boxplus y)\boxplus z$};
\node[below=12mm of yz] (r) {$x\boxplus(y\boxplus z)$};

\draw[->] (x) -- (xy);
\draw[->] (y) -- (xy);
\draw[->] (y) -- (yz);
\draw[->] (z) -- (yz);
\draw[->] (xy) -- (l);
\draw[->] (z) -- (l);
\draw[->] (x) -- (r);
\draw[->] (yz) -- (r);

\draw[dashed, <->] (l) -- node[midway, right] {$\AssocCurv$} (r);
\end{tikzpicture}
\caption{Two bracketings of a triple hyper--sum.  Their discrepancy is measured
by the associator curvature $\AssocCurv$.}
\label{fig:assoc-diagram}
\end{figure}
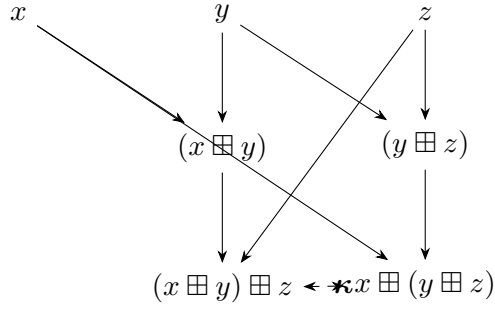

\medskip
\noindent\textbf{Ambient cancellation monoid and non--reconstruction.}
To explain the structure behind $\AssocCurv$, we construct an \emph{associative}
ambient cancellation monoid $(K,\oplus)$ encoding simultaneously a real part and
a cumulative cancellation mass, defined via the basic functional
\[
\Cfun(r_1,r_2)=|r_1|+|r_2|-|r_1+r_2|.
\]
There is a natural embedding $\Eemb:H\to K$ which is compatible with sign data.
However, we prove that $H$ cannot be recovered from $K$ by any
bracket--independent decoding procedure (Theorem~\ref{thm:nogo-ambient}); this is
a precise obstruction to ``associativizing'' the system without losing the
three--sign information.

\medskip
\noindent\textbf{Sign layer.}
On the sign set $\Sset=\{0,+,-,\Lambda\}$ we construct a canonical hypergroup
operation $\sop$ (Definition~\ref{def:S-hyperoperation},
Theorem~\ref{thm:S-canonical-hypergroup}) and show that the sign image of a
hyper--sum is described by an envelope property
(Section~\ref{sec:sign-envelope}).

\medskip
\noindent\textbf{Organization.}
Section~\ref{sec:prelim} recalls background on hyperoperations and hyperfields.
Section~\ref{sec:3CHS} defines the three--sign system $H$ and its basic
operations.  Section~\ref{sec:assoc} develops associator curvature and controlled
non--associativity.  Section~\ref{sec:sign-envelope} treats the sign hypergroup
and sign envelopes, while Section~\ref{sec:ambient} develops the ambient
cancellation monoid and proves the non--reconstruction theorem.
We conclude in Section~\ref{sec:conclusion} with perspectives and open
directions.
\section{Preliminaries and hyperstructure background}
\label{sec:prelim}

\subsection{Basic notation and real numbers}

We recall here the standard notation for real numbers and fix a few
elementary facts that will be used repeatedly throughout the paper.

\begin{definition}[Real line and subsets]\label{def:basic-R}
We write $\R$ for the field of real numbers, and we use the following
standard subsets:
\[
  \Rp := \{x\in\R : x>0\},\qquad
  \R_{\ge0} := \{x\in\R : x\ge0\}.
\]
We denote by
\[
  |\cdot|:\R\longrightarrow\R_{\ge0}
\]
the usual absolute value on $\R$.
\end{definition}

\begin{definition}[Elementary sign notation]\label{def:basic-signs}
We use the symbols
\[
  +,\quad -,\quad \Lambda
\]
to denote three distinguished ``signs''.  It is convenient to package
them into the sets
\[
  \Sset := \{0,+,-,\Lambda\},\qquad
  \Snon := \{+,-,\Lambda\}.
\]
Here $0$ is merely a symbol indicating the absence of sign; it should
not be confused with the real number $0\in\R$, although in many
contexts the identification is harmless.
\end{definition}

Later we will attach elements of $\Snon$ as sign labels to positive
real magnitudes in order to form the three--sign hypernumber system
$\Hnum$.  For the classical real line, we will use the usual sign
function on $\R$.

\begin{definition}[Sign of a real number]\label{def:basic-real-sign}
Define $\operatorname{sgn}_\R:\R\to\{+,-,0\}\subset\Sset$ by
\[
  \operatorname{sgn}_\R(x) :=
  \begin{cases}
    +, & x>0,\\
    0, & x=0,\\
    -, & x<0.
  \end{cases}
\]
\end{definition}

We record for later use the following elementary properties of the
absolute value.  They are classical but play a central role in the
construction of the ambient cancellation monoid that will appear in
Section~\ref{sec:ambient}.

\begin{lemma}[Basic properties of the absolute value]\label{lem:basic-abs}
For all $x,y\in\R$ one has:
\begin{enumerate}
  \item $|x|\ge0$, and $|x|=0$ if and only if $x=0$.
  \item $|xy| = |x|\cdot|y|$.
  \item $|x+y|\le|x|+|y|$ (triangle inequality).
\end{enumerate}
\end{lemma}

\begin{proof}
These are the standard properties of the absolute value on $\R$.  For
completeness we recall the usual arguments:

(1) By definition $|x|=\sqrt{x^2}\ge0$ for all $x\in\R$, with equality
if and only if $x^2=0$, i.e.\ $x=0$.

(2) For all $x,y\in\R$,
\[
  |xy| = \sqrt{(xy)^2} = \sqrt{x^2y^2}
       = \sqrt{x^2}\,\sqrt{y^2}
       = |x|\cdot|y|.
\]

(3) The triangle inequality $|x+y|\le|x|+|y|$ can be proved, for
example, by squaring both sides and expanding:
\[
  |x+y|^2 = x^2+2xy+y^2
          \le x^2+2|x||y|+y^2
          = (|x|+|y|)^2,
\]
so $|x+y|\le|x|+|y|$ since both sides are nonnegative.
\end{proof}

The triangle inequality allows us to define a nonnegative ``cancellation
mass'' associated to a pair of real numbers.  This will later become
the second coordinate in the ambient monoid $\Kamb$.

\begin{definition}[Cancellation mass on $\R$]\label{def:basic-C}
For $r_1,r_2\in\R$ we define
\[
  \Cfun(r_1,r_2)
  := |r_1| + |r_2| - |r_1+r_2|\in\R_{\ge0}.
\]
\end{definition}

\begin{lemma}[Nonnegativity and symmetry of $\Cfun$]\label{lem:basic-C}
For all $r_1,r_2\in\R$ one has:
\begin{enumerate}
  \item $\Cfun(r_1,r_2)\ge0$.
  \item $\Cfun(r_1,r_2)=\Cfun(r_2,r_1)$.
\end{enumerate}
\end{lemma}

\begin{proof}
(1) By the triangle inequality,
\[
  |r_1+r_2|\le|r_1|+|r_2|,
\]
so
\[
  \Cfun(r_1,r_2)=|r_1|+|r_2|-|r_1+r_2|\ge0.
\]

(2) Symmetry follows from the symmetry of absolute value and addition:
\[
  \Cfun(r_1,r_2)
  = |r_1|+|r_2|-|r_1+r_2|
  = |r_2|+|r_1|-|r_2+r_1|
  = \Cfun(r_2,r_1).
\]
\end{proof}

\begin{remark}[Cancellation mass as lost absolute value]\label{rem:basic-C}
If $r_1,r_2$ have the same sign, then $|r_1+r_2|=|r_1|+|r_2|$ and
$\Cfun(r_1,r_2)=0$.  If $r_1$ and $r_2$ have opposite signs, then
$\Cfun(r_1,r_2)$ measures exactly how much absolute value is lost when
adding them.  For example, when $r_1=a>0$ and $r_2=-b<0$ with $a,b>0$,
a short computation shows
\[
  \Cfun(a,-b) = a+b - |a-b|.
\]
We will later see that this quantity coincides with the associativity
defect in certain triples of hypernumbers.
\end{remark}

The notation and basic facts collected in this subsection will be used
throughout the paper.  In the next subsection we recall, in a similarly
concise fashion, the notions of hyperoperation, canonical commutative
hypergroup, and hyperfield, which provide the broader context for our
three--sign construction.


\subsection{Notation at a glance}\label{subsec:notation}

For the reader's convenience we collect the principal objects and symbols that
recur throughout the paper.

\begin{table}[ht]
\centering
\caption{Core notation used in the paper.}
\label{tab:notation}
\begin{tabular}{@{}ll@{}}
\toprule
Symbol & Meaning \\ \midrule
$\Rp$ & positive reals $\{x\in\R:x>0\}$ \\
$H$ & three--sign cancellation hypernumbers (Section~\ref{sec:3CHS}) \\
$\boxplus$ & hyperaddition on $H$ (Definition~\ref{def:H-hyperadd}) \\
$\Sset$ & sign set $\{0,+,-,\Lambda\}$ (Section~\ref{sec:sign-envelope}) \\
$\sop$ & sign hyperoperation on $\Sset$ (Definition~\ref{def:S-hyperoperation}) \\
$\AssocCurv$ & associator curvature measuring non--associativity (Section~\ref{subsec:assoc-defect-quantitative}) \\
$\Cfun$ & cancellation mass on $\R$ (Definition~\ref{def:cancellation-mass}) \\
$K$ & ambient cancellation monoid (Section~\ref{sec:ambient}) \\
$\oplus$ & associative ambient operation on $K$ (Definition~\ref{def:ambient-monoid}) \\
$\Eemb$ & embedding of $H$ into $K$ (Definition~\ref{def:ambient-representation}) \\
\bottomrule
\end{tabular}
\end{table}

\subsection{Hyperoperations and canonical hypergroups}
\label{subsec:hypergroups}

\noindent
\emph{Literature.}  Hypergroups were introduced by Marty and have since developed
into a broad hyperstructure theory; for historical context and modern
expositions see \cite{Golzio2018Survey,Massouros2023Borderline}.

We recall the basic language of hyperoperations and canonical
commutative hypergroups in the minimal form needed later.  Throughout
this subsection $X$ denotes a nonempty set.

\begin{definition}[Hyperoperation]\label{def:hyperoperation}
A \emph{hyperoperation} on $X$ is a map
\[
  \boxplus : X\times X \longrightarrow
             \mathcal{P}(X)\setminus\{\emptyset\},
\]
assigning to each pair $(x,y)\in X\times X$ a nonempty subset
$x\boxplus y\subset X$.  The pair $(X,\boxplus)$ is then called a
\emph{hypermagma}.
\end{definition}

It is customary to think of $x\boxplus y$ as the ``set of possible
sums'' of $x$ and $y$.  Ordinary binary operations are recovered as
special cases in which all these sets are singletons.

\begin{definition}[Commutative hypermagma]\label{def:comm-hypermagma}
A hypermagma $(X,\boxplus)$ is called \emph{commutative} if
\[
  x\boxplus y = y\boxplus x
  \qquad\text{for all }x,y\in X.
\]
\end{definition}

Associativity in the hyper setting is formulated in terms of unions of
sets.  Given $x,y,z\in X$, one considers the two subsets
\[
  (x\boxplus y)\boxplus z
  := \bigcup_{u\in x\boxplus y} u\boxplus z, \qquad
  x\boxplus(y\boxplus z)
  := \bigcup_{v\in y\boxplus z} x\boxplus v
\]
of $X$.

\begin{definition}[Associativity at a triple]\label{def:assoc-hyper}
A hypermagma $(X,\boxplus)$ is said to be \emph{associative at the
triple} $(x,y,z)\in X^3$ if
\[
  (x\boxplus y)\boxplus z
  \;=\; x\boxplus(y\boxplus z)
\]
as subsets of $X$.  It is called \emph{associative} if this holds for
all triples $(x,y,z)\in X^3$.
\end{definition}

In analogy with abelian groups, one often requires the existence of a
distinguished element playing the role of additive identity, together
with additive inverses.

\begin{definition}[Neutral element and inverses]\label{def:neutral-inverse}
Let $(X,\boxplus)$ be a hypermagma.
\begin{enumerate}
  \item An element $0\in X$ is called a \emph{neutral element} (or
        \emph{additive identity}) if
  \[
    0\boxplus x = x\boxplus 0 = \{x\}
    \qquad\text{for all }x\in X.
  \]
  \item Assume $0\in X$ is a neutral element.  An element $y\in X$ is
        called an \emph{additive inverse} of $x\in X$ if
  \[
    0 \in x\boxplus y.
  \]
  When such a $y$ exists and is unique, it will be denoted by $-x$.
\end{enumerate}
\end{definition}

Canonical commutative hypergroups arise when one imposes natural
axioms paralleling those of abelian groups.  We use the following
formulation, which is standard in the literature on hyperrings.

\begin{definition}[Canonical commutative hypergroup]\label{def:canonical-hypergroup}
A \emph{canonical commutative hypergroup} is a pair $(X,\boxplus)$
consisting of a nonempty set $X$ and a hyperoperation $\boxplus$ such
that:
\begin{enumerate}
  \item[\textnormal{(HG1)}] \textbf{Commutativity:}
  \[
    x\boxplus y = y\boxplus x
    \qquad\text{for all }x,y\in X.
  \]

  \item[\textnormal{(HG2)}] \textbf{Associativity:} for all
        $x,y,z\in X$ one has
  \[
    (x\boxplus y)\boxplus z
    = x\boxplus(y\boxplus z).
  \]

  \item[\textnormal{(HG3)}] \textbf{Neutral element:} there exists a
        distinguished element $0\in X$ such that
  \[
    0\boxplus x = x\boxplus 0 = \{x\}
    \qquad\text{for all }x\in X.
  \]

  \item[\textnormal{(HG4)}] \textbf{Unique inverses:} for each
        $x\in X$ there exists a unique element $y\in X$ such that
  \[
    0 \in x\boxplus y,
  \]
  and we denote this element by $-x$.

  \item[\textnormal{(HG5)}] \textbf{Reversibility:} for all
        $x,y,z\in X$ one has
  \[
    x\in y\boxplus z
    \quad\Longleftrightarrow\quad
    z \in x\boxplus(-y).
  \]
\end{enumerate}
\end{definition}

\begin{remark}\label{rem:canonical-hypergroup}
The reversibility axiom (HG5) is the characteristic feature of
canonical hypergroups and is crucial in the definition of hyperrings:
it ensures that additive translation behaves in a controlled way.  In a
canonical commutative hypergroup, the neutral element $0$ and the
inverse $-x$ are uniquely determined, and the inverse map $x\mapsto-x$
is an involution.  We will use canonical hypergroups mainly as a
reference class: the additive structure $(\Hnum,\op)$ of our
three--sign system is \emph{not} a canonical hypergroup, and part of
our analysis in later sections consists of pinpointing exactly which
of (HG2) and (HG5) fail.
\end{remark}

The basic example to keep in mind is that of an ordinary abelian group
viewed as a hypergroup with singletons as sums.

\begin{theorem}[Abelian groups as canonical hypergroups]\label{thm:group-hypergroup}
Let $(G,+)$ be an abelian group with neutral element $0$ and inverse
map $x\mapsto-x$.  Define a hyperoperation $\boxplus$ on $G$ by
\[
  x\boxplus y := \{x+y\}
  \qquad\text{for all }x,y\in G.
\]
Then $(G,\boxplus)$ is a canonical commutative hypergroup.
\end{theorem}

\begin{proof}
We check the axioms (HG1)--(HG5).

(HG1) For all $x,y\in G$ one has $x+y=y+x$ by commutativity of the
group, hence $x\boxplus y=\{x+y\}=\{y+x\}=y\boxplus x$.

(HG2) For all $x,y,z\in G$,
\[
  (x\boxplus y)\boxplus z
  = \bigcup_{u\in x\boxplus y} u\boxplus z
  = \{(x+y)+z\},
\]
and similarly
\[
  x\boxplus(y\boxplus z)
  = \bigcup_{v\in y\boxplus z} x\boxplus v
  = \{x+(y+z)\}.
\]
Since $(x+y)+z=x+(y+z)$ in $G$, the two sets coincide.  Thus
associativity holds.

(HG3) The group neutral element $0\in G$ satisfies
\[
  0\boxplus x = \{0+x\} = \{x\},
  \qquad
  x\boxplus 0 = \{x+0\} = \{x\},
\]
so $0$ is a neutral element in the sense of
Definition~\ref{def:neutral-inverse}.

(HG4) For each $x\in G$, the group inverse $-x$ satisfies
\[
  x\boxplus(-x) = \{x+(-x)\} = \{0\},
\]
so $0\in x\boxplus(-x)$.  If $y\in G$ is any other element with
$0\in x\boxplus y$, then $0=x+y$, hence $y=-x$ by uniqueness of
inverses in the group.  Thus each element has a unique additive
inverse in the hypergroup sense, and it coincides with its group
inverse.

(HG5) Finally, let $x,y,z\in G$.  We have
\[
  x\in y\boxplus z
  \quad\Longleftrightarrow\quad
  x\in\{y+z\}
  \quad\Longleftrightarrow\quad
  x=y+z,
\]
and
\[
  z\in x\boxplus(-y)
  \quad\Longleftrightarrow\quad
  z\in\{x+(-y)\}
  \quad\Longleftrightarrow\quad
  z=x-y.
\]
The two conditions are equivalent because $x=y+z$ if and only if
$z=x-y$ in the abelian group.  Hence reversibility holds.  All axioms
are satisfied, so $(G,\boxplus)$ is a canonical commutative hypergroup.
\end{proof}

\begin{example}[Sign hyperfield as a canonical hypergroup]\label{ex:sign-hyperfield}
Let $X=\{0,+,-\}$ and define a hyperoperation $\boxplus$ by the rules
\[
  0\boxplus x = x\boxplus 0 = \{x\},\quad
  +\boxplus + = \{+\},\quad
  -\boxplus - = \{-\},
\]
and
\[
  +\boxplus - = -\boxplus +
  = \{+,0,-\}.
\]
Then $(X,\boxplus)$ is the additive hypergroup underlying the standard
sign hyperfield.  One checks directly that it satisfies
(HG1)--(HG5), so it is a canonical commutative hypergroup.  In
Section~\ref{sec:sign-envelope} we will extend this example by adding
the extra sign $\Lambda$ and constructing a canonical hypergroup
structure $(\Sset,\sop)$ on $\Sset=\{0,+,-,\Lambda\}$ which serves as a
sign--layer envelope for the three--sign hypernumber system.
\end{example}

The notion of canonical commutative hypergroup is central in the
theory of hyperrings and hyperfields: a hyperring is a set equipped
with a canonical commutative hypergroup structure for addition
together with a compatible single--valued multiplication.  In this
paper, we will use these concepts mainly as a reference point.  The
additive structure $(\Hnum,\op)$ of our three--sign system is
commutative and admits unique inverses but fails to satisfy (HG2) and
(HG5) in general; precise counterexamples will be given in
Section~\ref{sec:assoc}.  On the other hand, we will construct a
canonical commutative hypergroup on the sign set $\Sset$ in
Section~\ref{sec:sign-envelope}, illustrating how canonical
hypergroups can still appear at the level of sign data even when the
full hypernumber system is non--associative.

\subsection{Hyperfields and the sign hyperfield}
\label{subsec:hyperfields}

\noindent
\emph{Literature.}  Hyperrings and hyperfields in the sense of Krasner and their
modern incarnations play a central role in tropical geometry, matroid theory,
and ``absolute'' approaches to arithmetic; see
\cite{Viro2010Hyperfields,ConnesConsani2010Monoids,Jun2021GeometryHyperfields,BakerBowler2019MatroidsHyperfields,Lorscheid2022TropicalHyperfield}
for representative entry points.

We briefly recall the notion of a hyperfield, which generalizes the
usual concept of a field by allowing addition to be multivalued.  In
this framework, the additive structure is a canonical commutative
hypergroup in the sense of Definition~\ref{def:canonical-hypergroup},
while multiplication remains an ordinary single–valued operation.

\begin{definition}[Hyperring and hyperfield]\label{def:hyperfield}
A \emph{hyperring} is a triple $(F,\oplus,\cdot)$ consisting of a
nonempty set $F$, a hyperoperation
\[
  \oplus : F\times F \longrightarrow
           \mathcal{P}(F)\setminus\{\emptyset\},
\]
and a single–valued multiplication
\[
  \cdot : F\times F \longrightarrow F,
\]
such that:
\begin{enumerate}
  \item The pair $(F,\oplus)$ is a canonical commutative hypergroup
        with neutral element $0\in F$.
  \item The set $F^\times := F\setminus\{0\}$ is a (not necessarily
        commutative) group under $\cdot$ with identity element $1$.
  \item Multiplication distributes over hyperaddition on both sides:
        for all $a,b,c\in F$ one has
  \[
    a\cdot(b\oplus c)
    \subseteq (a\cdot b)\oplus(a\cdot c),
    \qquad
    (b\oplus c)\cdot a
    \subseteq (b\cdot a)\oplus(c\cdot a),
  \]
  where $a\cdot(B):=\{a\cdot x : x\in B\}$ and
  $(B)\cdot a := \{x\cdot a : x\in B\}$ for subsets $B\subseteq F$.
  \item The element $0$ is absorbing for multiplication:
  \[
    0\cdot a = a\cdot 0 = 0
    \qquad\text{for all }a\in F.
  \]
\end{enumerate}
A \emph{hyperfield} is a hyperring $(F,\oplus,\cdot)$ whose
multiplicative group $F^\times$ is abelian.
\end{definition}

Thus hyperfields are to fields what canonical commutative hypergroups
are to abelian groups: the multiplicative structure is group–like,
while addition is allowed to be multivalued but must satisfy a strong
set of axioms.

\begin{example}[Ordinary fields as hyperfields]\label{ex:field-as-hyperfield}
Let $(K,+,\cdot)$ be a field in the usual sense.  Define
\[
  a\oplus b := \{a+b\}
  \qquad\text{for all }a,b\in K.
\]
Then $(K,\oplus,\cdot)$ is a hyperfield.
\end{example}

\begin{proof}
By Theorem~\ref{thm:group-hypergroup}, the additive group $(K,+)$
gives rise to a canonical commutative hypergroup $(K,\oplus)$ by
taking singletons as hyper–sums, so condition~(1) of
Definition~\ref{def:hyperfield} holds.  The multiplicative group
$K^\times$ is an abelian group under $\cdot$, with the usual identity
$1$, so (2) holds as well.  Distributivity (3) is immediate from the
field axioms: for any $a,b,c\in K$,
\[
  a\cdot(b\oplus c) = a\cdot\{b+c\} = \{a(b+c)\}
                   = \{ab+ac\}
                   = (a\cdot b)\oplus(a\cdot c),
\]
and similarly on the right.  Finally, $0$ is absorbing for
multiplication in $K$, so (4) holds.  Hence $(K,\oplus,\cdot)$ is a
hyperfield.
\end{proof}

The most important example for our purposes is the \emph{sign
hyperfield}, whose underlying set is $\{0,+,-\}$ and whose additive
structure encodes the possible signs of sums of real numbers.  We view
this here as a special case of Example~\ref{ex:sign-hyperfield}.

\begin{definition}[Sign hyperfield]\label{def:sign-hyperfield}
Let $S_{\mathrm{sign}}:=\{0,+,-\}$.  Define a hyperoperation
$\oplus_{\mathrm{sign}}$ on $S_{\mathrm{sign}}$ by
\begin{align*}
  0\oplus_{\mathrm{sign}} x
  &= x\oplus_{\mathrm{sign}} 0
   = \{x\}
   &&\text{for all }x\in S_{\mathrm{sign}},\\
  +\oplus_{\mathrm{sign}} +
  &= \{+\},\\
  -\oplus_{\mathrm{sign}} -
  &= \{-\},\\
  +\oplus_{\mathrm{sign}} -
  &= -\oplus_{\mathrm{sign}} +
   = \{+,0,-\}.
\end{align*}
Define multiplication on $S_{\mathrm{sign}}$ by
\[
  0\cdot x = x\cdot 0 := 0\quad\text{for all }x,\qquad
  +\cdot + := +,\quad
  +\cdot - := -,\quad
  -\cdot + := -,\quad
  -\cdot - := +.
\]
\end{definition}

\begin{proposition}\label{prop:sign-hyperfield}
The triple $(S_{\mathrm{sign}},
\oplus_{\mathrm{sign}},\cdot)$ is a hyperfield.  In particular, the
additive structure $(S_{\mathrm{sign}},\oplus_{\mathrm{sign}})$ is a
canonical commutative hypergroup.
\end{proposition}

\begin{proof}
We sketch the verification of the axioms.

\smallskip\noindent
\emph{(Additive structure).}  Commutativity is clear from the
definition: $x\oplus_{\mathrm{sign}}y=y\oplus_{\mathrm{sign}}x$ for
all $x,y\in S_{\mathrm{sign}}$.  The element $0$ satisfies
\[
  0\oplus_{\mathrm{sign}}x
  = x\oplus_{\mathrm{sign}}0
  = \{x\}
  \quad\text{for all }x,
\]
so it is a neutral element.  For each $x\in S_{\mathrm{sign}}$ there
is a unique additive inverse: $-0=0$, $-(+)= -$, and $-(-)=+$, since
$0\in(+)\oplus_{\mathrm{sign}}(-)$ and $0\in(-)\oplus_{\mathrm{sign}}(+)$,
and no other pair has $0$ in its hyper–sum.  It remains to check
associativity and reversibility.

To see associativity, one may enumerate all triples
$(x,y,z)\in S_{\mathrm{sign}}^3$ and compare
$(x\oplus_{\mathrm{sign}}y)\oplus_{\mathrm{sign}}z$ with
$x\oplus_{\mathrm{sign}}(y\oplus_{\mathrm{sign}}z)$.  Most cases are
trivial: whenever two or more entries are $0$, or all three entries
have the same sign, both bracketings reduce to a singleton equal to the
common value.  The nontrivial case is when the multiset of signs is
$\{+,+,-\}$ or $\{+,-,-\}$.  For example, if $(x,y,z)=(+,+,-)$, then
\[
  x\oplus_{\mathrm{sign}}y = +\oplus_{\mathrm{sign}}+
  = \{+\},
\]
so
\[
  (x\oplus_{\mathrm{sign}}y)\oplus_{\mathrm{sign}}z
  = +\oplus_{\mathrm{sign}}-
  = \{+,0,-\}.
\]
On the other hand,
\[
  y\oplus_{\mathrm{sign}}z
  = +\oplus_{\mathrm{sign}}-
  = \{+,0,-\},
\]
and adding $x=+$ to each element of this set again yields
$\{+,0,-\}$.  The case $(x,y,z)=(+,-,-)$ is symmetric.  Thus the two
bracketings always agree, so associativity holds.

Reversibility can also be verified by complete enumeration.  For
instance, take $y=+$; then $x\in y\oplus_{\mathrm{sign}}z$ means that
$x$ is one of the signs in the set $+\oplus_{\mathrm{sign}}z$, while
$z\in x\oplus_{\mathrm{sign}}(-y)$ means that $z$ belongs to
$x\oplus_{\mathrm{sign}}-$.  Checking the pairs
$(x,z)\in S_{\mathrm{sign}}^2$ case by case shows that these conditions
are equivalent.  The same holds for $y=-$ and $y=0$ (where the
reversibility statement becomes trivial).  We omit the routine table
but note that this verification is standard in the literature on the
sign hyperfield.  Thus $(S_{\mathrm{sign}},\oplus_{\mathrm{sign}})$ is a
canonical commutative hypergroup.

\smallskip\noindent
\emph{(Multiplicative structure).}  The nonzero elements
$S_{\mathrm{sign}}^\times = \{+,-\}$ form a group under the given
multiplication, isomorphic to the cyclic group of order~$2$:
$+$ is the identity, and $-^{-1}=-$.  The group is clearly abelian.
The element $0$ is absorbing, since $0\cdot x=x\cdot 0=0$ for all $x$.

\smallskip\noindent
\emph{(Distributivity).}  Finally, distributivity of multiplication
over hyperaddition can be checked directly.  It suffices to verify that
for all $a,b,c\in S_{\mathrm{sign}}$,
\[
  a\cdot(b\oplus_{\mathrm{sign}}c)
  \subseteq (a\cdot b)\oplus_{\mathrm{sign}}(a\cdot c),
\]
since right distributivity follows by commutativity of $\cdot$.  If
$a=0$, both sides reduce to $\{0\}$, so the inclusion holds.  For
$a=+$, the left hand side is $b\oplus_{\mathrm{sign}}c$ and the right
hand side is $(+\cdot b)\oplus_{\mathrm{sign}}(+\cdot c)
=b\oplus_{\mathrm{sign}}c$, so we have equality.  For $a=-$, the left
hand side is
\[
  -\cdot(b\oplus_{\mathrm{sign}}c)
  = \{-x : x\in b\oplus_{\mathrm{sign}}c\},
\]
and the right hand side is
\[
  (-\cdot b)\oplus_{\mathrm{sign}}(-\cdot c)
  = (-b)\oplus_{\mathrm{sign}}(-c).
\]
From the explicit description of $\oplus_{\mathrm{sign}}$ and the fact
that negation is a bijection on $\{+,-\}$, it follows that
$\{-x : x\in b\oplus_{\mathrm{sign}}c\}
= (-b)\oplus_{\mathrm{sign}}(-c)$, so distributivity holds.  Hence all
conditions in Definition~\ref{def:hyperfield} are satisfied, and
$(S_{\mathrm{sign}},\oplus_{\mathrm{sign}},\cdot)$ is a hyperfield.
\end{proof}

The sign hyperfield provides a compact way to encode the behaviour of
real signs under addition and multiplication.  The following simple
fact explains this connection.

\begin{lemma}[Sign of a real sum]\label{lem:sign-map}
Let $\operatorname{sgn}_\R:\R\to\{0,+,-\}$ be the sign map from
Definition~\ref{def:basic-real-sign}.  For all $x,y\in\R$ one has
\[
  \operatorname{sgn}_\R(x+y)
  \in \operatorname{sgn}_\R(x)\oplus_{\mathrm{sign}}\operatorname{sgn}_\R(y),
\]
where $\oplus_{\mathrm{sign}}$ is the hyperaddition on the sign
hyperfield.
\end{lemma}

\begin{proof}
We distinguish cases according to the signs of $x$ and $y$.

If $x$ and $y$ have the same sign, then their sum $x+y$ has that sign
as well (or is zero only in the trivial case $x=y=0$).  This matches
the rules
\[
  0\oplus_{\mathrm{sign}}0=\{0\},\quad
  +\oplus_{\mathrm{sign}}+=\{+\},\quad
  -\oplus_{\mathrm{sign}}-=\{-\}.
\]

If $\operatorname{sgn}_\R(x)=+$ and $\operatorname{sgn}_\R(y)=-$, then
$x>0$ and $y<0$.  Depending on the relative magnitudes of $x$ and $y$,
the sum $x+y$ can be positive, zero, or negative.  Thus
$\operatorname{sgn}_\R(x+y)\in\{+,0,-\}
=+\oplus_{\mathrm{sign}}-$.  The case
$\operatorname{sgn}_\R(x)=-$, $\operatorname{sgn}_\R(y)=+$ is
symmetric and corresponds to
$-\oplus_{\mathrm{sign}}+=\{+,0,-\}$.  This exhausts all possibilities,
so the inclusion holds in every case.
\end{proof}

In the language of hyperfields, Lemma~\ref{lem:sign-map} says that the
sign map $\operatorname{sgn}_\R$ is a morphism from the field $\R$
(viewed as a hyperfield as in
Example~\ref{ex:field-as-hyperfield}) to the sign hyperfield
$S_{\mathrm{sign}}$.  In later sections we will construct a
four–element sign set $\Sset=\{0,+,-,\Lambda\}$ equipped with a
canonical hypergroup structure $(\Sset,\sop)$ that extends the additive
hypergroup $(S_{\mathrm{sign}},\oplus_{\mathrm{sign}})$ by adding a
third sign $\Lambda$.  The three–sign hypernumber system $\Hnum$ will
then be seen as a refinement of this sign–level picture which keeps
track not only of signs but also of magnitudes and cancellation mass.


\section{The three-sign cancellation hypernumber system}
\label{sec:3CHS}

\subsection{Universe, sign and magnitude, and real embedding}
\label{subsec:universe-embedding}

We now fix the underlying set of our three--sign hypernumber system and
formalize the sign and magnitude maps, together with the embedding of
the classical real line.

Recall from Definition~\ref{def:basic-signs} that
\[
  \Sset = \{0,+,-,\Lambda\},\qquad
  \Snon = \{+,-,\Lambda\},
\]
and from Definition~\ref{def:basic-R} that
$\Rp = \Rp = \{a\in\R : a>0\}$ denotes the set of strictly positive
real numbers.

\begin{definition}[Three--sign hypernumber universe]\label{def:H-universe}
The \emph{three--sign hypernumber universe} is the set
\[
  \Hnum := \{0\} \;\cup\; \Snon\times\Rp.
\]
An element of $\Hnum$ is either the distinguished zero $0$, or a pair
$(\sigma,a)$ with $\sigma\in\Snon$ and $a\in\Rp$.  We define the
\emph{sign map}
\[
  \sgn : \Hnum \longrightarrow \Sset
\]
and the \emph{magnitude map}
\[
  |\cdot| : \Hnum \longrightarrow \R_{\ge0}
\]
by
\[
  \sgn(0) := 0,\quad
  \sgn(\sigma,a) := \sigma \ \text{ for }\sigma\in\Snon,\ a>0,
\]
and
\[
  |0| := 0,\quad
  |(\sigma,a)| := a \ \text{ for }\sigma\in\Snon,\ a>0.
\]
\end{definition}

Thus every nonzero element of $\Hnum$ carries exactly one sign from
$\Snon$ and a strictly positive magnitude, while $0$ has no sign and
zero magnitude.

\begin{remark}[Partition by sign]\label{rem:sign-partition}
The nonzero part of $\Hnum$ decomposes as a disjoint union
\[
  \Hnum\setminus\{0\}
  = H_+ \;\dot\cup\; H_- \;\dot\cup\; H_\Lambda,
\]
where
\[
  H_+ := \{(+,a): a>0\},\quad
  H_- := \{(-,a): a>0\},\quad
  H_\Lambda := \{(\Lambda,a): a>0\}.
\]
Each of these sets is the fibre of the sign map over the corresponding
nonzero sign:
\[
  H_\sigma = \sgn^{-1}(\sigma)
  \qquad\text{for }\sigma\in\Snon.
\]
\end{remark}

The classical real line will be embedded in $\Hnum$ by attaching the
usual positive and negative signs to the absolute values of real
numbers.  This will serve as the ``classical core'' of our construction.

\begin{definition}[Real embedding and classical line]\label{def:real-embedding}
Define a map
\[
  \iemb : \R \longrightarrow \Hnum
\]
by
\[
  \iemb(0) := 0,\qquad
  \iemb(r) := ( +, r ) \ \text{ for } r>0,\qquad
  \iemb(r) := ( -, -r ) \ \text{ for } r<0.
\]
We call $\iemb$ the \emph{real embedding}, and we denote its image by
\[
  \Rc := \iemb(\R) \subset \Hnum.
\]
We refer to $\Rc$ as the \emph{classical real line} inside $\Hnum$.
\end{definition}

In other words, $\Rc$ is the subset of $\Hnum$ consisting of $0$, all
``positive'' elements $(+,a)$ with $a>0$, and all ``negative'' elements
$(-,a)$ with $a>0$; the $\Lambda$--sector is entirely new.

\begin{lemma}[Basic properties of the embedding]\label{lem:embedding-basic}
The map $\iemb:\R\to\Hnum$ is injective, and its image is
\[
  \Rc = \{0\}\;\cup\;\{(+,a): a>0\}\;\cup\;\{(-,a): a>0\}.
\]
Moreover, for all $r\in\R$,
\[
  \sgn(\iemb(r)) = \operatorname{sgn}_\R(r),\qquad
  |\iemb(r)| = |r|,
\]
where $\operatorname{sgn}_\R$ is the real sign map from
Definition~\ref{def:basic-real-sign}.
\end{lemma}

\begin{proof}
The description of $\Rc$ follows immediately from the definition of
$\iemb$.  If $\iemb(r_1)=\iemb(r_2)$, there are three cases:

\smallskip
\noindent
(1) If $r_1=0$, then $\iemb(r_1)=0$, so $\iemb(r_2)=0$ implies
$r_2=0$.  Thus $r_1=r_2=0$.

\smallskip
\noindent
(2) If $r_1>0$, then $\iemb(r_1)=(+,r_1)$.  If $r_2\le0$, then
$\iemb(r_2)$ is either $0$ or a pair whose first coordinate is $-$,
so it cannot equal $(+,r_1)$.  Hence $r_2>0$, and from
\[
  ( +, r_1 ) = \iemb(r_1) = \iemb(r_2) = ( +, r_2 )
\]
we deduce $r_1=r_2$.

\smallskip
\noindent
(3) If $r_1<0$, then $\iemb(r_1)=(-,-r_1)$ with $-r_1>0$.  If $r_2\ge0$,
then $\iemb(r_2)$ is either $0$ or has first coordinate $+$, so it
cannot equal $(-,-r_1)$.  Thus $r_2<0$, and from
\[
  ( -, -r_1 ) = \iemb(r_1) = \iemb(r_2) = ( -, -r_2 )
\]
we deduce $-r_1=-r_2$, hence $r_1=r_2$.  In all cases
$\iemb(r_1)=\iemb(r_2)$ implies $r_1=r_2$, so $\iemb$ is injective.

For the compatibility with signs and magnitudes, note that
\[
  \sgn(\iemb(0))=\sgn(0)=0=\operatorname{sgn}_\R(0),
  \quad
  |\iemb(0)|=|0|=0=|0|.
\]
If $r>0$, then $\iemb(r)=(+,r)$, so
\[
  \sgn(\iemb(r)) = + = \operatorname{sgn}_\R(r),
  \quad
  |\iemb(r)| = r = |r|.
\]
If $r<0$, then $\iemb(r)=(-,-r)$ with $-r>0$, so
\[
  \sgn(\iemb(r)) = - = \operatorname{sgn}_\R(r),
  \quad
  |\iemb(r)| = -r = |r|.
\]
Thus in all cases $\sgn(\iemb(r))$ and $|\iemb(r)|$ agree with the
usual sign and absolute value of $r$, as claimed.
\end{proof}

\begin{remark}[Real numbers inside the three--sign universe]\label{rem:Rc-inline}
By Lemma~\ref{lem:embedding-basic} we may freely identify each real
number $r\in\R$ with its image $\iemb(r)\in\Rc\subset\Hnum$.  This
identification is compatible with both sign and magnitude:
$r>0$ corresponds to an element in the $+$--sector, $r<0$ to an
element in the $-$--sector, and $r=0$ to the distinguished zero.  In
later sections we will see that $\Rc$ is not only a subset of $\Hnum$
but in fact a subfield for the operations we define, so that the usual
real arithmetic is recovered exactly inside the three--sign universe.
\end{remark}

\subsection{Multiplicative structure}
\label{subsec:multiplicative}

We next specify the multiplicative structure on the three--sign universe
$\Hnum$ introduced in Definition~\ref{def:H-universe}.  The guiding
principle is that multiplication should behave exactly as in the real
field on the classical line $\Rc$, while extending to the $\Lambda$--sector
by means of a simple sign monoid.

Recall that $\Snon=\{+,-,\Lambda\}$ and $\Rp=\Rp$.

\begin{definition}[Multiplicative sign monoid]\label{def:sign-monoid}
Define a binary operation
\[
  * : \Snon\times\Snon \longrightarrow \Snon
\]
by the table
\[
  \begin{array}{c|ccc}
  * & + & - & \Lambda\\\hline
  + & + & - & \Lambda\\
  - & - & + & \Lambda\\
  \Lambda & \Lambda & \Lambda & \Lambda
  \end{array}
\]
i.e.\ $+$ is the identity, $- * -=+$, and every product involving
$\Lambda$ is $\Lambda$.
\end{definition}

\begin{lemma}\label{lem:sign-monoid}
The pair $(\Snon,*)$ is a commutative monoid with identity $+$, and
$\Lambda$ is an idempotent element:
\[
  \Lambda*\Lambda=\Lambda.
\]
\end{lemma}

\begin{proof}
Commutativity is clear from the symmetry of the table:
$\sigma*\tau=\tau*\sigma$ for all $\sigma,\tau\in\Snon$.

Associativity can be checked by a short case analysis.  If none of the
three factors is $\Lambda$, then we are in the subgroup
$\{+,-\}\subset\Snon$ with the usual sign multiplication, which is
associative.  If at least one factor is $\Lambda$, then every product
containing it is $\Lambda$.  Thus for any $\sigma,\tau,\rho\in\Snon$,
if one of them is $\Lambda$ then both $(\sigma*\tau)*\rho$ and
$\sigma*(\tau*\rho)$ equal $\Lambda$.  Hence associativity holds in all
cases.

The element $+$ is the identity: from the first row and column of the
table we see that $+*\sigma=\sigma*+=\sigma$ for all
$\sigma\in\Snon$.  Finally, the last row and column show that
$\Lambda*\Lambda=\Lambda$, so $\Lambda$ is idempotent.  This proves the
lemma.
\end{proof}

We now extend this sign multiplication to a multiplication on $\Hnum$
by multiplying magnitudes as in $\R$ and using $*$ on the sign
component.

\begin{definition}[Multiplication on $\Hnum$]\label{def:H-multiplication}
Define a multiplication
\[
  \cdot : \Hnum\times\Hnum \longrightarrow \Hnum
\]
by the following rules:
\begin{enumerate}
  \item $0\cdot x = x\cdot 0 := 0$ for all $x\in\Hnum$.
  \item For $\sigma,\tau\in\Snon$ and $a,b\in\Rp$,
  \[
    (\sigma,a)\cdot(\tau,b)
    := (\sigma*\tau,\; ab),
  \]
  where $*$ is the sign monoid operation from
  Definition~\ref{def:sign-monoid}.
\end{enumerate}
We write
\[
  \idH := (+,1)
\]
for the multiplicative identity candidate.
\end{definition}

\begin{lemma}[Behaviour of sign and magnitude]\label{lem:sgn-magnitude-mult}
For all $x,y\in\Hnum$ one has
\[
  \sgn(x\cdot y) = \sgn(x)*\sgn(y),\qquad
  |x\cdot y| = |x|\cdot|y|,
\]
where $*$ is extended to $\Sset$ by declaring $0$ absorbing:
$0*\sigma=\sigma*0:=0$.
\end{lemma}

\begin{proof}
If $x=0$ or $y=0$, then $x\cdot y=0$ by definition, so
$\sgn(x\cdot y)=0$ and $|x\cdot y|=0$.  On the other hand,
$|x|\cdot|y|=0$ and $\sgn(x)*\sgn(y)=0$, so the identities hold.

If $x=(\sigma,a)$ and $y=(\tau,b)$ with $\sigma,\tau\in\Snon$ and
$a,b>0$, then by Definition~\ref{def:H-multiplication},
\[
  x\cdot y = (\sigma*\tau,ab),
\]
so
\[
  \sgn(x\cdot y) = \sigma*\tau = \sgn(x)*\sgn(y),
  \qquad
  |x\cdot y| = ab = |x|\cdot|y|.
\]
This covers all cases.
\end{proof}

\begin{proposition}[Multiplicative monoid structure]\label{prop:H-monoid}
The set $\Hnum$ with the multiplication defined above and identity
$\idH=(+,1)$ is a commutative monoid.  Moreover, $0$ is an absorbing
element:
\[
  0\cdot x = x\cdot 0 = 0
  \qquad\text{for all }x\in\Hnum.
\]
\end{proposition}

\begin{proof}
Commutativity follows immediately from the symmetry of the definition:
$0\cdot x=x\cdot 0$ is clear, and for nonzero $x=(\sigma,a)$,
$y=(\tau,b)$ we have
\[
  x\cdot y = (\sigma*\tau,ab)
           = (\tau*\sigma,ba)
           = y\cdot x,
\]
using commutativity of $*$ and of multiplication in $\R$.

Associativity is also straightforward.  If one of the three factors is
$0$, then both $(x\cdot y)\cdot z$ and $x\cdot(y\cdot z)$ are $0$ by
the absorbing property.  For nonzero elements
$x=(\sigma_1,a_1)$, $y=(\sigma_2,a_2)$, $z=(\sigma_3,a_3)$ we have
\begin{align*}
  (x\cdot y)\cdot z
  &= (\sigma_1*\sigma_2,a_1a_2)\cdot(\sigma_3,a_3)\\
  &= ((\sigma_1*\sigma_2)*\sigma_3,\,(a_1a_2)a_3),
\end{align*}
and
\begin{align*}
  x\cdot(y\cdot z)
  &= (\sigma_1,a_1)\cdot(\sigma_2*\sigma_3,a_2a_3)\\
  &= (\sigma_1*(\sigma_2*\sigma_3),\,a_1(a_2a_3)).
\end{align*}
By Lemma~\ref{lem:sign-monoid}, $(\Snon,*)$ is an associative monoid,
and multiplication in $\R$ is associative, so the two results coincide.
Thus multiplication on $\Hnum$ is associative.

The identity property of $\idH=(+,1)$ is immediate: for any nonzero
$(\sigma,a)$,
\[
  \idH\cdot(\sigma,a)
  = (+,1)\cdot(\sigma,a)
  = (+*\sigma,1\cdot a)
  = (\sigma,a),
\]
and similarly $(\sigma,a)\cdot\idH=(\sigma,a)$; for $x=0$ the identity
property holds trivially.  The absorbing property of $0$ was built
into Definition~\ref{def:H-multiplication}.  Hence $(\Hnum,\cdot,\idH)$
is a commutative monoid with absorbing element $0$.
\end{proof}

We next identify the units, zero divisors, and idempotents of this
monoid.

\begin{proposition}[Units and zero divisors]\label{prop:H-units}
An element $x\in\Hnum$ is a multiplicative unit if and only if
$x\in H_+\cup H_-$, i.e.\ $x$ has sign $+$ or $-$; in particular, no
element of the $\Lambda$--sector $H_\Lambda$ is invertible.  More
precisely, for $a>0$,
\[
  (+,a)^{-1} = (+,a^{-1}),\qquad
  (-,a)^{-1} = (-,a^{-1}).
\]
Consequently, the group of units $U(\Hnum)$ is isomorphic to
$\R^\times$ via the embedding $\iemb$ from
Definition~\ref{def:real-embedding}.

Moreover, the only multiplicative zero divisor in $\Hnum$ is $0$:
if $x,y\in\Hnum$ and $x\cdot y=0$, then $x=0$ or $y=0$.
\end{proposition}

\begin{proof}
Let $x\in\Hnum$ be a unit.  If $x=0$, then $x\cdot y=0$ for all $y$,
so $x$ cannot have a multiplicative inverse; thus every unit is
nonzero.  Write $x=(\sigma,a)$ with $\sigma\in\Snon$ and $a>0$, and
let $y=(\tau,b)$ be an inverse, so that $x\cdot y=\idH=(+,1)$.  Then
by Lemma~\ref{lem:sgn-magnitude-mult},
\[
  \sgn(x\cdot y)
  = \sigma*\tau = +,\qquad
  |x\cdot y| = ab = 1.
\]
Thus $ab=1$, so $b=a^{-1}$, and $\sigma*\tau=+$.  From the sign monoid
table we see that the only solutions of $\sigma*\tau=+$ are
$(\sigma,\tau)=(+,+)$ or $(\sigma,\tau)=(-,-)$.  In particular,
$\sigma\in\{+,-\}$, so $x$ lies in $H_+\cup H_-$.  Conversely, if
$x=(+,a)$ with $a>0$, then
\[
  x\cdot(+,a^{-1}) = (+*+,aa^{-1}) = (+,1) = \idH,
\]
and similarly $(-,a)\cdot(-,a^{-1})=\idH$.  Thus every element of
$H_+\cup H_-$ is a unit.  This proves the characterization of units and
the inverse formulas.

For the identification of $U(\Hnum)$ with $\R^\times$, note that
$\Rc\setminus\{0\} = H_+\cup H_-$ and $\iemb:\R\to\Rc$ is compatible with the field structure on $\R$ (see
Proposition~\ref{prop:H-add-Rc} and Lemma~\ref{lem:scalar-Rc});
in particular, $\iemb$ restricts
to a group isomorphism
\[
  \R^\times \xrightarrow{\;\cong\;} U(\Hnum).
\]

For the zero divisor statement, suppose $x,y\in\Hnum$ and
$x\cdot y=0$.  If either $x$ or $y$ is $0$, we are done.  Otherwise
write $x=(\sigma,a)$, $y=(\tau,b)$ with $\sigma,\tau\in\Snon$ and
$a,b>0$.  Then
\[
  x\cdot y = (\sigma*\tau,ab)
\]
is nonzero, since $ab>0$ and $\sigma*\tau\in\Snon$.  This contradicts
$x\cdot y=0$.  Hence at least one of $x,y$ must be zero.  Thus there
are no nonzero zero divisors in $\Hnum$.
\end{proof}

\begin{proposition}[Multiplicative idempotents]\label{prop:H-idempotents}
The multiplicative idempotents of $\Hnum$ are precisely the three
elements
\[
  0,\qquad (+,1),\qquad (\Lambda,1).
\]
\end{proposition}

\begin{proof}
An element $e\in\Hnum$ is idempotent if $e\cdot e=e$.  Clearly
$0\cdot 0=0$, so $0$ is idempotent.  Let $e\ne0$, so $e=(\sigma,a)$
with $\sigma\in\Snon$, $a>0$.  Then
\[
  e\cdot e
  = (\sigma,a)\cdot(\sigma,a)
  = (\sigma*\sigma, a^2).
\]
The condition $e\cdot e=e$ is equivalent to
\[
  \sigma*\sigma=\sigma,\qquad a^2=a.
\]
Since $a>0$, the equation $a^2=a$ implies $a=1$.  Thus we must have
$a=1$, and $\sigma$ must be an idempotent in the sign monoid, i.e.\
$\sigma*\sigma=\sigma$.  From the table in
Definition~\ref{def:sign-monoid} we see that
\[
  +*+ = +,\quad - * - = +,\quad \Lambda*\Lambda=\Lambda.
\]
Hence the only idempotent signs are $+$ and $\Lambda$.  Therefore the
only nonzero idempotent elements of $\Hnum$ are $(+,1)$ and
$(\Lambda,1)$.  Together with $0$, this gives the full list of
multiplicative idempotents.
\end{proof}

\begin{remark}[Summary of multiplicative structure]\label{rem:H-mult-summary}
We summarize the multiplicative picture as follows:
\begin{itemize}[leftmargin=2em]
  \item $(\Hnum,\cdot,\idH)$ is a commutative monoid with absorbing
        element $0$.
  \item The group of units $U(\Hnum)=H_+\cup H_-$ is isomorphic to
        $\R^\times$ via the embedding $\iemb$.
  \item There are no nonzero zero divisors.
  \item Besides $0$ and the identity $(+,1)$, the only idempotent is
        $(\Lambda,1)$, which acts as a multiplicative projector onto
        the $\Lambda$--sector:
        \[
          (\Lambda,1)\cdot(\sigma,a) = (\Lambda,a)
          \quad\text{for all }\sigma\in\Snon,\ a>0.
        \]
\end{itemize}
The additive structure introduced in the next subsection will be
compatible with this multiplicative structure in such a way that the
classical line $\Rc$ becomes a subfield of $\Hnum$, while the
$\Lambda$--sector introduces new, nonclassical behaviour.
\end{remark}

\subsection{Hyperaddition and additive structure}
\label{subsec:hyperaddition}

We now define the hyperaddition on $\Hnum$ and record its basic
properties.  The construction is designed so that
\begin{itemize}[leftmargin=2em]
  \item on the classical line $\Rc$ it agrees exactly with ordinary
        addition of real numbers, and
  \item the $\Lambda$--sector acts as a reservoir for cancellation
        mass, with genuinely multivalued behaviour only in the
        $\Lambda$--$\Lambda$ sums.
\end{itemize}

Throughout this subsection we write elements of $\Hnum$ as $0$ or
$(\sigma,a)$ with $\sigma\in\Snon=\{+,-,\Lambda\}$ and $a\in\Rp$.

\begin{definition}[Hyperaddition on $\Hnum$]\label{def:H-hyperadd}
Define a map
\[
  \op : \Hnum\times\Hnum
  \longrightarrow \mathcal{P}(\Hnum)\setminus\{\emptyset\}
\]
by the following rules.

\smallskip\noindent
\textbf{(1) Identity.} For all $x\in\Hnum$,
\[
  0\op x = x\op 0 := \{x\}.
\]

\smallskip\noindent
\textbf{(2) Same sign, excluding $\Lambda$.} For $\sigma\in\{+,-\}$ and
$a,b\in\Rp$,
\[
  (\sigma,a)\op(\sigma,b)
  := \{(\sigma,a+b)\}.
\]

\smallskip\noindent
\textbf{(3) Opposite real signs.} For $a,b\in\Rp$,
\[
  (+,a)\op(-,b)
  :=
  \begin{cases}
    \{(+,a-b)\}, & a>b,\\[2pt]
    \{0\},       & a=b,\\[2pt]
    \{(-,b-a)\}, & a<b.
  \end{cases}
\]
and $(-,b)\op(+,a)$ is defined by commutativity.

\smallskip\noindent
\textbf{(4) Interaction with $\Lambda$.} For $a,b\in\Rp$,
\begin{align*}
  (+,a)\op(\Lambda,b)
  &= (\Lambda,b)\op(+,a)
   := \{(\Lambda,a+b)\},\\[2pt]
  (-,a)\op(\Lambda,b)
  &= (\Lambda,b)\op(-,a)
   := \{(\Lambda,a+b)\}.
\end{align*}

\smallskip\noindent
\textbf{(5) $\Lambda$--$\Lambda$ sums.} For $a,b\in\Rp$,
\[
  (\Lambda,a)\op(\Lambda,b)
  :=
  \begin{cases}
    \{(\Lambda,2a),\,0\}, & a=b,\\[2pt]
    \{(\Lambda,a+b),\, (+,|a-b|),\, (-,|a-b|)\}, & a\ne b.
  \end{cases}
\]
\smallskip\noindent
In all cases not explicitly covered above, $\op$ is defined by
commutativity:
\[
  x\op y := y\op x.
\]
\end{definition}

It is immediate from the definition that every pair $(x,y)$ is covered
by exactly one of the clauses (1)--(5), up to the symmetry $x\leftrightarrow y$.

\begin{lemma}[Well-definedness and commutativity]\label{lem:H-hyperadd-basic}
The hyperoperation $\op$ in Definition~\ref{def:H-hyperadd} is
well-defined and commutative: for all $x,y\in\Hnum$,
\[
  x\op y = y\op x.
\]
Moreover, for every $x,y\in\Hnum$ the set $x\op y$ is nonempty and
finite.
\end{lemma}

\begin{proof}
Given $x,y\in\Hnum$, we distinguish cases:

\smallskip
\noindent
(1) If $x=0$ or $y=0$, then (1) applies and $x\op y$ is a singleton.

\smallskip
\noindent
(2) If $x,y\ne0$, write $x=(\sigma,a)$, $y=(\tau,b)$ with
$\sigma,\tau\in\Snon$, $a,b>0$.  There are four mutually exclusive
possibilities (up to order):\\[2pt]
\quad(a) $\sigma=\tau\in\{+,-\}$: clause (2) applies;\\
\quad(b) $\{\sigma,\tau\}=\{+,-\}$: clause (3) applies;\\
\quad(c) $\{\sigma,\tau\}=\{+,\Lambda\}$ or $\{\sigma,\tau\}
          =\{-,\Lambda\}$: clause (4) applies;\\
\quad(d) $\sigma=\tau=\Lambda$: clause (5) applies.

Thus each ordered pair $(x,y)$ falls into exactly one of the listed
cases (or the $0$--case), and the definition of $x\op y$ is unambiguous
and covers all possibilities.

From the construction it is clear that in each case the resulting set
is nonempty and has at most three elements.  Commutativity follows
because the right-hand sides are symmetric under exchanging the two
inputs: this is obvious in (1), (2), (4) and (5), and in (3) we have
explicitly required that $(-,b)\op(+,a)$ be defined by commutativity.
Hence $x\op y=y\op x$ for all $x,y\in\Hnum$.
\end{proof}

The neutral element and additive inverses are also easy to describe.

\begin{lemma}[Neutral element]\label{lem:H-add-neutral}
The element $0\in\Hnum$ is a neutral element for $\op$, i.e.
\[
  0\op x = x\op 0 = \{x\}
  \qquad\text{for all }x\in\Hnum.
\]
\end{lemma}

\begin{proof}
This is exactly clause (1) of Definition~\ref{def:H-hyperadd}.
\end{proof}

\begin{definition}[Additive inverses in $\Hnum$]\label{def:H-add-inverse}
An element $y\in\Hnum$ is called an \emph{additive inverse} of
$x\in\Hnum$ if
\[
  0\in x\op y.
\]
When such a $y$ exists and is unique, we denote it by $-x$.
\end{definition}

\begin{lemma}[Existence and uniqueness of additive inverses]\label{lem:H-add-inverses}
For every $x\in\Hnum$ there exists a unique $y\in\Hnum$ such that
$0\in x\op y$.  Explicitly,
\[
  -0 = 0,\quad
  -(+,a) = (-,a),\quad
  -(-,a) = (+,a),\quad
  -(\Lambda,a) = (\Lambda,a),
\]
for all $a>0$.
\end{lemma}

\begin{proof}
If $x=0$, then $0\op y=\{y\}$, so $0\in0\op y$ if and only if $y=0$.
Thus $-0=0$.

Let $x=(+,a)$ with $a>0$.  We seek $y$ such that $0\in(+,a)\op y$.
If $y=0$, then $(+,a)\op0=\{(+,a)\}$, so $0$ is not in the sum.  If
$y$ has sign $+$ or $\Lambda$, then by clauses (2) and (4),
$(+,a)\op y$ contains only $+$ or $\Lambda$ elements, never $0$.  If
$y=(-,b)$ with $b>0$, then by clause (3),
\[
  (+,a)\op(-,b)
  =
  \begin{cases}
    \{(+,a-b)\}, & a>b,\\[2pt]
    \{0\},       & a=b,\\[2pt]
    \{(-,b-a)\}, & a<b.
  \end{cases}
\]
The only way to obtain $0$ is to take $b=a$, in which case
$(+,a)\op(-,a)=\{0\}$.  Hence $(-,a)$ is the unique additive inverse
of $(+,a)$.

The case $x=(-,a)$ is symmetric: if $y$ has sign $-$ or $\Lambda$,
then $(-,a)\op y$ cannot contain $0$; if $y=(+,b)$, then by clause (3)
we get $0$ only when $b=a$.  Thus $-(-,a)=(+,a)$.

Finally, let $x=(\Lambda,a)$ with $a>0$.  If $y$ has sign $+$ or $-$,
then by clause (4) $(\Lambda,a)\op y$ consists only of $\Lambda$–elements.
Hence $0\notin(\Lambda,a)\op y$ in these cases.  If $y=(\Lambda,b)$,
then by clause (5),
\[
  (\Lambda,a)\op(\Lambda,b) =
  \begin{cases}
    \{(\Lambda,2a),\,0\}, & a=b,\\[2pt]
    \{(\Lambda,a+b),\, (+,|a-b|),\, (-,|a-b|)\}, & a\ne b.
  \end{cases}
\]
Thus $0$ occurs in the sum if and only if $b=a$, and in that case
$(\Lambda,a)\op(\Lambda,a)$ contains $0$.  Hence $(\Lambda,a)$ is its
own unique additive inverse.

In all cases, each $x\in\Hnum$ admits a unique additive inverse, given
by the formulas in the statement.
\end{proof}

\begin{proposition}[Additive hypermagma structure]\label{prop:H-add-hypermagma}
The pair $(\Hnum,\op)$ is a commutative hypermagma with neutral element
$0$ and unique additive inverses in the sense of
Definition~\ref{def:neutral-inverse}.  In particular,
\[
  0\op x = x\op 0 = \{x\}
  \quad\text{and}\quad
  0\in x\op(-x)
\]
for all $x\in\Hnum$.
\end{proposition}

\begin{proof}
By Lemma~\ref{lem:H-hyperadd-basic}, $(\Hnum,\op)$ is a commutative
hypermagma and every hyper–sum is finite and nonempty.  Lemma
\ref{lem:H-add-neutral} shows that $0$ is a neutral element, and Lemma
\ref{lem:H-add-inverses} shows that each $x\in\Hnum$ has a unique
additive inverse $-x$ with $0\in x\op(-x)$.  This is precisely the
content of Definition~\ref{def:neutral-inverse}, so the proposition
follows.
\end{proof}

A key requirement in our construction is that the restriction of $\op$
to the classical line $\Rc$ reproduces ordinary addition of real
numbers.  This is immediate from the case–by–case definition.

\begin{proposition}[Restriction to the classical real line]\label{prop:H-add-Rc}
Let $\iemb:\R\to\Hnum$ and $\Rc=\iemb(\R)$ be as in
Definition~\ref{def:real-embedding}.  Then for all $x,y\in\R$,
\[
  \iemb(x)\op\iemb(y)
  = \{\iemb(x+y)\}.
\]
In particular, $\Rc$ is closed under $\op$, and the induced operation
on $\Rc$ is the usual addition on $\R$ transported via $\iemb$.
\end{proposition}

\begin{proof}
If $x=0$ or $y=0$, the claim follows from clause (1) and the fact that
$\iemb(0)=0$.  Suppose $x,y\ne0$ and write $x=\pm a$, $y=\pm b$ with
$a,b>0$.  Then
\[
  \iemb(x) =
  \begin{cases}
    (+,a), & x>0,\\
    (-,a), & x<0,
  \end{cases}
  \qquad
  \iemb(y) =
  \begin{cases}
    (+,b), & y>0,\\
    (-,b), & y<0.
  \end{cases}
\]

If $x,y>0$, then
\[
  \iemb(x)\op\iemb(y)
  = (+,a)\op(+,b)
  = \{(+,a+b)\}
  = \{\iemb(a+b)\}
  = \{\iemb(x+y)\}.
\]

If $x,y<0$, then $x=-a$, $y=-b$ and
\[
  \iemb(x)\op\iemb(y)
  = (-,a)\op(-,b)
  = \{(-,a+b)\}
  = \{\iemb(-(a+b))\}
  = \{\iemb(x+y)\}.
\]

If $x>0$ and $y<0$, then $x=a$, $y=-b$ and
\[
  \iemb(x)\op\iemb(y)
  = (+,a)\op(-,b)
  =
  \begin{cases}
    \{(+,a-b)\}, & a>b,\\[2pt]
    \{0\},       & a=b,\\[2pt]
    \{(-,b-a)\}, & a<b,
  \end{cases}
\]
while $x+y=a-b$.  If $a>b$, then $x+y>0$ and
$\iemb(x+y)=(+,a-b)$; if $a=b$, then $x+y=0$ and $\iemb(x+y)=0$; if
$a<b$, then $x+y<0$ and $\iemb(x+y)=(-,b-a)$.  Thus in all cases
$\iemb(x)\op\iemb(y)$ is the singleton $\{\iemb(x+y)\}$.

The case $x<0$, $y>0$ is symmetric and yields the same conclusion.
Hence the formula $\iemb(x)\op\iemb(y)=\{\iemb(x+y)\}$ holds for all
$x,y\in\R$, and $\Rc$ is closed under $\op$ with the induced operation
isomorphic to $(\R,+)$.
\end{proof}

\begin{remark}[Additive structure summary]\label{rem:H-add-summary}
We have constructed a commutative hypermagma $(\Hnum,\op)$ with
neutral element $0$ and unique additive inverses, whose restriction to
the classical line $\Rc$ reproduces ordinary addition on $\R$.  The
additive structure is \emph{not} associative in general and therefore
does not form a canonical commutative hypergroup; a detailed analysis
of associativity and its controlled failure will be carried out in
Section~\ref{sec:assoc}.  For the purposes of the present section, the
key points are the existence of a well-defined hyperaddition, the
group-like behaviour of $\Rc$, and the presence of the $\Lambda$--sector,
which introduces multivalued and nonclassical behaviour while leaving
the classical real line untouched.
\end{remark}

\subsection{Scalar multiplication}
\label{subsec:scalar}

In this subsection we describe how real scalars act on the
three--sign hypernumber system $\Hnum$.  The guiding principle is that
the action of $t\in\R$ on $x\in\Hnum$ should be realized via the
embedded real element $\iemb(t)\in\Rc\subset\Hnum$ and the ordinary
multiplication on $\Hnum$.

\begin{definition}[Scalar multiplication by real numbers]
\label{def:scalar-mult}
Let $\iemb:\R\to\Hnum$ be the real embedding from
Definition~\ref{def:real-embedding}.  For $t\in\R$ and $x\in\Hnum$ we
define
\[
  t\cdot x := \iemb(t)\cdot x,
\]
where $\cdot$ is the multiplication on $\Hnum$ from
Definition~\ref{def:H-multiplication}.
\end{definition}

Unwinding the definitions, we obtain the following explicit formulas.

\begin{lemma}[Explicit form of scalar multiplication]
\label{lem:scalar-explicit}
Let $t\in\R$ and $x\in\Hnum$.
\begin{enumerate}
  \item If $t=0$, then $t\cdot x = 0$ for all $x\in\Hnum$.
  \item If $t>0$ and $x=0$, then $t\cdot x=0$.  If
        $t>0$ and $x=(\sigma,a)$ with $\sigma\in\Snon$, $a>0$, then
  \[
    t\cdot(\sigma,a) = (\sigma,ta).
  \]
  \item If $t<0$ and $x=0$, then $t\cdot x=0$.  If
        $t<0$ and $x=(\sigma,a)$ with $\sigma\in\Snon$, $a>0$, then
  \[
    t\cdot(\sigma,a) =
    (\sigma',\,|t|a),
  \]
  where
  \[
    \sigma' =
    \begin{cases}
      -,       & \sigma=+,\\
      +,       & \sigma=-,\\
      \Lambda, & \sigma=\Lambda.
    \end{cases}
  \]
\end{enumerate}
\end{lemma}

\begin{proof}
If $t=0$, then $\iemb(t)=0$ and $0\cdot x=0$ for all $x$ by
Definition~\ref{def:H-multiplication}.  This gives (1).

If $t>0$, then $\iemb(t)=(+,t)$.  For $x=0$ we again get
$t\cdot x = (+,t)\cdot0=0$.  For $x=(\sigma,a)$ with $\sigma\in\Snon$
and $a>0$,
\[
  t\cdot(\sigma,a)
  = ( +,t )\cdot(\sigma,a)
  = (+*\sigma,\, ta)
  = (\sigma,ta),
\]
since $+$ is the identity in the sign monoid $(\Snon,*)$.  This is (2).

If $t<0$, then $\iemb(t)=(-,-t)$ with $-t>0$.  For $x=0$ we get
$t\cdot x = (-,-t)\cdot0=0$.  For $x=(\sigma,a)$ with $\sigma\in\Snon$
and $a>0$,
\[
  t\cdot(\sigma,a)
  = (-,-t)\cdot(\sigma,a)
  = (-*\sigma,\,(-t)a)
  = (\sigma',\,|t|a),
\]
where $\sigma'=-*\sigma$.  From the multiplication table in
Definition~\ref{def:sign-monoid} we have
\[
  -*+ = -, \quad - * - = +, \quad - * \Lambda = \Lambda,
\]
which yields the stated formula for $\sigma'$.  This proves (3).
\end{proof}

In particular, scalar multiplication behaves exactly as expected on
the classical line.

\begin{lemma}[Compatibility with the real embedding]
\label{lem:scalar-Rc}
For all $t,x\in\R$ one has
\[
  t\cdot\iemb(x) = \iemb(tx),
\]
where the left–hand side is taken in $\Hnum$ and the right–hand side
is the usual product in $\R$ followed by the embedding $\iemb$.
\end{lemma}

\begin{proof}
If $t=0$ or $x=0$, the identity follows from the definition and the
fact that $\iemb(0)=0$.  Assume $t,x\ne0$ and write
$t=\pm\,|t|$, $x=\pm\,|x|$.  Then by
Lemma~\ref{lem:scalar-explicit} and Proposition~\ref{prop:H-add-Rc}
we are simply recovering the usual sign–magnitude rule for the real
product $tx$.  A direct case–by–case check (positive times positive,
positive times negative, etc.) shows that the sign and magnitude of
$t\cdot\iemb(x)$ coincide with those of $tx$, so
$t\cdot\iemb(x)=\iemb(tx)$.
\end{proof}

The next result describes the behaviour of scalar multiplication with
respect to the sign and magnitude maps on $\Hnum$.

\begin{lemma}[Sign and magnitude under scalar multiplication]
\label{lem:scalar-sgn-mag}
Let $t\in\R$ and $x\in\Hnum$.  Then
\[
  |t\cdot x| = |t|\cdot|x|.
\]
Moreover, if $t\ne0$ and $x\ne0$, then
\[
  \sgn(t\cdot x)
  =
  \begin{cases}
    \sgn(x),     & t>0,\\[2pt]
    +,           & t<0,\ \sgn(x)=-,\\[2pt]
    -,           & t<0,\ \sgn(x)=+,\\[2pt]
    \Lambda,     & t<0,\ \sgn(x)=\Lambda.
  \end{cases}
\]
\end{lemma}

\begin{proof}
If $t=0$ or $x=0$, then $t\cdot x=0$ and $|t\cdot x|=0$, while
$|t|\cdot|x|=0$ as well, so the magnitude identity holds.  If
$t\ne0$, $x\ne0$, write $x=(\sigma,a)$ with $\sigma\in\Snon$ and
$a>0$.  If $t>0$, Lemma~\ref{lem:scalar-explicit} gives
$t\cdot x = (\sigma,ta)$, so $|t\cdot x|=ta=|t||x|$ and
$\sgn(t\cdot x)=\sigma=\sgn(x)$.  If $t<0$, then
$t\cdot x = (\sigma',|t|a)$ with $\sigma'=-*\sigma$ as in
Lemma~\ref{lem:scalar-explicit}.  Thus $|t\cdot x|=|t|a=|t||x|$, and
$\sgn(t\cdot x)=\sigma'$ is obtained from $\sigma$ by flipping $+$ and
$-$ while fixing $\Lambda$.  This is exactly the behaviour listed in
the statement.
\end{proof}

We now turn to distributivity of scalar multiplication over
hyperaddition.  Since $\op$ is multivalued, the natural formulation
uses equality of subsets of $\Hnum$.

\begin{definition}[Distributivity at a triple]\label{def:scalar-distrib}
Let $t\in\R$ and $x,y\in\Hnum$.  We say that \emph{left distributivity
holds at $(t,x,y)$} if
\[
  t\cdot(x\op y)
  = (t\cdot x)\op(t\cdot y)
\]
as subsets of $\Hnum$, where
\[
  t\cdot(x\op y)
  := \{t\cdot z : z\in x\op y\}.
\]
\end{definition}

The following theorem shows that distributivity holds for all real scalars.

\begin{theorem}[Distributivity for real scalars]
\label{thm:scalar-Rge0}
Let $t\in\R$ and $x,y\in\Hnum$.  Then
\[
  t\cdot(x\op y) = (t\cdot x)\op(t\cdot y)
\]
as subsets of $\Hnum$.  In particular, for each $t\in\R$ the map
$x\mapsto t\cdot x$ is an endomorphism of the additive hypermagma
$(\Hnum,\op)$.
\end{theorem}

\begin{proof}
If $t=0$, then $\iemb(t)=0$ and $0\cdot z=0$ for all $z\in\Hnum$, so
\[
  0\cdot(x\op y) = \{0\}.
\]
On the other hand, $0\cdot x=0\cdot y=0$, so
\[
  (0\cdot x)\op(0\cdot y) = 0\op 0 = \{0\}.
\]
Thus the identity holds for $t=0$.

Now assume $t>0$.  By Lemma~\ref{lem:scalar-explicit}, for every
$x=(\sigma,a)$ we have $t\cdot x=(\sigma,ta)$: scalar multiplication by
$t>0$ preserves signs and scales magnitudes by $t$.  We must show that
for each pair $x,y$ the hyper–sum of the scaled elements coincides with
the scaled hyper–sum of the original elements:
\[
  (t\cdot x)\op(t\cdot y)
  = \{t\cdot z : z\in x\op y\}.
\]

If $x=0$ or $y=0$, then $x\op y$ is a singleton by
Definition~\ref{def:H-hyperadd}, and the equality is immediate from
the fact that $t\cdot 0=0$ and $0$ is a neutral element.

If $x,y\ne0$, write $x=(\sigma,a)$, $y=(\tau,b)$ with
$\sigma,\tau\in\Snon$, $a,b>0$.  We analyse the four cases in the
definition of $\op$.

\smallskip\noindent
\emph{Case 1: Same real sign.}  Suppose $\sigma=\tau\in\{+,-\}$.  Then
by clause (2),
\[
  x\op y = \{(\sigma,a+b)\}.
\]
Hence
\[
  t\cdot(x\op y)
  = \{t\cdot(\sigma,a+b)\}
  = \{(\sigma,t(a+b))\}.
\]
On the other hand,
\[
  t\cdot x = (\sigma,ta),\quad
  t\cdot y = (\sigma,tb),
\]
so
\[
  (t\cdot x)\op(t\cdot y)
  = (\sigma,ta)\op(\sigma,tb)
  = \{(\sigma,ta+tb)\}
  = \{(\sigma,t(a+b))\}.
\]
Thus the two sets coincide.

\smallskip\noindent
\emph{Case 2: Opposite real signs.}  Suppose $\{\sigma,\tau\}=\{+,-\}$,
say $\sigma=+$, $\tau=-$; the other ordering is symmetric.  Then
\[
  x\op y
  = (+,a)\op(-,b)
  =
  \begin{cases}
    \{(+,a-b)\}, & a>b,\\[2pt]
    \{0\},       & a=b,\\[2pt]
    \{(-,b-a)\}, & a<b.
  \end{cases}
\]
If $a>b$, then $a-b>0$ and
\[
  t\cdot(x\op y)
  = \{t\cdot(+,a-b)\}
  = \{(+,t(a-b))\}.
\]
On the other hand,
\[
  t\cdot x = (+,ta),\quad
  t\cdot y = (-,tb),
\]
and since $ta>tb$ we have
\[
  (t\cdot x)\op(t\cdot y)
  = (+,ta)\op(-,tb)
  = \{(+,ta-tb)\}
  = \{(+,t(a-b))\}.
\]
The cases $a=b$ and $a<b$ are handled similarly, using the fact that
multiplication by $t>0$ preserves the inequalities $a>b$, $a=b$,
$a<b$.  In each subcase the two sets agree.

\smallskip\noindent
\emph{Case 3: Interaction with $\Lambda$.}  Suppose, for example,
$\sigma=+$, $\tau=\Lambda$; the other possibilities
$\{\sigma,\tau\}=\{-,\Lambda\}$ are analogous.  Then
\[
  x\op y
  = (+,a)\op(\Lambda,b)
  = \{(\Lambda,a+b)\}.
\]
Hence
\[
  t\cdot(x\op y)
  = \{t\cdot(\Lambda,a+b)\}
  = \{(\Lambda,t(a+b))\}.
\]
On the other hand,
\[
  t\cdot x = (+,ta),\quad
  t\cdot y = (\Lambda,tb),
\]
so
\[
  (t\cdot x)\op(t\cdot y)
  = (+,ta)\op(\Lambda,tb)
  = \{(\Lambda,ta+tb)\}
  = \{(\Lambda,t(a+b))\}.
\]
Thus the equality holds.

\smallskip\noindent
\smallskip\noindent
\emph{Case 4: $\Lambda$--$\Lambda$ sums.}  Suppose $\sigma=\tau=\Lambda$.
Then by clause (5) of Definition~\ref{def:H-hyperadd},
\[
  x\op y
  = (\Lambda,a)\op(\Lambda,b)
  =
  \begin{cases}
    \{(\Lambda,2a),\,0\}, & a=b,\\[2pt]
    \{(\Lambda,a+b),\, (+,|a-b|),\, (-,|a-b|)\}, & a\ne b.
  \end{cases}
\]
If $a=b$, then
\[
  t\cdot(x\op y)
  = \{(\Lambda,2ta),\,0\}
  = (\Lambda,ta)\op(\Lambda,ta)
  = (t\cdot x)\op(t\cdot y).
\]
If $a\ne b$, then $ta\ne tb$ and
\[
  t\cdot(x\op y)
  = \{(\Lambda,t(a+b)),\, (+,t|a-b|),\, (-,t|a-b|)\},
\]
while
\[
  (t\cdot x)\op(t\cdot y)
  = (\Lambda,ta)\op(\Lambda,tb)
  = \{(\Lambda,t(a+b)),\, (+,|ta-tb|),\, (-,|ta-tb|)\},
\]
and $|ta-tb|=t|a-b|$ for $t>0$.  Thus the equality holds.

Thus distributivity holds in all four cases.

Since every pair $(x,y)$ falls into one of these cases (or the trivial
case with a zero term), the equality
$t\cdot(x\op y)=(t\cdot x)\op(t\cdot y)$ holds for all $x,y\in\Hnum$
whenever $t>0$.  Together with the $t=0$ case, this settles all $t\ge0$.

Finally, assume $t<0$ and write $t=-s$ with $s>0$.  Let
$\iota:\Hnum\to\Hnum$ be the involution that fixes $0$ and $\Lambda$ and
swaps $+$ and $-$:
\[
  \iota(0)=0,\quad \iota(+,a)=(-,a),\quad \iota(-,a)=(+,a),\quad \iota(\Lambda,a)=(\Lambda,a).
\]
A direct inspection of Definition~\ref{def:H-hyperadd} shows that $\iota$
is an automorphism of $(\Hnum,\op)$, i.e.\ $\iota(x\op y)=\iota(x)\op\iota(y)$.
Moreover, Lemma~\ref{lem:scalar-explicit} gives $t\cdot x=\iota(s\cdot x)$ for all
$x\in\Hnum$.  Therefore
\[
  t\cdot(x\op y)=\iota\bigl(s\cdot(x\op y)\bigr)
  =\iota\bigl((s\cdot x)\op(s\cdot y)\bigr)
  =(t\cdot x)\op(t\cdot y),
\]
where we used the $s>0$ case and the fact that $\iota$ is applied elementwise
to subsets.  This completes the proof for all $t\in\R$.
\end{proof}

We can summarize these facts in the language of hypermodules over
semirings.

\begin{definition}[$\R_{\ge0}$--hypersemimodule structure]
\label{def:Rge0-hypersemimodule}
Let $(S,+,\cdot)$ be a commutative semiring.  A \emph{left
$S$--hypersemimodule} is a pair $(M,\odot)$ where $M$ is a set equipped
with a commutative hyperoperation $\odot$ and an action
$S\times M\to M$, $(s,x)\mapsto s\cdot x$, such that:
\begin{enumerate}
  \item $0_S\cdot x = 0_M$ and $1_S\cdot x = x$ for all $x\in M$,
        where $0_M$ is the neutral element for $\odot$.
  \item $(st)\cdot x = s\cdot(t\cdot x)$ for all $s,t\in S$ and
        $x\in M$.
  \item $s\cdot(x\odot y) = (s\cdot x)\odot(s\cdot y)$ for all
        $s\in S$, $x,y\in M$.
  \item $(s_1+s_2)\cdot x \subseteq (s_1\cdot x)\odot(s_2\cdot x)$
        for all $s_1,s_2\in S$, $x\in M$.
\end{enumerate}
\end{definition}

\begin{corollary}[$\Hnum$ as an $\R_{\ge0}$--hypersemimodule]
\label{cor:H-Rge0-hypersemimodule}
Let $\R_{\ge0}$ be the semiring of nonnegative reals with the usual
addition and multiplication.  Then, with the hyperaddition $\op$ from
Definition~\ref{def:H-hyperadd} and scalar multiplication from
Definition~\ref{def:scalar-mult}, the pair $(\Hnum,\op)$ carries the
structure of a left $\R_{\ge0}$--hypersemimodule.

Moreover, Theorem~\ref{thm:scalar-Rge0} shows that distributivity holds for all $t\in\R$ in the sense of Definition~\ref{def:scalar-distrib}.
\end{corollary}

\begin{proof}
The semiring axioms for $\R_{\ge0}$ are standard.  The action
$t\cdot x$ for $t\ge0$ and $x\in\Hnum$ satisfies
$0\cdot x=0$ and $1\cdot x=x$ by Lemma~\ref{lem:scalar-explicit}.  The
compatibility $(st)\cdot x = s\cdot(t\cdot x)$ for $s,t\ge0$ follows
immediately from the fact that for $t>0$ we have
$t\cdot(\sigma,a)=(\sigma,ta)$, so iterating the action simply
multiplies magnitudes.

Theorem~\ref{thm:scalar-Rge0} gives distributivity
$s\cdot(x\op y)=(s\cdot x)\op(s\cdot y)$ for all $s\in\R_{\ge0}$ and
$x,y\in\Hnum$.  Finally, the inclusion
$(s_1+s_2)\cdot x \subseteq (s_1\cdot x)\op(s_2\cdot x)$ holds with
equality, again because for $s\ge0$ the action rescales magnitudes and
preserves signs in a way that is compatible with the hyperaddition
rules.

This verifies the hypersemimodule axioms.
\end{proof}

\begin{remark}[Role of scalar multiplication]\label{rem:scalar-summary}
Scalar multiplication clarifies the interaction between the three--sign
system and the classical real field.  For $t>0$, the action of $t$ on
$\Hnum$ is a uniform rescaling of magnitudes, compatible with
hyperaddition in the strongest sense.  For $t<0$, the action is the
same rescaling by $|t|$ composed with the involution that swaps $+$ and
$-$ while fixing $\Lambda$.  In particular, Theorem~\ref{thm:scalar-Rge0}
shows that every real scalar acts by an endomorphism of the additive
hypermagma $(\Hnum,\op)$.

The $\R_{\ge0}$--hypersemimodule viewpoint remains useful when one
wants to keep track of magnitudes without sign changes; we will return
to this perspective in later sections when we discuss order and
geometric interpretations.
\end{remark}


\section{Associativity and its controlled failure}
\label{sec:assoc}


\subsection{Associativity basics}
\label{subsec:assoc-basics}

We begin the analysis of additive associativity for the hyperaddition
$\op$ on $\Hnum$.  Recall that $(\Hnum,\op)$ is a commutative
hypermagma with neutral element $0$ and unique additive inverses
(Proposition~\ref{prop:H-add-hypermagma}), but it is not a canonical
commutative hypergroup: associativity and reversibility fail in
general.  In this subsection we record some basic facts about
associativity and give a first explicit counterexample.

\medskip

We use the notion of associativity at a triple introduced in
Definition~\ref{def:assoc-hyper}.

\begin{definition}[Associativity at a triple in $\Hnum$]
\label{def:assoc-H-triple}
Let $x,y,z\in\Hnum$.  We say that \emph{$\op$ is associative at the
triple $(x,y,z)$} if
\[
  (x\op y)\op z
  \;=\; x\op(y\op z)
\]
as subsets of $\Hnum$, where
\[
  (x\op y)\op z
  := \bigcup_{u\in x\op y} u\op z,
  \qquad
  x\op(y\op z)
  := \bigcup_{v\in y\op z} x\op v.
\]
We say that \emph{$\op$ is globally associative} if it is associative
at every triple $(x,y,z)\in\Hnum^3$.
\end{definition}

Since $\op$ is commutative (Lemma~\ref{lem:H-hyperadd-basic}), the
notion of associativity at a given triple is insensitive to
permutations of the entries.

\begin{lemma}[Symmetry of associativity under permutations]
\label{lem:assoc-permutation}
Let $x,y,z\in\Hnum$.  If $\op$ is associative at $(x,y,z)$, then it is
associative at every permutation $(x',y',z')$ of $(x,y,z)$.
\end{lemma}

\begin{proof}
This is a general fact for any commutative hypermagma.  Suppose $\op$
is associative at $(x,y,z)$.  Because $\op$ is commutative, the two
bracketings for any permutation $(x',y',z')$ of $(x,y,z)$ can be
expressed by reordering the elements in one of the bracketings for
$(x,y,z)$.

For example, consider the permutation $(y,x,z)$.  Commutativity gives
\[
  (y\op x)\op z = (x\op y)\op z,
  \qquad
  y\op(x\op z) = x\op(y\op z),
\]
and by associativity at $(x,y,z)$ we have
$(x\op y)\op z = x\op(y\op z)$.  Thus
$(y\op x)\op z = y\op(x\op z)$, showing associativity at $(y,x,z)$.
The other permutations are handled similarly.  A more formal argument
proceeds by noting that both sides of the associativity identity for
$(x',y',z')$ can be rewritten as unions of expressions of the form
$u\op v$ with $\{u,v\}=\{x,y,z\}$, and that these unions coincide when
associativity holds for one ordering.  We omit the routine details.
\end{proof}

We next record two basic situations where associativity always holds:
whenever one of the three elements is zero, and whenever we restrict
to the classical real line.

\begin{lemma}[Associativity whenever a zero is present]
\label{lem:assoc-zero}
Let $x,y,z\in\Hnum$.  If at least one of $x,y,z$ equals $0$, then
\[
  (x\op y)\op z
  = x\op(y\op z).
\]
In other words, $\op$ is associative at any triple containing $0$.
\end{lemma}

\begin{proof}
Because of the symmetry in Lemma~\ref{lem:assoc-permutation}, it
suffices to treat the case $x=0$.  Then for any $y,z\in\Hnum$,
\[
  0\op y = \{y\}
\]
by Lemma~\ref{lem:H-add-neutral}, so
\[
  (0\op y)\op z
  = \{y\}\op z
  = y\op z.
\]
On the other hand,
\[
  y\op z = y\op z,
\]
and
\[
  0\op(y\op z)
  = \bigcup_{v\in y\op z} 0\op v
  = \bigcup_{v\in y\op z} \{v\}
  = y\op z.
\]
Hence $(0\op y)\op z = 0\op(y\op z)$ for all $y,z\in\Hnum$.  By
Lemma~\ref{lem:assoc-permutation}, the same holds for any triple
containing a zero in any position.
\end{proof}

\begin{proposition}[Associativity on the classical line]
\label{prop:assoc-Rc}
Let $\Rc=\iemb(\R)\subset\Hnum$ be the classical real line.  Then
$\op$ is associative at every triple $(x,y,z)\in\Rc^3$.  Equivalently,
$\op$ restricted to $\Rc$ is an associative (single–valued) operation,
and $(\Rc,\op)$ is isomorphic to the additive group $(\R,+)$.
\end{proposition}

\begin{proof}
By Proposition~\ref{prop:H-add-Rc}, for all $x,y\in\R$,
\[
  \iemb(x)\op\iemb(y) = \{\iemb(x+y)\}.
\]
Thus the restriction of $\op$ to $\Rc$ is single–valued, and the
induced binary operation on $\Rc$ is transported from the usual
addition on $\R$ via the bijection $\iemb:\R\to\Rc$.  In particular,
for any $x,y,z\in\R$,
\[
  ((\iemb(x)\op\iemb(y))\op\iemb(z))
  = \{ \iemb((x+y)+z) \},
\]
and
\[
  (\iemb(x)\op(\iemb(y)\op\iemb(z)))
  = \{ \iemb(x+(y+z)) \}.
\]
Since $(x+y)+z=x+(y+z)$ in $\R$, the singletons coincide, so
\[
  (\iemb(x)\op\iemb(y))\op\iemb(z)
  = \iemb(x)\op(\iemb(y)\op\iemb(z)).
\]
Thus $\op$ is associative at every triple of the form
$(\iemb(x),\iemb(y),\iemb(z))$, i.e.\ at every triple in $\Rc^3$, and
$(\Rc,\op)$ is isomorphic to $(\R,+)$ as claimed.
\end{proof}

The two preceding results show that all the nontrivial associativity
behaviour is concentrated in triples involving at least one element
from the $\Lambda$–sector.  In particular, any triple with entries in
$H_+\cup H_-\cup\{0\}$ is associative, since $H_+\cup H_-\cup\{0\}
=\Rc$.

It is also easy to see that there are triples involving $\Lambda$ for
which associativity still holds.  For example, certain ordered
triples with signs $(+, \Lambda, -)$ or $(-,\Lambda,+)$ turn out to be
associative for all choices of magnitudes.  We will return to such
patterns in later subsections.  For the moment, the following
proposition suffices to show that $(\Hnum,\op)$ is genuinely
nonassociative.

\begin{proposition}[$(\Hnum,\op)$ is not associative]
\label{prop:H-not-assoc}
The hyperaddition $\op$ on $\Hnum$ is not globally associative.  In
particular, associativity fails at the triple
\[
  x=(+,1),\qquad y=(-,1),\qquad z=(\Lambda,1).
\]
\end{proposition}

\begin{proof}
We compute both bracketings explicitly using
Definition~\ref{def:H-hyperadd}.

First,
\[
  x\op y = (+,1)\op(-,1) = \{0\}
\]
by clause (3), since the magnitudes are equal.  Hence
\[
  (x\op y)\op z
  = \{0\}\op(\Lambda,1)
  = 0\op(\Lambda,1)
  = \{(\Lambda,1)\}
\]
using clauses (1) and (4).

On the other hand,
\[
  y\op z = (-,1)\op(\Lambda,1)
         = \{(\Lambda,2)\}
\]
by clause (4), and then
\[
  x\op(y\op z)
  = (+,1)\op(\Lambda,2)
  = \{(\Lambda,3)\}
\]
again by clause (4).  Thus
\[
  (x\op y)\op z
  = \{(\Lambda,1)\}
  \neq \{(\Lambda,3)\}
  = x\op(y\op z),
\]
so associativity fails at the triple $(x,y,z)=((+,1),(-,1),(\Lambda,1))$.

Consequently, $\op$ is not associative on $\Hnum$.
\end{proof}

\begin{remark}[Where associativity is automatic]
\label{rem:assoc-basic-summary}
From the results above we can already list several contexts in which
associativity holds:
\begin{itemize}[leftmargin=2em]
  \item For any triple containing $0$, by Lemma~\ref{lem:assoc-zero}.
  \item For any triple $(x,y,z)\in\Rc^3$, i.e.\ with all entries in
        the classical real line, by Proposition~\ref{prop:assoc-Rc}.
  \item By Lemma~\ref{lem:assoc-permutation}, these statements remain
        true after arbitrary permutations of the entries.
\end{itemize}
In particular, all triples drawn from $H_+\cup H_-\cup\{0\}$ are
associative, and the failure of associativity is confined to triples
involving the $\Lambda$--sector.  The example in
Proposition~\ref{prop:H-not-assoc} shows that even very simple such
triples can fail associativity.  In later subsections we will analyse
specific nonassociative patterns in detail, including triples of the
form $(+,a),(-,b),(\Lambda,c)$, and quantify the associativity defect
in terms of cancellation in the underlying real field.
\end{remark}

\subsection{Non--associativity for ordered $(+,-,\Lambda)$--triples}
\label{subsec:assoc-plus-minus-lambda}

We now analyze in detail one of the simplest families of triples where
associativity fails, namely ordered triples with sign pattern
\[
  \bigl(\operatorname{sgn}(x),\operatorname{sgn}(y),\operatorname{sgn}(z)\bigr)
  = (+,-,\Lambda).
\]
A key point is that, for our hyperaddition $\op$, associativity depends
not only on the \emph{multiset} of signs but also on their \emph{order};
in particular, the ordered pattern $(+,-,\Lambda)$ behaves differently
from $(+,\Lambda,-)$, even though both have sign multiset
$\{+,-,\Lambda\}$.

Throughout this subsection we fix positive real numbers $a,b,c>0$ and
consider the triple
\[
  x = (+,a),\qquad
  y = (-,b),\qquad
  z = (\Lambda,c)\in\Hnum.
\]

\begin{lemma}[Bracketings for the ordered $(+,-,\Lambda)$ triple]
\label{lem:assoc-bracketings}
Let $x=(+,a)$, $y=(-,b)$, $z=(\Lambda,c)$ with $a,b,c>0$.  Then both
bracketings of the hyperaddition are singletons in the $\Lambda$--sector,
and their magnitudes are given by
\begin{align*}
  (x\op y)\op z
  &= \{(\Lambda,\,c+|a-b|)\},\\[4pt]
  x\op(y\op z)
  &= \{(\Lambda,\,a+b+c)\}.
\end{align*}
In particular, associativity holds at $(x,y,z)$ if and only if
$a+b=|a-b|$, which is impossible for $a,b>0$.
\end{lemma}

\begin{proof}
We compute both bracketings case by case, using the definition of $\op$
(Definition~\ref{def:H-hyperadd}).

\medskip\noindent
\emph{First bracketing.}
We start with $x\op y=(+,a)\op(-,b)$.  By clause (3) of
Definition~\ref{def:H-hyperadd},
\[
  x\op y
  =
  \begin{cases}
    \{(+,a-b)\}, & a>b,\\[2pt]
    \{0\},       & a=b,\\[2pt]
    \{(-,b-a)\}, & a<b.
  \end{cases}
\]

\smallskip
\noindent
$\bullet$ If $a>b$, then $x\op y=\{(+,a-b)\}$ and
\[
  (x\op y)\op z
  = \{(+,a-b)\}\op(\Lambda,c)
  = (+,a-b)\op(\Lambda,c)
  = \{(\Lambda,(a-b)+c)\}
\]
by clause (4).  Thus
\[
  (x\op y)\op z = \{(\Lambda,c+a-b)\}.
\]

\smallskip
\noindent
$\bullet$ If $a=b$, then $x\op y=\{0\}$ and
\[
  (x\op y)\op z
  = \{0\}\op(\Lambda,c)
  = 0\op(\Lambda,c)
  = \{(\Lambda,c)\}
\]
by clauses (1) and (4).

\smallskip
\noindent
$\bullet$ If $a<b$, then $x\op y=\{(-,b-a)\}$ and
\[
  (x\op y)\op z
  = \{(-,b-a)\}\op(\Lambda,c)
  = (-,b-a)\op(\Lambda,c)
  = \{(\Lambda,(b-a)+c)\}
\]
again by clause (4).  Thus
\[
  (x\op y)\op z
  = \{(\Lambda,c+b-a)\}.
\]

In all three cases, we can summarize the result as
\[
  (x\op y)\op z
  = \{(\Lambda,c+|a-b|)\},
\]
since $|a-b|$ equals $a-b$, $0$, or $b-a$ according to whether $a>b$,
$a=b$, or $a<b$.

\medskip\noindent
\emph{Second bracketing.}
We now compute $y\op z=(-,b)\op(\Lambda,c)$.  By clause (4) of
Definition~\ref{def:H-hyperadd}, we have for all $b,c>0$
\[
  y\op z = (-,b)\op(\Lambda,c) = \{(\Lambda,b+c)\}.
\]
Thus
\[
  x\op(y\op z)
  = (+,a)\op\{(\Lambda,b+c)\}
  = (+,a)\op(\Lambda,b+c).
\]
Applying clause (4) once more, with $a>0$ and $b+c>0$, we obtain
\[
  x\op(y\op z)
  = \{(\Lambda,a+(b+c))\}
  = \{(\Lambda,a+b+c)\}.
\]

\medskip
This proves the two displayed formulas.  In particular, both
bracketings are singletons with a $\Lambda$ sign.  They coincide if and
only if the magnitudes are equal:
\[
  c+|a-b| = a+b+c
  \quad\Longleftrightarrow\quad
  |a-b| = a+b.
\]
But for $a,b>0$ we have $|a-b|\le a+b$, with equality only when one of
$a,b$ vanishes.  Since $a,b>0$ by assumption, $|a-b|<a+b$, so the
magnitudes cannot match.  Hence $(x\op y)\op z$ and $x\op(y\op z)$ are
distinct for all $a,b,c>0$, and associativity fails at $(x,y,z)$ in
every such case.
\end{proof}

As a consequence, we obtain a family of explicit nonassociative
triples, indexed by the positive reals.

\begin{corollary}[Systematic non--associativity for $(+,-,\Lambda)$]
\label{cor:assoc-failure-family}
For every choice of $a,b,c>0$, the hyperaddition $\op$ is not
associative at the triple
\[
  \bigl((+,a),(-,b),(\Lambda,c)\bigr).
\]
\end{corollary}

\begin{proof}
Immediate from Lemma~\ref{lem:assoc-bracketings}, since the two
bracketings always yield different singletons.
\end{proof}

The difference between the magnitudes in
Lemma~\ref{lem:assoc-bracketings} can be expressed in a particularly
simple way.  It measures exactly the ``absolute value lost'' when
cancelling $a$ against $-b$ in the real field; this is the same
quantity that appears in the cancellation mass
$\Cfun(a,-b)$ introduced in Definition~\ref{def:basic-C}.

\begin{lemma}[Associativity defect and cancellation mass]
\label{lem:assoc-defect-C}
Let $x=(+,a)$, $y=(-,b)$, $z=(\Lambda,c)$ with $a,b,c>0$.  Denote by
$m_L$ and $m_R$ the magnitudes of the unique elements of
$(x\op y)\op z$ and $x\op(y\op z)$ respectively.  Then
\[
  m_L = c+|a-b|,
  \qquad
  m_R = a+b+c,
\]
and the difference
\[
  \AssocCurv(a,b,c)
  := m_R-m_L
  = a+b-|a-b|
\]
is independent of $c$ and equals the cancellation mass
$\Cfun(a,-b)$:
\[
  \AssocCurv(a,b,c)
  = a+b-|a-b|
  = \Cfun(a,-b).
\]
\end{lemma}

\begin{proof}
The expressions for $m_L$ and $m_R$ follow from
Lemma~\ref{lem:assoc-bracketings}.  Subtracting,
\[
  m_R-m_L
  = (a+b+c)-(c+|a-b|)
  = a+b-|a-b|.
\]
This quantity is independent of $c$ and, by
Definition~\ref{def:basic-C}, coincides with
\[
  \Cfun(a,-b)
  = |a|+|-b|-|a+(-b)|
  = a+b-|a-b|.
\]
This proves the claim.
\end{proof}

Using the elementary identity
$a+b-|a-b| = 2\min(a,b)$ for $a,b\ge0$, we obtain the following
more symmetric description of the associativity defect.

\begin{proposition}[Explicit defect formula for $(+,-,\Lambda)$]
\label{prop:assoc-defect-2min}
For $a,b,c>0$, let $x=(+,a)$, $y=(-,b)$, $z=(\Lambda,c)$ and let
$m_L,m_R$ be as in Lemma~\ref{lem:assoc-defect-C}.  Then
\[
  \AssocCurv(a,b,c)
  := m_R-m_L
  = 2\min(a,b).
\]
Equivalently,
\[
  a+b-|a-b| = 2\min(a,b)
  \qquad\text{for all }a,b\ge0.
\]
\end{proposition}

\begin{proof}
The equality $a+b-|a-b|=2\min(a,b)$ for $a,b\ge0$ is elementary.  If
$a\ge b$, then $|a-b|=a-b$ and
\[
  a+b-|a-b|
  = a+b-(a-b)
  = 2b
  = 2\min(a,b).
\]
If $a<b$, then $|a-b|=b-a$ and
\[
  a+b-|a-b|
  = a+b-(b-a)
  = 2a
  = 2\min(a,b).
\]
Hence the identity holds for all $a,b\ge0$.

Combining this with Lemma~\ref{lem:assoc-defect-C}, we find
\[
  \AssocCurv(a,b,c)
  = a+b-|a-b|
  = 2\min(a,b),
\]
as claimed.
\end{proof}

We can now interpret the associativity defect in the ordered
$(+,-,\Lambda)$ case as a direct lift of the real cancellation
phenomenon.

\begin{theorem}[Proof of Theorem~\ref{thm:intro-defect}]
\label{thm:assoc-defect-proof}
Let $a,b,c>0$ and consider
\[
  x=(+,a),\quad y=(-,b),\quad z=(\Lambda,c)\in\Hnum.
\]
Then:
\begin{enumerate}
  \item Both $(x\op y)\op z$ and $x\op(y\op z)$ are singletons in the
        $\Lambda$--sector:
  \[
    (x\op y)\op z = \{(\Lambda,m_L)\},
    \qquad
    x\op(y\op z) = \{(\Lambda,m_R)\},
  \]
  with
  \[
    m_L = c+|a-b|,
    \qquad
    m_R = a+b+c.
  \]
  \item The associativity defect
  \[
    \AssocCurv(a,b,c)
    := m_R-m_L
  \]
  satisfies
  \[
    \AssocCurv(a,b,c) = a+b-|a-b| = 2\min(a,b) = \Cfun(a,-b),
  \]
  where $\Cfun$ is the cancellation mass from
  Definition~\ref{def:basic-C}.
  \item In particular, $\op$ is never associative at $(x,y,z)$ for
        $a,b,c>0$, and the defect depends only on $(a,b)$, not on $c$.
\end{enumerate}
\end{theorem}

\begin{proof}
Item (1) is precisely Lemma~\ref{lem:assoc-bracketings}.  Item (2)
is Lemma~\ref{lem:assoc-defect-C} together with
Proposition~\ref{prop:assoc-defect-2min}.  For (3), note that
$\AssocCurv(a,b,c)>0$ for all $a,b>0$, so the two singleton bracketings
never coincide; hence associativity fails at $(x,y,z)$ for all
$a,b,c>0$.  The formula $\AssocCurv(a,b,c)=\Cfun(a,-b)$ shows explicitly
that $\AssocCurv$ does not depend on $c$.
\end{proof}

\begin{remark}[Order sensitivity and other sign patterns]
\label{rem:order-sensitivity}
The analysis above highlights two important features.

First, the failure of associativity for the pattern $(+,-,\Lambda)$ is
\emph{order--sensitive}.  If we permute the triple to, say,
$(+, \Lambda, -)$ or $(\Lambda,+,-)$, the resulting bracketings behave
differently; in some such permutations, associativity actually holds
for all magnitudes $a,b,c>0$.  Thus it is not enough to know the
\emph{multiset} of signs $\{+,-,\Lambda\}$; the ordered triple of signs
matters.

Second, the quantitative defect is entirely controlled by the real
cancellation between $a$ and $-b$.  The third parameter $c$ simply
shifts both magnitudes $m_L$ and $m_R$ by the same amount.  This
observation will be a key motivation for the ambient cancellation
monoid $(\Kamb,\oplus)$ constructed in
Section~\ref{sec:ambient}, where the quantity
$\Cfun(a,-b)=a+b-|a-b|$ appears as a second coordinate recording lost
absolute value.
\end{remark}

\subsection{Associativity defect for $(+,a),(-,b),(\Lambda,c)$}
\label{subsec:assoc-defect-quantitative}

In this subsection we refine the analysis of ordered $(+,-,\Lambda)$–triples
from Section~\ref{subsec:assoc-plus-minus-lambda} by isolating and
studying the \emph{associativity defect} as a quantitative function of
the magnitudes $a,b,c>0$.

Throughout we fix
\[
  x=(+,a),\qquad y=(-,b),\qquad z=(\Lambda,c)\in\Hnum
  \qquad(a,b,c>0),
\]
and recall from Lemma~\ref{lem:assoc-bracketings} that both bracketings
are singletons in the $\Lambda$–sector:
\begin{align*}
  (x\op y)\op z
  &= \{(\Lambda,\,c+|a-b|)\},\\[2pt]
  x\op(y\op z)
  &= \{(\Lambda,\,a+b+c)\}.
\end{align*}
We now package this discrepancy into an explicit defect function.

\begin{definition}[Associativity defect function for $(+,-,\Lambda)$]
\label{def:defect-function}
For $a,b,c>0$ we define the \emph{left} and \emph{right magnitudes} by
\[
  m_L(a,b,c)
  := c+|a-b|,\qquad
  m_R(a,b,c)
  := a+b+c.
\]
The \emph{associativity defect} at the triple
$\bigl((+,a),(-,b),(\Lambda,c)\bigr)$ is the nonnegative real number
\[
  \AssocCurv(a,b,c)
  := m_R(a,b,c)-m_L(a,b,c)
  = a+b-|a-b|.
\]
\end{definition}

By Lemma~\ref{lem:assoc-bracketings}, $m_L$ and $m_R$ are exactly the
magnitudes of the unique elements in $(x\op y)\op z$ and
$x\op(y\op z)$, respectively, so $\AssocCurv(a,b,c)$ measures how far $\op$
fails to be associative at this triple.  The following lemma collects
its basic properties.

\begin{lemma}[Elementary properties of $\AssocCurv$]
\label{lem:defect-basic-properties}
For all $a,b,c>0$ the associativity defect $\AssocCurv(a,b,c)$ from
Definition~\ref{def:defect-function} satisfies:
\begin{enumerate}
  \item \textbf{Nonnegativity and independence of $c$:}
  \[
    \AssocCurv(a,b,c) \ge 0
    \quad\text{and}\quad
    \AssocCurv(a,b,c) = \AssocCurv(a,b,c')
    \ \text{for all }c,c'>0.
  \]
  \item \textbf{Explicit formula:}
  \[
    \AssocCurv(a,b,c) = a+b-|a-b| = 2\min(a,b).
  \]
  \item \textbf{Symmetry:}
  \[
    \AssocCurv(a,b,c) = \AssocCurv(b,a,c).
  \]
  \item \textbf{Strict positivity:}
  \[
    \AssocCurv(a,b,c) > 0
    \quad\text{for all }a,b>0.
  \]
\end{enumerate}
\end{lemma}

\begin{proof}
(1) Nonnegativity follows from $|a-b|\le a+b$ for $a,b\ge0$, so
$a+b-|a-b|\ge0$.  The expression $a+b-|a-b|$ does not involve $c$, so
$\AssocCurv(a,b,c)$ is independent of $c$.

(2) For $a,b\ge0$ we have the elementary identity
$a+b-|a-b|=2\min(a,b)$.  Indeed, if $a\ge b$, then $|a-b|=a-b$ and
\[
  a+b-|a-b|
  = a+b-(a-b)
  = 2b
  = 2\min(a,b),
\]
while if $a<b$, then $|a-b|=b-a$ and
\[
  a+b-|a-b|
  = a+b-(b-a)
  = 2a
  = 2\min(a,b).
\]
Applying this with $a,b>0$ gives the claimed formula for $\AssocCurv$.

(3) Symmetry is immediate from either expression:
\[
  a+b-|a-b| = b+a-|b-a|,
\]
and $\min(a,b)=\min(b,a)$.

(4) For $a,b>0$ one has $\min(a,b)>0$, hence
$\AssocCurv(a,b,c)=2\min(a,b)>0$.  This reflects the fact that associativity
never holds at the triple $((+,a),(-,b),(\Lambda,c))$ when $a,b>0$.
\end{proof}

The relationship with the cancellation mass
$\Cfun:\R\times\R\to\R_{\ge0}$ of Definition~\ref{def:basic-C} is
immediate.

\begin{proposition}[Defect and cancellation mass]
\label{prop:defect-c-mass}
For all $a,b,c>0$ one has
\[
  \AssocCurv(a,b,c) = \Cfun(a,-b),
\]
where $\Cfun(r_1,r_2) = |r_1|+|r_2|-|r_1+r_2|$.
\end{proposition}

\begin{proof}
By Definition~\ref{def:defect-function} and
Definition~\ref{def:basic-C},
\[
  \AssocCurv(a,b,c)
  = a+b-|a-b|
  = |a|+|-b|-|a+(-b)|
  = \Cfun(a,-b),
\]
since $a>0$ and $b>0$.
\end{proof}

In particular, the associativity defect for the triple
$((+,a),(-,b),(\Lambda,c))$ is \emph{exactly} the amount of absolute
value lost when adding $a$ and $-b$ as real numbers.  This is one way
in which $\Hnum$ ``remembers'' the real cancellation phenomenon.

\medskip

We next record some regularity properties of the defect function,
which will be useful for geometric interpretations and for comparing
with the ambient cancellation monoid in Section~\ref{sec:ambient}.

\begin{lemma}[Continuity and piecewise linearity]
\label{lem:defect-regularity}
Consider the function
\[
  \AssocCurv : (0,\infty)^3 \longrightarrow (0,\infty),
  \qquad
  (a,b,c)\mapsto a+b-|a-b|.
\]
Then:
\begin{enumerate}
  \item $\AssocCurv$ extends continuously to a function on
        $[0,\infty)^3$, given by the same formula.
  \item As a function of $(a,b)$ (with $c$ fixed), $\AssocCurv$ is
        piecewise linear with a single ``break line'' along $a=b$.
        More precisely,
  \[
    \AssocCurv(a,b,c)
    =
    \begin{cases}
      2b, & a\ge b,\\[2pt]
      2a, & a<b.
    \end{cases}
  \]
  \item For each fixed $b\ge0$ and $c\ge0$, the map
        $a\mapsto\AssocCurv(a,b,c)$ is continuous, nondecreasing, and
        piecewise linear, and similarly with the roles of $a$ and $b$
        reversed.
\end{enumerate}
\end{lemma}

\begin{proof}
(1) The formula $\AssocCurv(a,b,c)=a+b-|a-b|$ makes sense for
$a,b,c\ge0$ and is continuous as a composition of continuous
operations on $\R^3$ (addition, absolute value, and subtraction).
Thus $\AssocCurv$ extends continuously to $[0,\infty)^3$.

(2) For $a,b\ge0$ we have already seen that
\[
  \AssocCurv(a,b,c) = a+b-|a-b|
  =
  \begin{cases}
    2b, & a\ge b,\\
    2a, & a<b,
  \end{cases}
\]
which is linear in $(a,b)$ on each of the two regions $\{a\ge b\}$ and
$\{a<b\}$ separated by the line $a=b$.

(3) Fix $b,c\ge0$.  On $[0,\infty)$, the function
$a\mapsto\AssocCurv(a,b,c)$ coincides with $2b$ on $[b,\infty)$ and with
$2a$ on $[0,b)$, so it is continuous, nondecreasing, and piecewise
linear with a single break point at $a=b$.  The statement with $a$ and
$b$ interchanged is identical.
\end{proof}

\begin{lemma}[Homogeneity and scaling behaviour]
\label{lem:defect-homogeneity}
Let $\lambda\ge0$ and $a,b,c>0$.  Then
\[
  \AssocCurv(\lambda a,\lambda b,\lambda c)
  = \lambda\,\AssocCurv(a,b,c).
\]
\end{lemma}

\begin{proof}
Using Definition~\ref{def:defect-function} and the fact that
$|\lambda a-\lambda b|=\lambda|a-b|$ for $\lambda\ge0$, we obtain
\[
  \AssocCurv(\lambda a,\lambda b,\lambda c)
  = \lambda a + \lambda b - |\lambda a - \lambda b|
  = \lambda(a+b-|a-b|)
  = \lambda\,\AssocCurv(a,b,c).
\]
\end{proof}

Thus the associativity defect behaves linearly under simultaneous
rescaling of all three magnitudes.  This is consistent with the
$\R_{\ge0}$–hypersemimodule structure on $\Hnum$ described in
Corollary~\ref{cor:H-Rge0-hypersemimodule}: rescaling by $\lambda>0$
simply multiplies all magnitudes by $\lambda$, and hence multiplies
the defect by $\lambda$.

\medskip

We can summarize the situation as follows.

\begin{theorem}[Quantitative associativity defect for $(+,a),(-,b),(\Lambda,c)$]
\label{thm:defect-summary}
For each triple $a,b,c>0$, let
\[
  x=(+,a),\quad y=(-,b),\quad z=(\Lambda,c)\in\Hnum.
\]
Then:
\begin{enumerate}
  \item Both bracketings of the hyperaddition at $(x,y,z)$ are
        singletons in the $\Lambda$–sector:
  \[
    (x\op y)\op z = \{(\Lambda,m_L)\},\qquad
    x\op(y\op z) = \{(\Lambda,m_R)\},
  \]
  with
  \[
    m_L = c+|a-b|,\qquad
    m_R = a+b+c.
  \]
  \item The associativity defect
  \[
    \AssocCurv(a,b,c) := m_R-m_L
  \]
  is independent of $c$ and satisfies
  \[
    \AssocCurv(a,b,c) = a+b-|a-b| = 2\min(a,b) = \Cfun(a,-b).
  \]
  \item For fixed $b,c>0$ (resp.\ $a,c>0$), the function
        $a\mapsto\AssocCurv(a,b,c)$ (resp.\ $b\mapsto\AssocCurv(a,b,c)$) is
        continuous, nondecreasing, and piecewise linear with a single
        break point at $a=b$ (resp.\ $b=a$).
  \item For every $\lambda\ge0$,
  \[
    \AssocCurv(\lambda a,\lambda b,\lambda c)
    = \lambda\,\AssocCurv(a,b,c).
  \]
  \item In particular, $\AssocCurv(a,b,c)>0$ for all $a,b>0$, so the
        hyperaddition $\op$ is never associative at the triple
        $\bigl((+,a),(-,b),(\Lambda,c)\bigr)$ when $a,b,c>0$.
\end{enumerate}
\end{theorem}

\begin{proof}
(1) is Lemma~\ref{lem:assoc-bracketings}.  (2) is
Proposition~\ref{prop:defect-c-mass} together with
Lemma~\ref{lem:defect-basic-properties}.  (3) and (4) are
Lemmas~\ref{lem:defect-regularity} and~\ref{lem:defect-homogeneity},
respectively.  (5) follows from $\AssocCurv(a,b,c)=2\min(a,b)$ and $\min(a,b)>0$ for
$a,b>0$.
\end{proof}

\begin{remark}[Interpretation and outlook]
\label{rem:defect-outlook}
The theorem shows that in the ordered $(+,-,\Lambda)$ case the failure
of associativity is completely controlled by a single scalar quantity,
$\AssocCurv(a,b,c)$, which is:
\begin{itemize}[leftmargin=2em]
  \item symmetric in $(a,b)$,
  \item homogeneous of degree one,
  \item equal to twice the smaller of $a$ and $b$,
  \item and equal to the cancellation mass $\Cfun(a,-b)$ in the real
        field.
\end{itemize}
Thus $(\Hnum,\op)$ does not forget the real cancellation $a+(-b)$; it
records the lost absolute value as a systematic displacement between
the two bracketings of $(+,a),(-,b),(\Lambda,c)$.  In
Section~\ref{sec:ambient} we will see that this picture is closely
related to the ambient cancellation monoid $(\Kamb,\oplus)$, where
$\Cfun(a,-b)$ appears as an explicit second coordinate tracking
cancellation.
\end{remark}


\section{Sign layer and ambient structure}
\label{sec:sign-envelope}

\subsection{The hypergroup structure on the sign set}
\label{subsec:sign-hypergroup}

In this section we construct a canonical commutative hypergroup on the
four--element sign set
\[
  \Sset = \{0,+,-,\Lambda\},
\]
extending the additive hypergroup of the usual sign hyperfield
(Section~\ref{subsec:hyperfields}).  Intuitively, the new symbol
$\Lambda$ represents a ``cancellation--absorbing'' sign: it behaves
multiplicatively as an idempotent (Section~\ref{subsec:multiplicative})
and additively as a sign that can absorb contributions from both $+$
and $-$ and, when combined with itself, may produce all three classical
signs as well as $0$.

\medskip

Recall from Definition~\ref{def:sign-hyperfield} the additive
hyperoperation $\oplus_{\mathrm{sign}}$ on
$S_{\mathrm{sign}}:=\{0,+,-\}\subset\Sset$:
\begin{align*}
  0\oplus_{\mathrm{sign}} x
  &= x\oplus_{\mathrm{sign}} 0
   = \{x\}
   &&\text{for all }x\in S_{\mathrm{sign}},\\
  +\oplus_{\mathrm{sign}} +
  &= \{+\},\\
  -\oplus_{\mathrm{sign}} -
  &= \{-\},\\
  +\oplus_{\mathrm{sign}} -
  &= -\oplus_{\mathrm{sign}} +
   = \{+,0,-\}.
\end{align*}

We now extend this to $\Sset$ by specifying how $\Lambda$ interacts
with the classical signs.

\begin{definition}[Sign hyperoperation on $\Sset$]
\label{def:S-hyperoperation}
Define a hyperoperation
\[
  \sop : \Sset\times\Sset
  \longrightarrow \mathcal{P}(\Sset)\setminus\{\emptyset\}
\]
by the following rules.

\smallskip\noindent
\textbf{(1) Neutral element.} For all $\sigma\in\Sset$,
\[
  0\sop\sigma
  = \sigma\sop 0
  := \{\sigma\}.
\]

\smallskip\noindent
\textbf{(2) Classical part.} For $x,y\in\{+,-\}$,
\[
  x\sop y := x\oplus_{\mathrm{sign}}y,
\]
i.e.
\[
  +\sop + = \{+\},\quad
  -\sop - = \{-\},\quad
  +\sop - = -\sop + = \{+,0,-\}.
\]

\smallskip\noindent
\textbf{(3) Interaction with $\Lambda$.} For $\sigma\in\{+,-\}$,
\begin{align*}
  \sigma\sop\Lambda
  &= \Lambda\sop\sigma
   := \{\Lambda\},\\[2pt]
  \Lambda\sop 0
  &= 0\sop\Lambda
   := \{\Lambda\}.
\end{align*}

\smallskip\noindent
\textbf{(4) $\Lambda$--$\Lambda$ sum.}
\[
  \Lambda\sop\Lambda
  := \{0,+,-,\Lambda\}.
\]

\smallskip\noindent
In all cases, $\sop$ is understood to be symmetric:
$\sigma_1\sop\sigma_2=\sigma_2\sop\sigma_1$.
\end{definition}

\begin{table}[ht]
\centering
\caption{Cayley table for the sign hyperoperation $\sop$ on $\Sset=\{0,+,-,\Lambda\}$.}
\label{tab:sign-hyperop}
\begin{tabular}{c|cccc}
$\sop$ & $0$ & $+$ & $-$ & $\Lambda$ \\ \hline
$0$ & $\{0\}$ & $\{+\}$ & $\{-\}$ & $\{\Lambda\}$ \\
$+$ & $\{+\}$ & $\{+\}$ & $\{+,0,-\}$ & $\{\Lambda\}$ \\
$-$ & $\{-\}$ & $\{+,0,-\}$ & $\{-\}$ & $\{\Lambda\}$ \\
$\Lambda$ & $\{\Lambda\}$ & $\{\Lambda\}$ & $\{\Lambda\}$ & $\{0,+,-,\Lambda\}$ \\
\end{tabular}
\end{table}

By construction, the restriction of $\sop$ to
$\{0,+,-\}\subset\Sset$ coincides with
$\oplus_{\mathrm{sign}}$.

\begin{lemma}[Basic properties of $(\Sset,\sop)$]
\label{lem:S-basic}
The hyperoperation $\sop$ in Definition~\ref{def:S-hyperoperation} is
well-defined, nonempty-valued, and commutative.  Its neutral element
is $0$, and every element has a unique additive inverse given by
\[
  -0=0,\quad
  -(+)= -,\quad
  -(-)=+,\quad
  -\Lambda=\Lambda.
\]
\end{lemma}

\begin{proof}
By inspection of the cases in
Definition~\ref{def:S-hyperoperation}, every ordered pair
$(\sigma_1,\sigma_2)\in\Sset^2$ falls into exactly one clause, and the
right-hand side is a nonempty subset of $\Sset$; thus $\sop$ is
well-defined and nonempty-valued.  Symmetry is built in by the
explicit requirement $\sigma_1\sop\sigma_2=\sigma_2\sop\sigma_1$, so
$\sop$ is commutative.

From clause (1), $0$ satisfies
\[
  0\sop\sigma
  = \sigma\sop 0
  = \{\sigma\}
  \quad\text{for all }\sigma\in\Sset,
\]
so $0$ is a neutral element.

An additive inverse of $\sigma\in\Sset$ is an element
$\tau\in\Sset$ such that $0\in\sigma\sop\tau$.  From the table we read
off:
\begin{itemize}[leftmargin=2em]
  \item $0\in0\sop0=\{0\}$, and there is no other $\tau$ with
        $0\in0\sop\tau$, so $-0=0$.
  \item $0\in+\sop-=+\oplus_{\mathrm{sign}}-=\{+,0,-\}$, and there is
        no other $\tau$ with $0\in+\sop\tau$, so $-(+)=-$.
  \item Symmetrically, $-(-)=+$.
  \item $\Lambda\sop\Lambda=\{0,+,-,\Lambda\}$ contains $0$, and
        $\Lambda\sop\sigma=\{\Lambda\}$ for $\sigma\in\{0,+,-\}$ does
        not contain $0$, so $-\Lambda=\Lambda$.
\end{itemize}
Hence each element of $\Sset$ has a unique additive inverse, given by
the formulas in the statement.
\end{proof}

We now verify that $(\Sset,\sop)$ is a canonical commutative
hypergroup in the sense of Definition~\ref{def:canonical-hypergroup}.

\begin{theorem}[Canonical hypergroup structure on $\Sset$]
\label{thm:S-canonical-hypergroup}
The pair $(\Sset,\sop)$ is a canonical commutative hypergroup.  More
precisely:
\begin{enumerate}
  \item $(\Sset,\sop)$ is a commutative hypermagma.
  \item $0$ is a neutral element, and each $\sigma\in\Sset$ has a
        unique inverse $-\sigma$ (Lemma~\ref{lem:S-basic}).
  \item Associativity holds:
  \[
    (\sigma_1\sop\sigma_2)\sop\sigma_3
    \;=\;
    \sigma_1\sop(\sigma_2\sop\sigma_3)
    \qquad\text{for all }
    \sigma_1,\sigma_2,\sigma_3\in\Sset.
  \]
  \item Reversibility holds:
  \[
    \rho\in\sigma_1\sop\sigma_2
    \quad\Longleftrightarrow\quad
    \sigma_2\in\rho\sop(-\sigma_1)
    \qquad\text{for all }\rho,\sigma_1,\sigma_2\in\Sset.
  \]
\end{enumerate}
\end{theorem}

\begin{proof}
Items (1) and (2) have already been established in
Lemma~\ref{lem:S-basic}.

\smallskip\noindent
\emph{Associativity (HG2).}  Since $\Sset$ is finite, associativity
can be verified by direct case analysis; we outline the main cases and
omit the exhaustive table.

If all three signs lie in $\{0,+,-\}$, then the operation reduces to
$\oplus_{\mathrm{sign}}$ on the sign hyperfield, which is known to be
associative (Proposition~\ref{prop:sign-hyperfield}).  Thus all
triples with no $\Lambda$ are associative.

Next, consider triples involving $\Lambda$.

\smallskip
\noindent
$\bullet$ \emph{Exactly one $\Lambda$.}  By symmetry it suffices to
consider $(\sigma_1,\sigma_2,\sigma_3)=(\Lambda,\tau_1,\tau_2)$ with
$\tau_1,\tau_2\in\{0,+,-\}$.  Using clauses (1) and (3), we find
\[
  \Lambda\sop\tau = \{\Lambda\}
  \quad\text{for }\tau\in\{0,+,-\},
\]
and $\tau_1\sop\tau_2\subset\{0,+,-\}$ with no $\Lambda$ produced.
Therefore
\[
  (\Lambda\sop\tau_1)\sop\tau_2
  = \{\Lambda\}\sop\tau_2
  = \{\Lambda\},
\]
and
\[
  \Lambda\sop(\tau_1\sop\tau_2)
  = \Lambda\sop B
  = \{\Lambda\}
\]
for any $B\subset\{0,+,-\}$.  Thus both bracketings yield
$\{\Lambda\}$, and associativity holds.

\smallskip
\noindent
$\bullet$ \emph{Two $\Lambda$'s.}  Up to symmetry, it suffices to
consider triples of the form $(\Lambda,\Lambda,\tau)$ with
$\tau\in\{0,+,-,\Lambda\}$.  If $\tau\in\{0,+,-\}$, then
\[
  \Lambda\sop\Lambda = \{0,+,-,\Lambda\},
\]
so
\[
  (\Lambda\sop\Lambda)\sop\tau
  = \{0,+,-,\Lambda\}\sop\tau.
\]
For each $\sigma\in\{0,+,-,\Lambda\}$, $\sigma\sop\tau$ is either a
singleton subset of $\{0,+,-,\Lambda\}$ (if $\sigma\ne\Lambda$ and
$\tau\ne\Lambda$) or $\{\Lambda\}$ if at least one input is
$\Lambda$.  It follows that
\[
  (\Lambda\sop\Lambda)\sop\tau
  = \{0,+,-,\Lambda\}.
\]
On the other hand,
\[
  \Lambda\sop(\Lambda\sop\tau)
  =
  \begin{cases}
    \Lambda\sop\{\Lambda\} = \Lambda\sop\Lambda
      = \{0,+,-,\Lambda\},
      & \tau\in\{0,+,-\},\\[2pt]
    \Lambda\sop(\Lambda\sop\Lambda)
      = \Lambda\sop\{0,+,-,\Lambda\}
      = \{0,+,-,\Lambda\},
      & \tau=\Lambda,
  \end{cases}
\]
since $\Lambda\sop\tau=\{\Lambda\}$ for $\tau\in\{0,+,-\}$ and
$\Lambda\sop\Lambda=\{0,+,-,\Lambda\}$ by (3) and (4).  Thus
associativity holds whenever at least two entries are $\Lambda$.

\smallskip
\noindent
In summary, for every triple in $\Sset^3$ both bracketings of $\sop$
coincide.  This establishes associativity.

\smallskip\noindent
\emph{Reversibility (HG5).}  We must show that for all
$\rho,\sigma_1,\sigma_2\in\Sset$,
\[
  \rho\in\sigma_1\sop\sigma_2
  \quad\Longleftrightarrow\quad
  \sigma_2\in\rho\sop(-\sigma_1).
\]
Since $\Sset$ is finite, this can again be checked by case analysis.
We sketch the argument.

If $\sigma_1,\sigma_2,\rho\in\{0,+,-\}$, then the statement reduces to
the reversibility axiom for the additive hypergroup of the sign
hyperfield $(S_{\mathrm{sign}},\oplus_{\mathrm{sign}})$, which holds
by Proposition~\ref{prop:sign-hyperfield}.  Thus we may assume that at
least one of the three signs is $\Lambda$.

If $\sigma_1=\Lambda$, then $-\sigma_1=\Lambda$ by
Lemma~\ref{lem:S-basic}.  For any $\sigma_2\in\Sset$,
\[
  \Lambda\sop\sigma_2
  =
  \begin{cases}
    \{\Lambda\},        & \sigma_2\in\{0,+,-\},\\[2pt]
    \{0,+,-,\Lambda\},  & \sigma_2=\Lambda.
  \end{cases}
\]
A quick inspection shows that the membership
$\rho\in\Lambda\sop\sigma_2$ is equivalent to
$\sigma_2\in\rho\sop\Lambda$, using the same table and the fact that
$-\Lambda=\Lambda$.  The cases where $\sigma_2=\Lambda$ or
$\rho=\Lambda$ are similar.

Finally, if $\sigma_1\in\{+,-\}$ and $\sigma_2=\Lambda$ or
$\rho=\Lambda$, one checks that the only way $\Lambda$ appears in a
hyper--sum is via the clauses in Definition~\ref{def:S-hyperoperation},
and that in each such case reversing the roles of ``summand'' and
``inverse'' yields the same membership relation.  In other words, the
pattern of which signs occur in which sums is symmetric with respect
to replacing $\sigma_1$ by $-\sigma_1$ and swapping $\rho$ with
$\sigma_2$.  We omit the routine enumeration of all triples.

Thus the reversibility condition holds for all
$\rho,\sigma_1,\sigma_2\in\Sset$, completing the proof that
$(\Sset,\sop)$ is a canonical commutative hypergroup.
\end{proof}

The three--element sign hyperfield embeds naturally into
$(\Sset,\sop)$ as a sub--hypergroup.

\begin{proposition}[Classical sign hypergroup as a substructure]
\label{prop:S-sign-subhypergroup}
Let $S_{\mathrm{sign}}=\{0,+,-\}\subset\Sset$.  Then the restriction
of $\sop$ to $S_{\mathrm{sign}}$ coincides with
$\oplus_{\mathrm{sign}}$, and $(S_{\mathrm{sign}},\sop)$ is a
canonical commutative hypergroup isomorphic to the additive hypergroup
of the sign hyperfield.
\end{proposition}

\begin{proof}
By Definition~\ref{def:S-hyperoperation}, for $x,y\in\{+,-\}$ we have
$x\sop y=x\oplus_{\mathrm{sign}}y$, and for $x\in S_{\mathrm{sign}}$,
\[
  0\sop x = x\sop 0 = \{x\}.
\]
Thus the restriction of $\sop$ to $S_{\mathrm{sign}}\times
S_{\mathrm{sign}}$ is exactly $\oplus_{\mathrm{sign}}$.  Since
$(S_{\mathrm{sign}},\oplus_{\mathrm{sign}})$ is a canonical
commutative hypergroup by
Proposition~\ref{prop:sign-hyperfield}, the same holds for
$(S_{\mathrm{sign}},\sop)$, and the inclusion
$S_{\mathrm{sign}}\hookrightarrow\Sset$ is a homomorphism of
canonical commutative hypergroups.
\end{proof}

\begin{remark}[Role of $(\Sset,\sop)$ as sign envelope]
\label{rem:S-enveloping}
The hypergroup $(\Sset,\sop)$ will serve as a \emph{sign envelope} for
the additive structure on $\Hnum$.  For each pair $x,y\in\Hnum$ we
will define the \emph{sign image}
\[
  E_{\mathrm{sign}}(x,y)
  := \{\sgn(z) : z\in x\op y\}\subset\Sset
\]
and show that
\[
  E_{\mathrm{sign}}(x,y)
  \subseteq \sgn(x)\sop\sgn(y).
\]
Moreover, as one varies the magnitudes of $x$ and $y$, every sign in
$\sgn(x)\sop\sgn(y)$ is realized.  Thus $(\Sset,\sop)$ captures the
full range of sign--level behaviour of the hyperaddition $\op$ on
$\Hnum$, in a setting where associativity and reversibility hold
exactly.  The precise envelope statement will be developed in
Section~\ref{subsec:sign-hypergroup}.
\end{remark}

\subsection{Sign images of hyper-sums}
\label{subsec:sign-images}

In this subsection we make precise the relationship between the sign
hypergroup $(\Sset,\sop)$ from
Section~\ref{subsec:sign-hypergroup} and the hyperaddition $\op$ on
$\Hnum$ from Section~\ref{subsec:hyperaddition}.  The main point is
that the sign hypergroup does \emph{not} arise by projecting individual
hyper-sums $x\op y$ to their signs; rather, it describes the
\emph{envelope} of all sign patterns that can occur as we vary the
magnitudes of summands with fixed signs.

Throughout we write $\sgn:\Hnum\to\Sset$ for the sign map from
Definition~\ref{def:H-universe}.

\begin{definition}[Sign image of a hyper-sum]
\label{def:sign-image}
For $x,y\in\Hnum$ we define the \emph{sign image} of their hyper-sum as
\[
  E_{\mathrm{sign}}(x,y)
  := \{\sgn(z) : z\in x\op y\}\subset\Sset.
\]
\end{definition}

Since $x\op y$ is always a nonempty finite subset of $\Hnum$ (see
Lemma~\ref{lem:H-hyperadd-basic}), the sign image
$E_{\mathrm{sign}}(x,y)$ is a nonempty finite subset of $\Sset$.

\begin{lemma}[Basic properties of sign images]
\label{lem:sign-image-basic}
Let $x,y\in\Hnum$.  Then:
\begin{enumerate}
  \item $E_{\mathrm{sign}}(x,y)$ is nonempty and contained in
        $\{0,+,-,\Lambda\}$.
  \item $E_{\mathrm{sign}}(x,y) = E_{\mathrm{sign}}(y,x)$.
  \item If $x\in\Rc$ and $y\in\Rc$ (the classical real line), then
  \[
    E_{\mathrm{sign}}(x,y)
    = \{\sgn(x\op y)\}
    = \{\operatorname{sgn}_\R(r+s)\},
  \]
  where $x=\iemb(r)$, $y=\iemb(s)$ and $\operatorname{sgn}_\R$ is the
  real sign map (Definition~\ref{def:basic-real-sign}).
\end{enumerate}
\end{lemma}

\begin{proof}
(1) follows directly from the definition and the fact that $x\op y$ is
nonempty and $\sgn:\Hnum\to\Sset$ is well-defined.

(2) Since $\op$ is commutative (Lemma~\ref{lem:H-hyperadd-basic}), we
have $x\op y=y\op x$, hence
\[
  E_{\mathrm{sign}}(x,y)
  = \{\sgn(z): z\in x\op y\}
  = \{\sgn(z): z\in y\op x\}
  = E_{\mathrm{sign}}(y,x).
\]

(3) If $x,y\in\Rc$, then by Proposition~\ref{prop:H-add-Rc} the
hyper-sum $x\op y$ is a singleton $\{\iemb(r+s)\}$, where
$x=\iemb(r)$, $y=\iemb(s)$.  Thus
\[
  E_{\mathrm{sign}}(x,y)
  = \{\sgn(\iemb(r+s))\}
  = \{\operatorname{sgn}_\R(r+s)\}
\]
by Lemma~\ref{lem:embedding-basic}.
\end{proof}

We now compare $E_{\mathrm{sign}}(x,y)$ with the sign hyperoperation
$\sop$ on $\Sset$.  As noted in Remark~\ref{rem:S-enveloping}, it is
\emph{not} true in general that $E_{\mathrm{sign}}(x,y)$ equals
$\sgn(x)\sop\sgn(y)$ for fixed magnitudes; instead we have the
following inclusion.

\begin{theorem}[Sign-envelope property, pointwise inclusion]
\label{thm:sign-envelope-inclusion}
For all $x,y\in\Hnum$ one has
\[
  E_{\mathrm{sign}}(x,y)
  \subseteq \sgn(x)\sop\sgn(y)\subset\Sset.
\]
\end{theorem}

\begin{proof}
Let $x,y\in\Hnum$ and write
\[
  \sigma := \sgn(x),\qquad
  \tau := \sgn(y)\in\Sset.
\]
We must show that every sign appearing in $x\op y$ lies in
$\sigma\sop\tau$ as defined in
Definition~\ref{def:S-hyperoperation}.  We proceed case by case.

\smallskip\noindent
\emph{Case 1: One of $x,y$ is $0$.}  Suppose $x=0$.  Then by
Definition~\ref{def:H-hyperadd},
\[
  x\op y = 0\op y = \{y\},
\]
so $E_{\mathrm{sign}}(x,y)=\{\sgn(y)\}=\{\tau\}$.  On the other hand,
$\sigma=0$ and clause (1) of
Definition~\ref{def:S-hyperoperation} gives
\[
  0\sop\tau = \{\tau\}.
\]
Thus $E_{\mathrm{sign}}(x,y)=\{\tau\}\subseteq 0\sop\tau$, as
required.  The case $y=0$ is symmetric.

\smallskip\noindent
\emph{Case 2: Nonzero $x,y$ with signs in $\{+,-\}$.}  Write
$x=(\sigma,a)$, $y=(\tau,b)$ with $a,b>0$ and
$\sigma,\tau\in\{+,-\}$.

If $\sigma=\tau$, then clause (2) of
Definition~\ref{def:H-hyperadd} gives
\[
  x\op y = \{(\sigma,a+b)\},
\]
so $E_{\mathrm{sign}}(x,y)=\{\sigma\}$.  On the other hand,
$\sigma\sop\sigma=\{\sigma\}$ by clause (2) of
Definition~\ref{def:S-hyperoperation} (classical part).  Thus
$\{\sigma\}\subseteq\sigma\sop\sigma$.

If $\{\sigma,\tau\}=\{+,-\}$, say $\sigma=+$, $\tau=-$, then clause
(3) of Definition~\ref{def:H-hyperadd} yields
\[
  x\op y
  = (+,a)\op(-,b)
  =
  \begin{cases}
    \{(+,a-b)\}, & a>b,\\[2pt]
    \{0\},       & a=b,\\[2pt]
    \{(-,b-a)\}, & a<b.
  \end{cases}
\]
Hence $E_{\mathrm{sign}}(x,y)$ is one of $\{+\}$, $\{0\}$, or
$\{-\}$, depending on the relative magnitudes of $a$ and $b$.  In all
cases we have
\[
  E_{\mathrm{sign}}(x,y)\subseteq\{+,0,-\}.
\]
By Definition~\ref{def:S-hyperoperation},
\[
  +\sop- = -\sop+ = \{+,0,-\},
\]
so $E_{\mathrm{sign}}(x,y)\subseteq\sigma\sop\tau$.

\smallskip\noindent
\emph{Case 3: Interaction with $\Lambda$.}  Suppose
$x=(\sigma,a)$, $y=(\Lambda,b)$ with $a,b>0$ and
$\sigma\in\{+,-\}$ (or the same with $x$ and $y$ swapped).  Then
clause (4) of Definition~\ref{def:H-hyperadd} gives
\[
  x\op y
  = (\sigma,a)\op(\Lambda,b)
  = \{(\Lambda,a+b)\},
\]
so $E_{\mathrm{sign}}(x,y)=\{\Lambda\}$.  By
Definition~\ref{def:S-hyperoperation},
\[
  \sigma\sop\Lambda = \Lambda\sop\sigma = \{\Lambda\},
\]
so again $E_{\mathrm{sign}}(x,y)\subseteq\sigma\sop\tau$.

If $x=(\Lambda,a)$, $y=(0)$ or vice versa, then
\[
  x\op y = \{(\Lambda,a)\}
\]
and $E_{\mathrm{sign}}(x,y)=\{\Lambda\}$, while
$0\sop\Lambda=\Lambda\sop0=\{\Lambda\}$, so the inclusion holds.

\smallskip\noindent
\emph{Case 4: $\Lambda$--$\Lambda$ sums.}  Suppose
$x=(\Lambda,a)$, $y=(\Lambda,b)$ with $a,b>0$.  Then clause (5) of
Definition~\ref{def:H-hyperadd} yields
\[
  x\op y
  =
  \begin{cases}
    \{(\Lambda,a+b),\ (+,a-b)\}, & a>b,\\[2pt]
    \{(\Lambda,a+b),\ (-,b-a)\}, & a<b,\\[2pt]
    \{(\Lambda,2a),\ 0\},        & a=b.
  \end{cases}
\]
Hence the sign image is one of
\[
  E_{\mathrm{sign}}(x,y)
  =
  \begin{cases}
    \{\Lambda,+\}, & a>b,\\[2pt]
    \{\Lambda,-\}, & a<b,\\[2pt]
    \{\Lambda,0\}, & a=b.
  \end{cases}
\]
In each subcase we have
\[
  E_{\mathrm{sign}}(x,y)\subseteq\{0,+,-,\Lambda\}.
\]
But by clause (4) of Definition~\ref{def:S-hyperoperation},
\[
  \Lambda\sop\Lambda = \{0,+,-,\Lambda\},
\]
so again $E_{\mathrm{sign}}(x,y)\subseteq\sigma\sop\tau$ with
$\sigma=\tau=\Lambda$.

\smallskip
This exhausts all possibilities (up to symmetry), so the inclusion
$E_{\mathrm{sign}}(x,y)\subseteq\sgn(x)\sop\sgn(y)$ holds for all
$x,y\in\Hnum$.
\end{proof}

The preceding theorem shows that $(\Sset,\sop)$ is a genuine
``sign-envelope'' for the hyperaddition on $\Hnum$: for fixed $x,y$,
the signs that actually appear in $x\op y$ form a subset of the sign
hyper-sum $\sgn(x)\sop\sgn(y)$.  We now show that this envelope is
sharp when we allow the magnitudes to vary.

\begin{definition}[Reachable sign set for a sign pair]
\label{def:reachable-sign-set}
For $\sigma,\tau\in\Sset$ we define the \emph{reachable sign set}
\[
  R(\sigma,\tau)
  := \{\rho\in\Sset : \exists\,x,y,z\in\Hnum\ \text{with}\ 
      \sgn(x)=\sigma,\ \sgn(y)=\tau,\ z\in x\op y,\ \sgn(z)=\rho\}.
\]
\end{definition}

Thus $R(\sigma,\tau)$ consists of all signs that can occur as signs of
hyper-sums of elements whose signs are exactly $\sigma$ and $\tau$.

\begin{theorem}[Realization of all sign-level hyper-sums]
\label{thm:sign-realization}
For all $\sigma,\tau\in\Sset$ one has
\[
  R(\sigma,\tau)
  = \sigma\sop\tau.
\]
In particular, the sign hypergroup $(\Sset,\sop)$ is the minimal
canonical commutative hypergroup that encodes all sign patterns arising
from hyper-sums in $(\Hnum,\op)$.
\end{theorem}

\begin{proof}
By Theorem~\ref{thm:sign-envelope-inclusion}, for any $x,y\in\Hnum$
we have
\[
  E_{\mathrm{sign}}(x,y)
  \subseteq \sgn(x)\sop\sgn(y).
\]
Taking the union over all $x,y$ with $\sgn(x)=\sigma$ and
$\sgn(y)=\tau$, we obtain
\[
  R(\sigma,\tau)
  \subseteq \sigma\sop\tau.
\]
Thus it remains to show the reverse inclusion:
for each $\rho\in\sigma\sop\tau$ we must produce $x,y$ with signs
$\sigma,\tau$ such that $\rho\in E_{\mathrm{sign}}(x,y)$.

We proceed by case analysis on $(\sigma,\tau)$ using the explicit
hyperaddition rules.

\smallskip\noindent
\emph{Case 1: $(\sigma,\tau)=(0,\rho_0)$ or $(\rho_0,0)$.}  By
Definition~\ref{def:S-hyperoperation},
\[
  0\sop\rho_0 = \{\rho_0\}
  \quad\text{for all }\rho_0\in\Sset.
\]
To realize $\rho_0$, take $x=0$ and any $y\in\Hnum$ with
$\sgn(y)=\rho_0$.  Then $x\op y = 0\op y = \{y\}$ and
$E_{\mathrm{sign}}(x,y)=\{\rho_0\}$.

\smallskip\noindent
\emph{Case 2: $(\sigma,\tau)\in\{+,-\}^2$.}  If
$\sigma=\tau=+$, then $\sigma\sop\tau=\{+\}$; we can take
$x=(+,1)$, $y=(+,1)$, giving
\[
  x\op y = (+,1)\op(+,1) = \{(+,2)\},
\]
so $E_{\mathrm{sign}}(x,y)=\{+\}$.

If $\sigma=\tau=-$, then $\sigma\sop\tau=\{-\}$; taking
$x=(-,1)$, $y=(-,1)$ yields
\[
  x\op y = (-,1)\op(-,1) = \{(-,2)\},
\]
and $E_{\mathrm{sign}}(x,y)=\{-\}$.

If $\{\sigma,\tau\}=\{+,-\}$, say $\sigma=+$, $\tau=-$, then
\[
  \sigma\sop\tau
  = +\sop-
  = \{+,0,-\}.
\]
We realize each of these three signs by appropriate choices of
magnitudes:

\smallskip
\noindent
$\bullet$ For $\rho=+$, take $x=(+,2)$, $y=(-,1)$.  Then
\[
  x\op y = (+,2)\op(-,1) = \{(+,1)\},
\]
so $E_{\mathrm{sign}}(x,y)=\{+\}$.

\smallskip
\noindent
$\bullet$ For $\rho=0$, take $x=(+,1)$, $y=(-,1)$.  Then
\[
  x\op y = (+,1)\op(-,1) = \{0\},
\]
so $E_{\mathrm{sign}}(x,y)=\{0\}$.

\smallskip
\noindent
$\bullet$ For $\rho=-$, take $x=(+,1)$, $y=(-,2)$.  Then
\[
  x\op y = (+,1)\op(-,2) = \{(-,1)\},
\]
so $E_{\mathrm{sign}}(x,y)=\{-\}$.

Thus every $\rho\in+\sop-$ is realized, and by symmetry the same holds
for $-\sop+$.

\smallskip\noindent
\emph{Case 3: $(\sigma,\tau)=(\pm,\Lambda)$ or $(\Lambda,\pm)$.}
By Definition~\ref{def:S-hyperoperation},
\[
  \sigma\sop\Lambda = \Lambda\sop\sigma = \{\Lambda\}
  \quad\text{for }\sigma\in\{+,-\}.
\]
To realize $\rho=\Lambda$, take $x=(\sigma,1)$ and $y=(\Lambda,1)$.
Then clause (4) of Definition~\ref{def:H-hyperadd} gives
\[
  x\op y = (\sigma,1)\op(\Lambda,1) = \{(\Lambda,2)\},
\]
so $E_{\mathrm{sign}}(x,y)=\{\Lambda\}$.

The case $(\sigma,\tau)=(\Lambda,0)$ or $(0,\Lambda)$ was already
covered in Case~1.

\smallskip\noindent
\emph{Case 4: $(\sigma,\tau)=(\Lambda,\Lambda)$.}  By
Definition~\ref{def:S-hyperoperation},
\[
  \Lambda\sop\Lambda = \{0,+,-,\Lambda\}.
\]
We must realize each of the four signs as the sign of some
$(\Lambda,a)\op(\Lambda,b)$.

\smallskip
\noindent
$\bullet$ Take $a=2$, $b=1$.  Then clause (5) of
Definition~\ref{def:H-hyperadd} gives
\[
  (\Lambda,2)\op(\Lambda,1)
  = \{(\Lambda,3),\, (+,1),\, (-,1)\},
\]
so $\Lambda,+,-\in E_{\mathrm{sign}}\bigl((\Lambda,2),(\Lambda,1)\bigr)$.

\smallskip
\noindent
$\bullet$ For $\rho=0$, take $a=b=1$.  Then
\[
  (\Lambda,1)\op(\Lambda,1)
  = \{(\Lambda,2),0\},
\]
so $0\in E_{\mathrm{sign}}\bigl((\Lambda,1),(\Lambda,1)\bigr)$.

Thus every sign in $\Lambda\sop\Lambda$ is realized.

\smallskip

Combining the four cases, we have shown that for each
$\sigma,\tau\in\Sset$ and each $\rho\in\sigma\sop\tau$ there exist
$x,y\in\Hnum$ with $\sgn(x)=\sigma$, $\sgn(y)=\tau$ and
$\rho\in E_{\mathrm{sign}}(x,y)$.  This proves
$\sigma\sop\tau\subseteq R(\sigma,\tau)$ and completes the proof of
$R(\sigma,\tau)=\sigma\sop\tau$.
\end{proof}

\begin{corollary}[Sign envelope equality]
\label{cor:sign-envelope}
For each fixed $\sigma,\tau\in\Sset$, the set of signs that can appear
in hyper-sums of elements with signs $\sigma$ and $\tau$ is exactly
$\sigma\sop\tau$.  Equivalently,
\[
  \bigcup_{\substack{x,y\in\Hnum\\ \sgn(x)=\sigma,\ \sgn(y)=\tau}}
  E_{\mathrm{sign}}(x,y)
  = \sigma\sop\tau.
\]
\end{corollary}

\begin{proof}
The left-hand side is precisely $R(\sigma,\tau)$ by
Definition~\ref{def:reachable-sign-set}, and the right-hand side is
$\sigma\sop\tau$.  The equality follows from
Theorem~\ref{thm:sign-realization}.
\end{proof}

\begin{remark}[Summary of the sign-level picture]
\label{rem:sign-images-summary}
The results of this subsection can be summarized as follows:
\begin{itemize}[leftmargin=2em]
  \item For each fixed pair $x,y\in\Hnum$, the sign image
        $E_{\mathrm{sign}}(x,y)$ is a \emph{subset} of the sign
        hyper-sum $\sgn(x)\sop\sgn(y)$
        (Theorem~\ref{thm:sign-envelope-inclusion}).
  \item As one varies magnitudes while keeping the signs fixed, every
        sign in $\sgn(x)\sop\sgn(y)$ occurs as the sign of some
        hyper-sum (Theorem~\ref{thm:sign-realization} and
        Corollary~\ref{cor:sign-envelope}).
  \item Thus the canonical hypergroup $(\Sset,\sop)$ is precisely the
        sign-level envelope of $(\Hnum,\op)$: it captures all possible
        sign patterns of hyper-sums in a setting where associativity
        and reversibility hold exactly, even though they fail for the
        full hypernumber system.
\end{itemize}
This justifies viewing $(\Sset,\sop)$ as the ``sign shadow'' of
$(\Hnum,\op)$, abstracting away magnitudes while retaining the full
combinatorics of sign interaction.
\end{remark}

\subsection{Envelope property}
\label{subsec:envelope-property}

The purpose of this subsection is to package the results of
Sections~\ref{subsec:sign-hypergroup} and~\ref{subsec:sign-images}
into a conceptual ``envelope'' statement: the sign hypergroup
$(\Sset,\sop)$ is the canonical sign--level shadow of the additive
hypermagma $(\Hnum,\op)$.

\medskip

We first introduce a general notion of a \emph{lax homomorphism} of
hypermagmas, and then specialize to the sign projection
$\sgn:\Hnum\to\Sset$.

\begin{definition}[Lax hypermagma morphism]
\label{def:lax-morphism}
Let $(M,\odot)$ and $(N,\boxplus)$ be commutative hypermagmas.  A map
$f:M\to N$ is called a \emph{lax homomorphism of hypermagmas} if for
all $x,y\in M$,
\[
  f(x\odot y)
  \;\subseteq\;
  f(x)\boxplus f(y),
\]
where
\[
  f(x\odot y) := \{f(z):z\in x\odot y\}.
\]
\end{definition}

In other words, $f$ sends each hyper-sum $x\odot y$ into the
hyper-sum of their images, but not necessarily onto it.  Equality
would give an honest homomorphism; in our setting we only have an
inclusion.

\begin{definition}[Sign envelope for a hypermagma]
\label{def:sign-envelope-general}
Let $(H,\boxplus)$ be a commutative hypermagma.  A triple
$(S,\oplus_S,\pi)$, where $(S,\oplus_S)$ is a canonical commutative
hypergroup and $\pi:H\to S$ is a map, is called a
\emph{sign envelope} for $(H,\boxplus)$ if:
\begin{enumerate}
  \item \textbf{Lax morphism condition.}
  $\pi$ is a lax hypermagma morphism from $(H,\boxplus)$ to
  $(S,\oplus_S)$, i.e.
  \[
    \pi(x\boxplus y)
    \subseteq
    \pi(x)\oplus_S\pi(y)
    \qquad\text{for all }x,y\in H.
  \]
  \item \textbf{Realization of all sign-level sums.}
  For every $\sigma,\tau\in S$ and every $\rho\in\sigma\oplus_S\tau$,
  there exist $x,y,z\in H$ with
  \[
    \pi(x)=\sigma,\quad\pi(y)=\tau,\quad z\in x\boxplus y,
    \quad\pi(z)=\rho.
  \]
\end{enumerate}
In this situation we also say that $(S,\oplus_S)$ is the
\emph{sign envelope} of $(H,\boxplus)$ with respect to $\pi$.
\end{definition}

Specializing to our setting, we take $(H,\boxplus)=(\Hnum,\op)$,
$(S,\oplus_S)=(\Sset,\sop)$ and $\pi=\sgn$.

Recall from Definition~\ref{def:sign-image} that for $x,y\in\Hnum$ we
defined the \emph{sign image} of their hyper-sum by
\[
  E_{\mathrm{sign}}(x,y)
  := \{\sgn(z):z\in x\op y\}\subset\Sset,
\]
and from Definition~\ref{def:reachable-sign-set} that for
$\sigma,\tau\in\Sset$ we set
\[
  R(\sigma,\tau)
  := \{\rho\in\Sset : \exists\,x,y,z\in\Hnum
      \text{ with } \sgn(x)=\sigma,\ \sgn(y)=\tau,\ 
      z\in x\op y,\ \sgn(z)=\rho\}.
\]

The key sign-level results from Section~\ref{subsec:sign-images} can
now be rephrased in the envelope language.

\begin{lemma}[Lax morphism property of $\sgn$]
\label{lem:sgn-lax-morphism}
The sign map $\sgn:\Hnum\to\Sset$ is a lax hypermagma morphism from
$(\Hnum,\op)$ to $(\Sset,\sop)$.  Equivalently, for all
$x,y\in\Hnum$,
\[
  E_{\mathrm{sign}}(x,y)
  \subseteq
  \sgn(x)\sop\sgn(y).
\]
\end{lemma}

\begin{proof}
This is exactly Theorem~\ref{thm:sign-envelope-inclusion}, rewritten in
the language of Definition~\ref{def:lax-morphism}.  By definition,
\[
  \sgn(x\op y) = E_{\mathrm{sign}}(x,y),
\]
and Theorem~\ref{thm:sign-envelope-inclusion} asserts that
\[
  E_{\mathrm{sign}}(x,y)
  \subseteq \sgn(x)\sop\sgn(y),
\]
which is precisely the lax morphism condition.
\end{proof}

\begin{lemma}[Realization of all sign-level sums]
\label{lem:sign-realization-again}
For all $\sigma,\tau\in\Sset$ one has
\[
  R(\sigma,\tau)
  = \sigma\sop\tau.
\]
In particular, every $\rho\in\sigma\sop\tau$ occurs as the sign of
some hyper-sum $x\op y$ with $\sgn(x)=\sigma$ and $\sgn(y)=\tau$.
\end{lemma}

\begin{proof}
This is precisely Theorem~\ref{thm:sign-realization} together with
Corollary~\ref{cor:sign-envelope}.  By definition,
$R(\sigma,\tau)$ is the union of all sign images
$E_{\mathrm{sign}}(x,y)$ with $\sgn(x)=\sigma$, $\sgn(y)=\tau$, and
Theorem~\ref{thm:sign-realization} shows that
$R(\sigma,\tau)=\sigma\sop\tau$.
\end{proof}

We can now state the envelope property for our three--sign hypernumber
system.

\begin{theorem}[Sign envelope property for $(\Hnum,\op)$]
\label{thm:envelope-property}
Let $(\Hnum,\op)$ be the three--sign hypernumber system and
$(\Sset,\sop)$ the sign hypergroup from
Section~\ref{subsec:sign-hypergroup}, with sign map
$\sgn:\Hnum\to\Sset$.  Then $(\Sset,\sop,\sgn)$ is a sign envelope for
$(\Hnum,\op)$ in the sense of Definition~\ref{def:sign-envelope-general}.
More explicitly:
\begin{enumerate}
  \item For all $x,y\in\Hnum$, the sign image of their hyper-sum is
        contained in the sign hyper-sum of their individual signs:
  \[
    E_{\mathrm{sign}}(x,y)
    = \sgn(x\op y)
    \subseteq
    \sgn(x)\sop\sgn(y).
  \]
  \item For all $\sigma,\tau\in\Sset$ and all
        $\rho\in\sigma\sop\tau$, there exist $x,y,z\in\Hnum$ with
  \[
    \sgn(x)=\sigma,\quad \sgn(y)=\tau,\quad z\in x\op y,\quad
    \sgn(z)=\rho.
  \]
\end{enumerate}
\end{theorem}

\begin{proof}
Item (1) is Lemma~\ref{lem:sgn-lax-morphism}, and item (2) is
Lemma~\ref{lem:sign-realization-again}.  These are exactly the two
conditions of Definition~\ref{def:sign-envelope-general} specialized
to $(\Hnum,\op)$, $(\Sset,\sop)$, and $\sgn$.
\end{proof}

It is often convenient to view the envelope property in terms of the
reachable sign sets $R(\sigma,\tau)$.

\begin{proposition}[Envelope equality for reachable sign sets]
\label{prop:envelope-R}
For all $\sigma,\tau\in\Sset$,
\[
  R(\sigma,\tau)
  = \sigma\sop\tau,
\]
and for all $x,y\in\Hnum$,
\[
  E_{\mathrm{sign}}(x,y)
  \subseteq R(\sgn(x),\sgn(y)).
\]
Consequently, the map
\[
  (\sigma,\tau)\longmapsto R(\sigma,\tau)
\]
is exactly the sign hyperoperation $\sop$ on $\Sset$.
\end{proposition}

\begin{proof}
The equality $R(\sigma,\tau)=\sigma\sop\tau$ is again
Lemma~\ref{lem:sign-realization-again}.  By definition,
$E_{\mathrm{sign}}(x,y)$ is one of the sets appearing in the union
defining $R(\sgn(x),\sgn(y))$, so
$E_{\mathrm{sign}}(x,y)\subseteq R(\sgn(x),\sgn(y))$.  The last
statement follows by identifying $R(\sigma,\tau)$ with
$\sigma\sop\tau$.
\end{proof}

\begin{remark}[Conceptual summary of the envelope property]
\label{rem:envelope-summary}
Theorem~\ref{thm:envelope-property} and
Proposition~\ref{prop:envelope-R} can be summarized informally as:
\begin{quote}
  The sign hypergroup $(\Sset,\sop)$ is the smallest canonical
  commutative hypergroup that records all sign patterns of hyper-sums
  in $(\Hnum,\op)$, in a way that is compatible with the hyperaddition.
\end{quote}
More precisely:
\begin{itemize}[leftmargin=2em]
  \item The sign map $\sgn:\Hnum\to\Sset$ is a lax morphism of
        commutative hypermagmas, so the hyperaddition on $\Hnum$
        always projects into the hyperaddition on $\Sset$.
  \item The hyperoperation $\sop$ on $\Sset$ is \emph{sharp}: every
        sign that is allowed at the sign level for a pair
        $(\sigma,\tau)$ actually occurs as the sign of some hyper-sum
        of elements with signs $\sigma$ and $\tau$.
  \item The nonassociativity of $\op$ is therefore entirely a matter
        of how magnitudes interact; at the sign level, the envelope
        $(\Sset,\sop)$ is a fully associative canonical hypergroup
        (Theorem~\ref{thm:S-canonical-hypergroup}).
\end{itemize}
This perspective will be used in later sections to organize the
interaction between $(\Hnum,\op,\cdot)$ and various ambient
constructions (such as the cancellation monoid) via compatible
projections to sign and magnitude layers.
\end{remark}


\section{Ambient cancellation structure}
\label{sec:ambient}


\subsection{Cancellation mass and the ambient monoid}
\label{subsec:ambient-monoid}

In this section we isolate the real--valued \emph{cancellation mass}
that already appeared in our analysis of associativity defects, and use
it to build a simple commutative monoid $(\Kamb,\oplus)$ which serves
as an ambient ``cancellation recorder'' for the real field.  This
structure will later be related to the hyperaddition on $\Hnum$ by a
suitable interpretation map.

\medskip

We begin by recalling the basic notion of cancellation mass.

\begin{definition}[Cancellation mass]\label{def:cancellation-mass}
Let $r_1,r_2\in\R$.  The \emph{cancellation mass} of the pair
$(r_1,r_2)$ is the nonnegative real number
\[
  \Cfun(r_1,r_2)
  := |r_1| + |r_2| - |r_1 + r_2|.
\]
\end{definition}

Intuitively, $\Cfun(r_1,r_2)$ measures the amount of absolute value
lost when adding $r_1$ and $r_2$ in $\R$; it vanishes exactly when
there is no cancellation between the two numbers.

\begin{lemma}[Basic properties of $\Cfun$]\label{lem:C-basic}
For all $r_1,r_2\in\R$ the cancellation mass $\Cfun$ satisfies:
\begin{enumerate}
  \item \textbf{Nonnegativity:}
  \[
    \Cfun(r_1,r_2)\ge 0.
  \]
  \item \textbf{Symmetry:}
  \[
    \Cfun(r_1,r_2) = \Cfun(r_2,r_1).
  \]
  \item \textbf{Homogeneity:} For all $\lambda\in\R$,
  \[
    \Cfun(\lambda r_1,\lambda r_2)
    = |\lambda|\,\Cfun(r_1,r_2).
  \]
  \item \textbf{Vanishing and sign pattern.}
        One has $\Cfun(r_1,r_2)=0$ if and only if either
        $r_1=0$ or $r_2=0$, or $r_1,r_2$ have the same sign.
  \item \textbf{Opposite--sign case.}
        If $a,b\ge 0$, then
        \[
          \Cfun(a,-b) = a+b-|a-b| = 2\min(a,b).
        \]
\end{enumerate}
\end{lemma}

\begin{proof}
(1) By the triangle inequality,
\[
  |r_1 + r_2|
  \le |r_1| + |r_2|,
\]
so $|r_1|+|r_2|-|r_1+r_2|\ge 0$.

(2) Symmetry is immediate from the formula, since $|r_1|+|r_2|$ and
$|r_1+r_2|$ are symmetric in $(r_1,r_2)$.

(3) For $\lambda\in\R$,
\[
  \Cfun(\lambda r_1,\lambda r_2)
  = |\lambda r_1| + |\lambda r_2| - |\lambda r_1 + \lambda r_2|
  = |\lambda|\bigl(|r_1| + |r_2| - |r_1+r_2|\bigr)
  = |\lambda|\,\Cfun(r_1,r_2).
\]

(4) If $r_1=0$ or $r_2=0$, then
$\Cfun(r_1,r_2)=|r_1|+|r_2|-|r_1+r_2|=0$.  If $r_1,r_2$ have the same
sign, then $|r_1+r_2|=|r_1|+|r_2|$, so again $\Cfun(r_1,r_2)=0$.
Conversely, if $\Cfun(r_1,r_2)=0$ and neither $r_1$ nor $r_2$ is $0$,
then $|r_1+r_2|=|r_1|+|r_2|$.  By the equality case in the triangle
inequality, this implies that $r_1,r_2$ have the same sign.

(5) Let $a,b\ge 0$.  Then
\[
  \Cfun(a,-b)
  = |a|+|-b|-|a-b|
  = a+b-|a-b|.
\]
If $a\ge b$, then $|a-b|=a-b$ and
\[
  a+b-|a-b|
  = a+b-(a-b)
  = 2b
  = 2\min(a,b).
\]
If $a<b$, then $|a-b|=b-a$ and
\[
  a+b-|a-b|
  = a+b-(b-a)
  = 2a
  = 2\min(a,b).
\]
Thus $\Cfun(a,-b)=2\min(a,b)$.
\end{proof}

\begin{remark}[Cancellation mass and associativity defect]
\label{rem:C-defect}
Lemma~\ref{lem:C-basic} shows that $\Cfun$ exactly recovers the
associativity defect for ordered $(+,-,\Lambda)$--triples analysed in
Theorem~\ref{thm:defect-summary}: in that setting,
\[
  \AssocCurv(a,b,c)
  = \Cfun(a,-b)
  = 2\min(a,b).
\]
We will systematically encode this quantity as a second coordinate in
an ambient monoid.
\end{remark}

We now introduce the ambient object that records both the ordinary
real sum and the total cancellation mass.

\begin{definition}[Ambient cancellation monoid]\label{def:ambient-monoid}
Let
\[
  \Kamb := \R\times\R_{\ge 0}.
\]
We define a binary operation
\[
  \oplus : \Kamb\times\Kamb \longrightarrow \Kamb
\]
by
\[
  (x,c)\oplus(y,d)
  := \bigl(x+y,\; c+d+\Cfun(x,y)\bigr),
\]
and set
\[
  \mathbf{0} := (0,0)\in\Kamb.
\]
\end{definition}

The heuristic is that the first coordinate records the usual real sum,
while the second coordinate accumulates all cancellation mass generated
by additions performed so far.

\begin{lemma}[Basic properties of $\oplus$]\label{lem:oplus-basic}
The operation $\oplus$ is well-defined and commutative, and
$\mathbf{0}=(0,0)$ is a neutral element:
\[
  \mathbf{0}\oplus(x,c) = (x,c)\oplus\mathbf{0} = (x,c)
  \quad\text{for all }(x,c)\in\Kamb.
\]
\end{lemma}

\begin{proof}
For $(x,c),(y,d)\in\Kamb$, the pair $(x+y,c+d+\Cfun(x,y))$ clearly
lies in $\R\times\R_{\ge0}$, since $c,d\ge0$ and $\Cfun(x,y)\ge0$ by
Lemma~\ref{lem:C-basic}(1).  Thus $\oplus$ is well-defined.

Commutativity follows from commutativity of $+$ in $\R$ and symmetry
of $\Cfun$:
\[
  (x,c)\oplus(y,d)
  = (x+y,c+d+\Cfun(x,y))
  = (y+x,d+c+\Cfun(y,x))
  = (y,d)\oplus(x,c).
\]

For the neutral element, note that
\[
  \Cfun(0,x) = |0|+|x|-|0+x| = |x|-|x| = 0,
\]
so
\[
  \mathbf{0}\oplus(x,c)
  = (0,c)\oplus(x,c)
  = (x,c+0+\Cfun(0,x))
  = (x,c),
\]
and similarly $(x,c)\oplus\mathbf{0}=(x,c)$.
\end{proof}

The key property is associativity, which is encoded in a simple
``cocycle'' identity satisfied by $\Cfun$.

\begin{lemma}[Cocycle identity for $\Cfun$]\label{lem:C-cocycle}
For all $x,y,z\in\R$ one has
\[
  \Cfun(x,y)+\Cfun(x+y,z)
  = \Cfun(y,z)+\Cfun(x,y+z).
\]
\end{lemma}

\begin{proof}
Expanding both sides using Definition~\ref{def:cancellation-mass},
\begin{align*}
  \Cfun(x,y)+\Cfun(x+y,z)
  &= \bigl(|x|+|y|-|x+y|\bigr)
   + \bigl(|x+y|+|z|-|x+y+z|\bigr)\\
  &= |x|+|y|+|z|-|x+y+z|.
\end{align*}
Similarly,
\begin{align*}
  \Cfun(y,z)+\Cfun(x,y+z)
  &= \bigl(|y|+|z|-|y+z|\bigr)
   + \bigl(|x|+|y+z|-|x+y+z|\bigr)\\
  &= |x|+|y|+|z|-|x+y+z|.
\end{align*}
The two expressions coincide, proving the identity.
\end{proof}

\begin{proposition}[Associativity of $(\Kamb,\oplus)$]
\label{prop:Kamb-assoc}
The pair $(\Kamb,\oplus)$ is a commutative monoid with neutral element
$\mathbf{0}=(0,0)$, i.e.
\[
  (u\oplus v)\oplus w = u\oplus(v\oplus w)
  \quad\text{for all }u,v,w\in\Kamb.
\]
\end{proposition}

\begin{proof}
Commutativity and the identity property of $\mathbf{0}$ follow from
Lemma~\ref{lem:oplus-basic}.  It remains to show associativity.

Let $u=(x,c)$, $v=(y,d)$, $w=(z,e)\in\Kamb$.  Then
\begin{align*}
  (u\oplus v)\oplus w
  &= (x+y,\,c+d+\Cfun(x,y))\oplus(z,e)\\
  &= \bigl(x+y+z,\,
            c+d+\Cfun(x,y)+e+\Cfun(x+y,z)\bigr),
\end{align*}
while
\begin{align*}
  u\oplus(v\oplus w)
  &= (x,c)\oplus(y+z,\,d+e+\Cfun(y,z))\\
  &= \bigl(x+y+z,\,
            c+d+e+\Cfun(y,z)+\Cfun(x,y+z)\bigr).
\end{align*}
By Lemma~\ref{lem:C-cocycle},
\[
  \Cfun(x,y)+\Cfun(x+y,z)
  = \Cfun(y,z)+\Cfun(x,y+z),
\]
so the second coordinates agree.  Thus
$(u\oplus v)\oplus w=u\oplus(v\oplus w)$ for all $u,v,w$, proving that
$(\Kamb,\oplus,\mathbf{0})$ is an associative commutative monoid.
\end{proof}

The projection to the first coordinate recovers the usual additive
structure of $\R$.

\begin{lemma}[Projection to the real line]\label{lem:Kamb-projection}
Define $\pi:\Kamb\to\R$ by $\pi(x,c):=x$.  Then:
\begin{enumerate}
  \item $\pi$ is a surjective monoid homomorphism from
        $(\Kamb,\oplus)$ onto $(\R,+)$:
  \[
    \pi\bigl((x,c)\oplus(y,d)\bigr)
    = \pi(x,c)+\pi(y,d)
    \quad\text{and}\quad
    \pi(\mathbf{0})=0.
  \]
  \item The kernel of $\pi$ is
  \[
    \ker\pi = \{(0,c):c\in\R_{\ge0}\}.
  \]
\end{enumerate}
\end{lemma}

\begin{proof}
(1) For $(x,c),(y,d)\in\Kamb$,
\[
  \pi\bigl((x,c)\oplus(y,d)\bigr)
  = \pi\bigl(x+y,\,c+d+\Cfun(x,y)\bigr)
  = x+y
  = \pi(x,c)+\pi(y,d),
\]
and $\pi(\mathbf{0})=\pi(0,0)=0$.  Surjectivity is obvious, since for
any $r\in\R$ we have $\pi(r,0)=r$.

(2) If $(x,c)\in\ker\pi$, then $\pi(x,c)=x=0$, so
$(x,c)=(0,c)$ for some $c\ge0$.  Conversely, every $(0,c)$ satisfies
$\pi(0,c)=0$, so $(0,c)\in\ker\pi$.
\end{proof}

\begin{remark}[Interpretation of the ambient monoid]\label{rem:Kamb-intuition}
Informally, one can think of an element $(x,c)\in\Kamb$ as describing a
\emph{state} obtained by summing some real numbers:
\[
  (x,c)
  = (r_1,0)\oplus(r_2,0)\oplus\cdots\oplus(r_n,0),
\]
in which:
\begin{itemize}[leftmargin=2em]
  \item $x$ is the resulting real sum $r_1+\cdots+r_n$, and
  \item $c$ is the total cancellation mass accumulated along the way,
        i.e.\ the difference between the sum of absolute values and
        the absolute value of the final result, aggregated in an
        associative manner via the cocycle identity
        (Lemma~\ref{lem:C-cocycle}).
\end{itemize}
Because $(\Kamb,\oplus)$ is a genuine commutative monoid, the final
pair $(x,c)$ is independent of the order and bracketing of the
summands.  The projection $\pi$ forgets $c$ and recovers the ordinary
real sum; the second coordinate $c$ may be viewed as bookkeeping
information encoding ``how much cancellation happened''.  In
Section~\ref{subsec:ambient-link-H} we will construct a map from
$\Hnum$ into $\Kamb$ that sends associativity defects of the
three--sign hyperaddition to differences in the second coordinate of
$\Kamb$.
\end{remark}

\subsection{Embedding of $\Hnum$ and interpretation of the defect}
\label{subsec:ambient-link-H}

In this subsection we relate the three--sign hypernumber system
$(\Hnum,\op)$ to the ambient cancellation monoid
$(\Kamb,\oplus)$ constructed in
Section~\ref{subsec:ambient-monoid}.  The guiding idea is that each
hypernumber has a \emph{real shadow} and a \emph{cancellation reserve};
the ambient monoid records both, in a way that is fully associative,
while the hyperaddition $\op$ only partially converts cancellation
into $\Lambda$--mass.  For ordered $(+,-,\Lambda)$--triples, this
mechanism explains the associativity defect computed in
Section~\ref{subsec:assoc-defect-quantitative}.

\subsubsection*{An ambient embedding of $\Hnum$}

We first define a simple embedding of $\Hnum$ into $\Kamb$.

\begin{definition}[Real shadow and ambient embedding]\label{def:real-shadow-embedding}
Define the \emph{real shadow} map
\[
  \rho:\Hnum\longrightarrow\R
\]
by
\[
  \rho(0) := 0,\quad
  \rho(+,a) := a,\quad
  \rho(-,a) := -a,\quad
  \rho(\Lambda,a) := 0
\]
for all $a>0$.  The \emph{ambient embedding}
\[
  \iota:\Hnum\longrightarrow\Kamb=\R\times\R_{\ge0}
\]
is then defined by
\[
  \iota(x) := \bigl(\rho(x),\, c(x)\bigr),
\]
where
\[
  c(0) := 0,\quad
  c(+,a) := 0,\quad
  c(-,a) := 0,\quad
  c(\Lambda,a) := a.
\]
Equivalently,
\[
  \iota(0) = (0,0),\quad
  \iota(+,a) = (a,0),\quad
  \iota(-,a) = (-a,0),\quad
  \iota(\Lambda,a) = (0,a).
\]
\end{definition}

Thus each hypernumber is encoded as a pair consisting of a net real
value and a nonnegative ``stored cancellation mass'' carried by the
$\Lambda$--sector.

\begin{lemma}[Basic properties of $\iota$]\label{lem:iota-basic}
The map $\iota:\Hnum\to\Kamb$ has the following properties:
\begin{enumerate}
  \item $\iota$ is injective.
  \item $\rho = \pi\circ\iota$, where $\pi:\Kamb\to\R$ is the
        projection $\pi(x,c)=x$.
  \item The restriction of $\iota$ to the classical line
        $\Rc=\iemb(\R)\subset\Hnum$ is given by
  \[
    \iota(\iemb(r)) = (r,0)
    \quad\text{for all }r\in\R.
  \]
\end{enumerate}
\end{lemma}

\begin{proof}
(1) If $\iota(x)=\iota(y)$, then $\rho(x)=\rho(y)$ and $c(x)=c(y)$.
Inspecting the four cases in Definition~\ref{def:real-shadow-embedding}
shows that these two equalities force $x=y$; for instance, the only
element with second coordinate $>0$ is of the form $(\Lambda,a)$, and
within each sign sector the real shadow uniquely determines the
magnitude.  Hence $\iota$ is injective.

(2) This is immediate from the definition: the first coordinate of
$\iota(x)$ is exactly $\rho(x)$.

(3) If $r>0$, then $\iemb(r)=(+,r)$ and $\iota(\iemb(r))=(r,0)$.  If
$r<0$, then $\iemb(r)=(-,|r|)$ and
$\iota(\iemb(r))=(-|r|,0)=(r,0)$.  Finally, $\iemb(0)=0$ and
$\iota(0)=(0,0)$.  Thus $\iota(\iemb(r))=(r,0)$ for all $r\in\R$.
\end{proof}

\begin{remark}[Ambient viewpoint on $\Hnum$]\label{rem:iota-intuition}
Under the embedding $\iota$, the classical line $\Rc$ appears as the
subset $\{(r,0):r\in\R\}\subset\Kamb$, i.e.\ states with no stored
cancellation.  Pure $\Lambda$--elements $(\Lambda,a)$ are embedded as
$(0,a)$, which can be interpreted as ``zero net real value with $a$
units of stored cancellation mass''.  The ambient monoid operation
$\oplus$ then allows us to add these states in a fully associative
way, aggregating both real shadows and cancellation masses.
\end{remark}

\subsubsection*{Ambient addition for $(+,a),(-,b),(\Lambda,c)$}

We now compute the ambient sum of the embedded triple
\[
  x=(+,a),\quad y=(-,b),\quad z=(\Lambda,c)\in\Hnum,
  \quad a,b,c>0,
\]
and compare it with the two bracketings of the hyperaddition $\op$.
Recall from Theorem~\ref{thm:defect-summary} that
\begin{align*}
  (x\op y)\op z
  &= \{(\Lambda,m_L)\},
  & m_L &= c+|a-b|,\\
  x\op(y\op z)
  &= \{(\Lambda,m_R)\},
  & m_R &= a+b+c,
\end{align*}
and that the associativity defect
\[
  \AssocCurv(a,b,c)
  := m_R-m_L
  = a+b-|a-b|
  = 2\min(a,b)
  = \Cfun(a,-b)
\]
is independent of $c$.

\begin{lemma}[Ambient sum of the embedded triple]\label{lem:ambient-sum-triple}
Let $x=(+,a)$, $y=(-,b)$, $z=(\Lambda,c)$ with $a,b,c>0$, and define
\[
  X := \iota(x),\quad Y := \iota(y),\quad Z := \iota(z)\in\Kamb.
\]
Then
\[
  X = (a,0),\quad Y=(-b,0),\quad Z=(0,c),
\]
and for any bracketing of $X\oplus Y\oplus Z$ one has
\[
  X\oplus Y\oplus Z
  = (a-b,\; c+\Cfun(a,-b)).
\]
\end{lemma}

\begin{proof}
The expressions for $X,Y,Z$ follow from
Definition~\ref{def:real-shadow-embedding}.  Since $\oplus$ is
associative and commutative (Proposition~\ref{prop:Kamb-assoc}), the
value of $X\oplus Y\oplus Z$ is independent of bracketing, but we
compute it explicitly for definiteness:
\begin{align*}
  X\oplus Y
  &= (a,0)\oplus(-b,0)\\
  &= \bigl(a-b,\; 0+0+\Cfun(a,-b)\bigr)\\
  &= (a-b,\Cfun(a,-b)).
\end{align*}
Then
\begin{align*}
  (X\oplus Y)\oplus Z
  &= (a-b,\Cfun(a,-b))\oplus(0,c)\\
  &= \bigl(a-b+0,\;
           \Cfun(a,-b)+c+\Cfun(a-b,0)\bigr).
\end{align*}
Since $\Cfun(r,0)=|r|+0-|r|=0$ for all $r\in\R$, we obtain
\[
  (X\oplus Y)\oplus Z
  = (a-b,\; c+\Cfun(a,-b)).
\]
By associativity of $\oplus$, this equals $X\oplus(Y\oplus Z)$ and any
other bracketing.
\end{proof}

Thus the ambient monoid records the triple by a single state: the net
real shadow $a-b$ together with a total cancellation mass
$c+\Cfun(a,-b)$.

\subsubsection*{Recovering $\Lambda$--mass from the ambient state}

To compare with the hyperaddition in $\Hnum$, we need a way of
extracting a $\Lambda$--magnitude from an ambient state.

\begin{definition}[Canonical $\Lambda$--projection]\label{def:Lambda-projection-ambient}
Define a map
\[
  P_\Lambda:\Kamb\longrightarrow H_\Lambda
\]
by
\[
  P_\Lambda(x,c)
  :=
  \begin{cases}
    0, & x=0,\ c=0,\\[2pt]
    (\Lambda,\,|x|+c), & \text{otherwise},
  \end{cases}
\]
where $H_\Lambda:=\{(\Lambda,a):a>0\}\cup\{0\}\subset\Hnum$ is the
$\Lambda$--sector together with $0$.
\end{definition}

Intuitively, $P_\Lambda$ forgets the sign of the real shadow, takes its
absolute value, adds the stored cancellation mass, and realizes the
result as a $\Lambda$--magnitude.  The exceptional case $(0,0)\mapsto0$
is just a harmless convention.

\begin{lemma}[Basic properties of $P_\Lambda$]\label{lem:PLambda-basic}
For all $(x,c)\in\Kamb$:
\begin{enumerate}
  \item $P_\Lambda(x,c)\in H_\Lambda$.
  \item If $(x,c)\ne(0,0)$, then $P_\Lambda(x,c)$ has sign $\Lambda$
        and magnitude $|x|+c$.
  \item $P_\Lambda$ is homogeneous for nonnegative scalars, i.e.,
        for $\lambda\ge 0$,
        \[
          P_\Lambda\bigl(\lambda x,\lambda c\bigr)
          =
          \begin{cases}
            0, & \lambda=0,\ (x,c)=(0,0),\\[2pt]
            (\Lambda,\lambda(|x|+c)), & \text{otherwise}.
          \end{cases}
        \]
\end{enumerate}
\end{lemma}

\begin{proof}
(1) and (2) are immediate from Definition~\ref{def:Lambda-projection-ambient}.
For (3), if $\lambda=0$, then $(\lambda x,\lambda c)=(0,0)$ and
$P_\Lambda(0,0)=0$.  If $\lambda>0$ and $(x,c)\ne(0,0)$, then
\[
  P_\Lambda\bigl(\lambda x,\lambda c\bigr)
  = (\Lambda,|\lambda x|+\lambda c)
  = (\Lambda,\lambda(|x|+c)).
\]
\end{proof}

We can now express the two hyper-sums $(x\op y)\op z$ and
$x\op(y\op z)$ for the triple $(+,a),(-,b),(\Lambda,c)$ in terms of
$P_\Lambda$ and the ambient sum $X\oplus Y\oplus Z$.

\begin{proposition}[Ambient interpretation of the two bracketings]
\label{prop:ambient-two-bracketings}
Let $x=(+,a)$, $y=(-,b)$, $z=(\Lambda,c)$ with $a,b,c>0$, and let
$X,Y,Z\in\Kamb$ be as in Lemma~\ref{lem:ambient-sum-triple}.  Set
\[
  U := X\oplus Y\oplus Z
  = (a-b,\; c+\Cfun(a,-b)).
\]
Then:
\begin{enumerate}
  \item The right--associated hyper-sum corresponds to the full
        $\Lambda$--projection of $U$:
  \[
    x\op(y\op z) = \{P_\Lambda(U)\}
    = \{(\Lambda,\,|a-b|+c+\Cfun(a,-b))\}
    = \{(\Lambda,a+b+c)\}.
  \]
  \item The left--associated hyper-sum corresponds to the partial
        $\Lambda$--projection that only uses the original $\Lambda$--mass
        $c$:
  \[
    (x\op y)\op z
    = \{(\Lambda,\,|a-b|+c)\}
    = \{P_\Lambda(a-b,c)\}.
  \]
\end{enumerate}
In particular, the associativity defect
\[
  \AssocCurv(a,b,c)
  = (|a-b|+c+\Cfun(a,-b)) - (|a-b|+c)
  = \Cfun(a,-b)
\]
is precisely the additional cancellation mass contributed by the
ambient state $U$ beyond the original $\Lambda$--reserve $c$.
\end{proposition}

\begin{proof}
(1) By Lemma~\ref{lem:ambient-sum-triple},
\[
  U = (a-b,\; c+\Cfun(a,-b)).
\]
Since $a,b>0$ and $(a,b)\ne(0,0)$, we have $U\ne(0,0)$, so
\[
  P_\Lambda(U)
  = (\Lambda,\,|a-b|+c+\Cfun(a,-b)).
\]
But by Theorem~\ref{thm:defect-summary},
\[
  x\op(y\op z) = \{(\Lambda,m_R)\}
  \quad\text{with }m_R=a+b+c.
\]
Using Lemma~\ref{lem:C-basic}(5),
\[
  |a-b|+\Cfun(a,-b)
  = |a-b| + (a+b-|a-b|)
  = a+b,
\]
so
\[
  |a-b|+c+\Cfun(a,-b) = a+b+c = m_R.
\]
Thus $x\op(y\op z)=\{P_\Lambda(U)\}$.

(2) The left--associated sum is given by
\[
  (x\op y)\op z = \{(\Lambda,m_L)\},\qquad m_L=c+|a-b|
\]
by Theorem~\ref{thm:defect-summary}.  On the other hand,
\[
  P_\Lambda(a-b,c)
  = (\Lambda,\,|a-b|+c),
\]
which coincides with $(\Lambda,m_L)$.

The last displayed identity for $\AssocCurv(a,b,c)$ follows by subtraction:
\[
  \AssocCurv(a,b,c)
  = m_R-m_L
  = (|a-b|+c+\Cfun(a,-b)) - (|a-b|+c)
  = \Cfun(a,-b).
\]
\end{proof}

Thus the right bracketing $x\op(y\op z)$ corresponds to ``using up''
all ambient cancellation mass stored in $U$, while the left bracketing
$(x\op y)\op z$ only uses the original $\Lambda$--reserve $c$ and
leaves the extra mass $\Cfun(a,-b)$ unconverted.  The associativity
defect is exactly the difference between these two ways of reading the
same ambient state.

\subsubsection*{Conceptual summary}

We can summarize the ambient picture of associativity failure as
follows.

\begin{theorem}[Ambient interpretation of associativity defect]
\label{thm:ambient-defect-summary}
Let $a,b,c>0$ and consider
\[
  x=(+,a),\quad y=(-,b),\quad z=(\Lambda,c)\in\Hnum,
\]
with embeddings $X=\iota(x)$, $Y=\iota(y)$, $Z=\iota(z)\in\Kamb$ and
ambient sum
\[
  U := X\oplus Y\oplus Z\in\Kamb.
\]
Then:
\begin{enumerate}
  \item $U$ is independent of the bracketing of $(X,Y,Z)$ and equals
  \[
    U=(a-b,\; c+\Cfun(a,-b)).
  \]
  \item The left and right hyper-sums in $\Hnum$ are obtained from
        \emph{different} $\Lambda$--projections of ambient data:
  \[
    (x\op y)\op z
    = \{P_\Lambda(a-b,c)\},
    \qquad
    x\op(y\op z)
    = \{P_\Lambda(U)\}.
  \]
  \item The associativity defect is exactly the extra cancellation
        mass recorded in the second coordinate of $U$:
  \[
    \AssocCurv(a,b,c)
    = \Cfun(a,-b)
    = \text{(second coordinate of $U$)} - c.
  \]
\end{enumerate}
\end{theorem}

\begin{proof}
(1) is Lemma~\ref{lem:ambient-sum-triple}.  (2) and (3) are precisely
the content of Proposition~\ref{prop:ambient-two-bracketings}.
\end{proof}

\begin{remark}[Ambient associativity vs.\ hypernumber nonassociativity]
\label{rem:ambient-vs-H}
The ambient monoid $(\Kamb,\oplus)$ is strictly associative
(Proposition~\ref{prop:Kamb-assoc}); all the nonassociativity of
$(\Hnum,\op)$ arises from the choice of how to interpret ambient
states back into hypernumbers, in particular how much of the available
cancellation mass is converted into $\Lambda$--magnitude at each step.
For the ordered $(+,-,\Lambda)$ pattern, the two natural bracketings
correspond to two natural ways of partially or fully applying
$P_\Lambda$ before or after aggregating cancellation in the ambient
monoid.  The difference between these interpretations is measured
exactly by the cancellation mass $\Cfun(a,-b)$, which is the
associativity defect $\AssocCurv(a,b,c)$.

This viewpoint suggests that $(\Hnum,\op)$ can be seen as a
nonassociative ``shadow'' of the associative ambient object
$(\Kamb,\oplus)$, obtained by applying a non-linear interpretation map
back to the hypernumber world.  A systematic exploration of such
interpretations, and their possible generalizations to $n$--sign
systems, is a natural direction for future work.
\end{remark}

\subsection{Ambient representation and an obstruction theorem}
\label{subsec:ambient-nogo}

In Sections~\ref{subsec:ambient-monoid} and~\ref{subsec:ambient-link-H}
we constructed an associative ambient monoid $(\Kamb,\oplus)$ together
with an embedding $\iota:\Hnum\hookrightarrow\Kamb$ and a
$\Lambda$–projection $P_\Lambda:\Kamb\to H_\Lambda$ that explain the
associativity defect for ordered $(+,-,\Lambda)$–triples.  It is
natural to ask whether $(\Hnum,\op)$ could itself be \emph{derived} in
a clean way from some ambient commutative monoid (or hypergroup) via a
suitable interpretation map.  In this subsection we formalize this
idea and prove an obstruction theorem: there is no ambient representation of
$\Hnum$ that is simultaneously bracket–independent and compatible with
all hyper–sums.

\subsubsection*{Ambient representations with bracket–independent decoding}

We begin with a general notion of ambient representation for a
commutative hypermagma, which captures the idea that the hyper–sum
should be obtained by encoding elements into an associative monoid,
adding them there, and then decoding back in a way that is independent
of the bracketing.

\begin{definition}[Ambient representation with bracket–independent decoding]
\label{def:ambient-representation}
Let $(H,\boxplus)$ be a commutative hypermagma and let
$(K,\oplus,\mathbf{0})$ be a commutative monoid.  An
\emph{ambient representation} of $(H,\boxplus)$ in $K$ with
bracket–independent decoding consists of:
\begin{enumerate}
  \item an injective map $\iota:H\to K$ (the \emph{encoding}), and
  \item a map $D:K\to H$ (the \emph{decoding}),
\end{enumerate}
such that:
\begin{enumerate}
  \item[(AR1)] $D\circ\iota = \mathrm{id}_H$ (so $D$ is a left inverse
               to $\iota$);
  \item[(AR2)] for all $x,y\in H$,
  \[
    D\bigl(\iota(x)\oplus\iota(y)\bigr)
    \in x\boxplus y;
  \]
  \item[(AR3)] for all $x,y,z\in H$, writing
  \[
    S(x,y,z) := \iota(x)\oplus\iota(y)\oplus\iota(z)\in K,
  \]
  one has
  \[
    D\bigl(S(x,y,z)\bigr)
    \in (x\boxplus y)\boxplus z
       \;\cap\;
       x\boxplus(y\boxplus z).
  \]
\end{enumerate}
\end{definition}

Condition (AR2) says that the ambient two–fold sum refines the
hyper–sum: if we encode $x,y$, add in $K$, and decode back, the result
is one of the admissible hyper–sums $x\boxplus y$.  Condition (AR3)
is a compatibility requirement with associativity in $K$: since
$\oplus$ is associative, the triple ambient sum
$S(x,y,z)=\iota(x)\oplus\iota(y)\oplus\iota(z)$ does not depend on the
bracketing, so its decoding should be simultaneously compatible with
both bracketings $(x\boxplus y)\boxplus z$ and $x\boxplus(y\boxplus
z)$.  In particular $(x\boxplus y)\boxplus z$ and $x\boxplus(y\boxplus
z)$ must intersect nontrivially for every triple $(x,y,z)$.

The following lemma isolates this simple but important consequence.

\begin{lemma}[Bracket intersection requirement]
\label{lem:ambient-bracket-intersection}
Let $(H,\boxplus)$ be a commutative hypermagma and
$(K,\oplus,\mathbf{0},\iota,D)$ an ambient representation in the sense
of Definition~\ref{def:ambient-representation}.  Then for every triple
$(x,y,z)\in H^3$,
\[
  (x\boxplus y)\boxplus z
  \;\cap\;
  x\boxplus(y\boxplus z)
  \;\neq\;\emptyset.
\]
In particular, $(H,\boxplus)$ cannot have a triple $(x,y,z)$ at which
the two bracketings are disjoint subsets of $H$.
\end{lemma}

\begin{proof}
For any $x,y,z\in H$, define
\[
  S(x,y,z)
  := \iota(x)\oplus\iota(y)\oplus\iota(z)\in K.
\]
By associativity of $\oplus$, the element $S(x,y,z)$ is independent of
the bracketing of $\iota(x),\iota(y),\iota(z)$ in $K$.  By (AR3),
\[
  D\bigl(S(x,y,z)\bigr)
  \in (x\boxplus y)\boxplus z
  \quad\text{and}\quad
  D\bigl(S(x,y,z)\bigr)
  \in x\boxplus(y\boxplus z).
\]
Hence $D(S(x,y,z))$ lies in the intersection of the two bracketings,
which must therefore be nonempty.
\end{proof}

Conversely, if in a given hypermagma $(H,\boxplus)$ we can find a
triple $(x,y,z)$ for which $(x\boxplus y)\boxplus z$ and
$x\boxplus(y\boxplus z)$ are disjoint, then no ambient representation
with bracket–independent decoding can exist.

\subsubsection*{Application to the three--sign hypernumber system}

We now apply the general lemma to the three--sign hypernumber system
$(\Hnum,\op)$.

Recall from Proposition~\ref{prop:H-not-assoc} that
\[
  x=(+,1),\quad y=(-,1),\quad z=(\Lambda,1)
\]
give a concrete failure of associativity:
\[
  (x\op y)\op z
  = \{(\Lambda,1)\},
  \qquad
  x\op(y\op z)
  = \{(\Lambda,3)\},
\]
so
\[
  (x\op y)\op z
  \;\cap\;
  x\op(y\op z)
  = \emptyset.
\]
More generally, by Theorem~\ref{thm:defect-summary}, for arbitrary
$a,b,c>0$ we have
\begin{align*}
  (x\op y)\op z
  &= \{(\Lambda,m_L)\},
  & m_L &= c+|a-b|,\\
  x\op(y\op z)
  &= \{(\Lambda,m_R)\},
  & m_R &= a+b+c,
\end{align*}
with $m_L\ne m_R$, so the two singleton sets are always disjoint.

We can now state the obstruction theorem.

\begin{theorem}[Obstruction for bracket–independent ambient representation]
\label{thm:nogo-ambient}
There is no ambient representation of the three--sign hypernumber
system $(\Hnum,\op)$ in any commutative monoid $(K,\oplus)$ in the
sense of Definition~\ref{def:ambient-representation}.  Equivalently,
there do not exist a commutative monoid $(K,\oplus)$, an injective map
$\iota:\Hnum\to K$, and a decoding map $D:K\to\Hnum$ such that
\emph{all} of (AR1)--(AR3) hold.
\end{theorem}

\begin{proof}
Suppose, for contradiction, that such $(K,\oplus,\iota,D)$ exist.
Then by Lemma~\ref{lem:ambient-bracket-intersection}, for every triple
$(x,y,z)\in\Hnum^3$ the two bracketings $(x\op y)\op z$ and
$x\op(y\op z)$ must intersect nontrivially.

However, as noted above, for
\[
  x=(+,1),\quad y=(-,1),\quad z=(\Lambda,1)
\]
we have
\[
  (x\op y)\op z
  = \{(\Lambda,1)\},
  \qquad
  x\op(y\op z)
  = \{(\Lambda,3)\},
\]
so
\[
  (x\op y)\op z
  \;\cap\;
  x\op(y\op z)
  = \emptyset.
\]
This contradicts the conclusion of Lemma~\ref{lem:ambient-bracket-intersection}.
Therefore no ambient representation satisfying (AR1)--(AR3) can exist.
\end{proof}

In particular, the associative ambient monoid $(\Kamb,\oplus)$
constructed in Section~\ref{subsec:ambient-monoid}, together with the
embedding $\iota:\Hnum\to\Kamb$ from
Definition~\ref{def:real-shadow-embedding}, cannot be completed by any
decoding map $D:\Kamb\to\Hnum$ to yield an ambient representation in
the sense of Definition~\ref{def:ambient-representation}.  The
obstruction is exactly the existence of strongly nonassociative
triples in $(\Hnum,\op)$ whose two bracketings have disjoint images.

\subsubsection*{Consequences and interpretation}

The obstruction theorem shows that the ambient picture developed in
Section~\ref{sec:ambient} cannot be upgraded to a full ``hidden
associative model'' for the three--sign hyperaddition: one cannot find
a globally bracket–independent decoding of ambient sums that recovers
all hyper–sums in $\Hnum$.  In particular:

\begin{itemize}[leftmargin=2em]
  \item Any attempt to model $(\Hnum,\op)$ as a coarse projection of
        an associative monoid or group must either
        \emph{(i)} restrict to a proper subset of triples (e.g.\ avoid
        the $(+,-,\Lambda)$ pattern), or
        \emph{(ii)} relax the compatibility conditions (AR2)–(AR3),
        for instance by allowing different decodings for different
        bracketings of the same ambient sum.
  \item The ambient cancellation monoid $(\Kamb,\oplus)$ should thus
        be viewed not as an exact associative cover of
        $(\Hnum,\op)$, but rather as a tool for \emph{measuring} the
        defect of associativity (via the second coordinate and the
        cancellation mass $\Cfun$), as formalized in
        Theorem~\ref{thm:ambient-defect-summary}.
  \item From a structural perspective, the obstruction result clarifies that
        the nonassociativity of $(\Hnum,\op)$ is intrinsic to the
        rules governing the $\Lambda$–sector and cannot be removed by
        simply ``lifting'' to an associative world and projecting
        back in a bracket–independent way while preserving all
        hyper–sums.
\end{itemize}

This motivates treating the ambient constructions as auxiliary
invariants for $(\Hnum,\op)$, rather than as an underlying associative
model.  In particular, it suggests that any $n$--sign generalization
with similar cancellation behaviour will necessarily exhibit an
intrinsic nonassociativity that cannot be repaired by a naive ambient
representation of the kind ruled out by
Theorem~\ref{thm:nogo-ambient}.


\section{Conclusion}
\label{sec:conclusion}

\noindent\textbf{Further directions.}
The present framework suggests several extensions: (i) a systematic study of
``enriched valuations'' into signed or extended hyperfields in the sense of
recent work on tropical extensions \cite{MaxwellSmith2024Extensions}; (ii)
connections with valuation theory for Krasner hyperfields and related spectral
constructions \cite{Linzi2025ValuationKrasner}; and (iii) comparisons with
matroidal and convexity phenomena over the sign and tropical hyperfields
\cite{BakerBowler2019MatroidsHyperfields,Loho2021SignedTropicalConvexity}.
\medskip

The hypernumber system $H$ introduced here provides a concrete model of a
commutative, cancellation--sensitive hyperaddition extending the real
field by a third ``neutral'' sign $\Lambda$.  The main structural inputs
are (i) the explicit description of additive inverses and the controlled
nonassociativity measured by the associator curvature $\AssocCurv$, (ii)
the canonical sign--layer envelope that records all admissible signs of
hypersums, and (iii) the ambient cancellation monoid together with the
obstruction theorem explaining why no associative additive structure can
extend $\op$ without losing essential cancellation information.

An extended version of the manuscript will contain additional material on
order, detailed computations, and visualizations; these are omitted here
to keep the exposition focused on the core algebraic structure.

\bibliographystyle{amsalpha}

\bibliography{three-sign-hypernumbers_submission_v2}

\end{document}